\documentclass[11pt,leqno]{amsart}

\usepackage{amsthm, amsmath, amssymb, amscd}
\usepackage[curve,matrix,arrow,frame,tips]{xy}
\usepackage{verbatim}
\usepackage{color}
\usepackage{graphicx}
\newtheorem{theorem}{Theorem}[section]

\newtheorem{corollary}[theorem]{Corollary}

\newtheorem{definition}[theorem]{Definition}
\newtheorem{example}[theorem]{Example}

\newtheorem{lemma}[theorem]{Lemma}

\newtheorem{problem}[theorem]{Problem}
\newtheorem{proposition}[theorem]{Proposition}
\newtheorem{remark}[theorem]{Remark}

\newcounter{ex}[section]

\newcommand{\E}{{\mathcal E}}

\newcommand{\G}{{\rm G}}

\newcommand{\Q}{{\mathbb Q}}

\newcommand{\Z}{{\mathbb Z}}

\renewcommand{\O}{{\mathcal O}}

\DeclareMathOperator{\Ima}{Im}

\newcommand{\M}{{\mathcal M}}

\newcommand{\Kr}{{\rm K}}

\newcommand{\GL}{{\rm GL}}
\newcommand{\SK}{{\rm SK}}
\newcommand{\Er}{{\rm E}}
\renewcommand{\mod}{\ {\rm mod}\ }

\def\thfill{\null\nobreak\hfill}
\def\endproof{\thfill\vbox{\hrule
  \hbox{\vrule\hbox to 5pt{\vbox to 5pt{\vfil}\hfil}\vrule}\hrule}}

\renewcommand{\to}{\longrightarrow}

\newcommand{\ov}{\overline}

\newcommand{\bF}{{\mathbb F}}

\begin{document}

\title[$\Kr_1$ of a $p$-adic group ring]{$\Kr_{1}$ of a $p$-adic group ring II. \\ The determinantal kernel $\SK_1$. }

\author[T. Chinburg]{T. Chinburg}\thanks{Chinburg is supported by NSF Grant \# DMS11-00355}
\address{Ted Chinburg, Dept. of Math\\Univ. of Penn.\\Phila. PA. 19104, U.S.A.}
\email{ted@math.upenn.edu}

\author[G. Pappas]{G. Pappas}
\thanks{Pappas is supported by NSF Grant \# DMS11-02208.}
\address{Georgios Pappas, Dept. of Math\\Michigan State Univ.\\E. Lansing, MI 48824, U.S.A.  }
\email{pappas@math.msu.edu}

\author[M. Taylor]{M. J. Taylor}
\thanks{2010 Mathematics Subject Classification. 19B28(primary), and 19C99, 16E20 (secondary)}
\address{Martin J. Taylor,  Merton College\\ Oxford University, Oxford OX1 4JE, U.K.}
\email{martin.taylor@merton.ox.ac.uk}

\date{\today}

%\subjclass[2010]{19B28(primary), and 19C99, 16E20(secondary)}

\begin{abstract}
We describe the group $\SK_{1}( R[ G] ) $ for group
rings $R[ G] $ where $G$ is an arbitrary finite group and where
the coefficient ring $R$ is a $p$-adically complete Noetherian integral
domain of characteristic zero which admits a lift of Frobenius and which
also satisfies a number of further mild conditions. Our results extend
previous work of R.~Oliver who obtained such results for the valuation rings
of finite extensions of the $p$-adic field.
\end{abstract}

\maketitle

% ----------------------------------------------------------------------
% ----------------------------------------------------------------------
\section{Introduction}\label{sIntro}

For an arbitrary (unital) ring $S$, the group $\Kr_1( S) $ is defined as
\begin{equation*}
\Kr_{1}(S)=\GL(S)/\Er(S)
\end{equation*}
where $\GL( S) $ denotes the (infinite) general linear group of $S$ and $\Er( S) $ denotes the subgroup of elementary matrices over $S$.
In this article we continue our study of $\Kr_{1}$ of a group ring $R[ G] $ for a finite group $G$, where the coefficient ring $R$ is a $p$-adically complete ring. The present paper is the second in a series of two
papers: the first paper [CPT1] dealt with the determinantal image of $\Kr_{1}( R[ G] ) $, whereas this paper is concerned with
the kernel of the determinant map, $\SK_{1}( R[ G])$.
It is interesting to note that recently there has been considerable
resurgence of interest in $\Kr_{1}$ of group rings with higher dimensional
rings of coefficients in equivariant Iwasawa theory (see for instance [FK],
[K], [Ka1-3]  and [RW1,2]).

\par The conditions we impose on $R$ are slightly stronger than in [CPT1]: $R$ will always denote an Noetherian integral domain of finite Krull dimension with field of fractions, denoted $N$, of characteristic zero, and $N^{c}$ will denote a chosen algebraic closure of $N$. We then have a map, which we denote $\mathrm{Det}$,
\begin{equation}
\mathrm{Det}:\Kr_{1}( R[ G] ) \rightarrow \Kr_{1}(
N^{c}[ G] ) =\oplus_\chi N^{c\times }
\end{equation}
where the direct sum extends over the irreducible $N^{c}$-valued characters
of $G$. We write $\SK_{1}( R[ G] ) =\ker (\mathrm{Det})$, so that we have the exact sequence
\begin{equation}
1\rightarrow \SK_{1}(R[G]) \rightarrow \Kr_{1}(R[G]) \rightarrow \mathrm{Det}(\Kr_{1}(R[G]))\rightarrow 1.
\end{equation}
%In [CPT1] we began our study of $\Kr_{1}(R[G])$ by
%providing a series of results on the image $\mathrm{Det} ( \Kr_{1}( R[ G]))$.
\par In this paper we complete our study by presenting a number of results on $\SK_{1}(R[G])$. One of our main motivations has been the generalization of Fr\"ohlich's theory to higher dimensional schemes over $\Z$; in particular, some of the results of this paper are used, in a crucial manner, in the proof
of the adelic Riemann-Roch theorem of [CPT2]. The results for $\mathrm{Det}(\Kr_{1}(R[G])$ were obtained in [CPT1] by a generalization of the group logarithm, as developed in [T] (see also [CR2] and [F]). Here, our results for $\SK_{1}(R[G]) $ are, to a large extent, obtained by developing the ideas of R.
Oliver as presented in his sequence of papers [O1-4] and his book [O5],
which also contain his own independent description of the group logarithm.
One of the key new achievements of this paper is the definition of a new
group logarithm which extends the previous group logarithm for $p$-groups to
all finite groups (see Sect. \ref{ss6c}). This is achieved by the use of Adams operations on the group $\mathrm{Det}( \Kr_{1}( R[ G] ) )$, which are
constructed by generalizing the work of P. Cassou-Nogu\`{e}s and the third
named author in [CNT].

\par Let $p$ be a prime number. Throughout this paper, unless explicitly
indicated to the contrary, we shall assume that:
\medskip

\textbf{Standing Hypotheses.}

(i) the natural map $R\rightarrow \varprojlim_n R/p^{n}R$ is an
isomorphism, so that $R$ is $p$-adically complete;

(ii) $R$ supports a Frobenius lift, i.e a $\Z_{p}$-algebra endomorphism $%
F=F_{R}:R\rightarrow R$ with the property that for all $r\in R$
\begin{equation*}
F( r) \equiv r^{p}\,\mod\, pR;
\end{equation*}
and that $R/(1-F)R$ is torsion free.

(iii)  $pR$ is a prime ideal of $R$;

(iv) $\SK_{1}( R\otimes _{\Z_{p}}W) =\{1\}$ for $W$ the
valuation ring of any finite unramified extension of $\Q_p$.
\medskip

Examples of such rings $R$ are: the valuation ring of a non-ramified
extension of the $p$-adic field $\Q_{p};$ the $p$-adic completion of
the polynomial ring $\Z_{p}[T_1,\ldots, T_n]$ over $\Z_{p} $
\begin{equation*}
\Z_{p} \langle \langle T_1,\ldots, T_n \rangle
 \rangle =\varprojlim_m\Z_{p}[
T_1,\ldots, T_n ] /( p^{m}) ;
\end{equation*}%
the $p$-adic completion of the ring of formal Laurent series
over $\Z_{p}$
\begin{equation*}
\Z_{p}\{\{T_1,\ldots, T_n \}\}=\varprojlim_m\Z_{p}(( T_1,\ldots, T_n ))/(
p^{m}).
\end{equation*}
Here $\Z_{p}(( T_1,\ldots, T_n ))=\Z_p[[T_1,\ldots, T_n ]][T^{-1}_1,\ldots, T_n^{-1} ]$.
In each of the latter two examples we may take $F(T)=T^{p}$. By Proposition 2.8 in [CPT2] we know that $\SK_{1} ( R\otimes_{\Z_p} W ) =\{1\}$,
when $R$ is either $\Z_{p}[T_1,\ldots, T_n]$ or $\Z_{p}(( T_1,\ldots, T_n ))$.  Now an approximation argument as in the proof of Theorem \ref{thm16} shows that $\SK_1(W\{\{T_1,\ldots, T_n\}\})=\{1\}$. Similarly, the result for the case when $R=W\langle\langle T_1,\ldots, T_n \rangle\rangle $ comes from the argument in Lemma 2.15 of [CPT2] that uses results of L. Gruson.
\medskip

Unless stated to the contrary we shall extend $F$ to an $R$-algebra
endomorphism of the group ring $R[ G] $ by setting $F(
rg) =F( r) g$ for all $r\in R$, $g\in G$.

Since $R$ is $p$-adically complete, $R[ G] $ is of course also $p$-adically
complete and the Jacobson radical of $R$ necessarily contains $pR$.
If $\{u_{n}\}$ is a $p$-adically convergent sequence of units converging to
$r$ in $R[ G] $, then $r$ will be congruent to  $u_{n}$ modulo $p$ for large $n$, and so $r$ is a unit, and hence $R[ G]
^{\times }$ is also seen to be $p$-adically complete. (See Remark
1.1 in [CPT1].)

The group $R[ G] ^{\times }=\GL_{1}( R[ G] )
$ embeds into $\GL( R[ G] ) $ as diagonal matrices with
all non-leading diagonal terms equal to 1. We recall Theorem 1.2 of [CPT1]:

\begin{theorem}\label{thm1}
Let $R$ be as in the Standing Hypotheses. The inclusion $R[ G]
^{\times }\subset \GL( R[ G] ) $ induces an equality%
\begin{equation*}
\mathrm{Det}( R[ G]^{\times }) =\mathrm{Det}(
\GL( R[ G]) ) =\mathrm{Det}( \Kr_{1}( R[ G] ) )
\end{equation*}%
in the following two circumstances:

(a) when $G$ is a $p$-group;

(b) when $G$ is an arbitrary finite group, if $R$ is in addition
 normal.
\end{theorem}

For the purposes of this article we shall be particularly interested in the
completed K-groups
\begin{equation*}
\widehat{\Kr}_{i}( R[ G]) =\varprojlim_{n}\Kr_{i}( R[ G] /(p^{n}))
\end{equation*}%
for $i=1$, $2$. From Proposition 1.5.1 in [FK] we know that, if all the
quotient rings $R_{n}=R/p^{n}R$ are finite, then the natural map $\Kr_{1}( R[ G]) \rightarrow \widehat{\Kr}_{1}( R[ G]) $ is an isomorphism. (See also the work of C. T. C.
Wall in [W1] and [W2].) The following two results are a generalization of their
result. For these we only need to assume that $R$ is $p$-adically complete,
i.e. that $R\simeq\varprojlim_n R/p^nR$:

\begin{theorem}\label{thm2}Assume that $R$ is a $p$-adically complete Noetherian integral domain
 with fraction field of characteristic zero.
 Then the natural map $\Kr_{1}( R[ G]) \rightarrow \widehat{\Kr}_{1}( R[ G]) $ is an isomorphism.
\end{theorem}

\begin{theorem}\label{thm3}Assume that $R$ is a $p$-adically complete Noetherian integral domain
 with fraction field of characteristic zero.
 For $p^{n}>2$ the reduction map $\Kr_{2}( R[ G] )
\rightarrow \Kr_{2}( R_{n}[ G]) $ is surjective.
\end{theorem}

The above leads us to formulate:

\begin{problem} Assume $R$ is as above.
Let $m$ be an arbitrary non-negative integer. Are the reduction maps $\Kr_{m}( R[ G])
\rightarrow \Kr_{m}( R_{n}[ G]) $   surjective? Is the natural map $\Kr_{m}( R[G])
\rightarrow \widehat{\Kr}_{m}( R[ G])$
an isomorphism?
\end{problem}

\noindent\textbf{Remark.} By the long exact sequence of K-theory (see for instance
(\ref{eq2.5}) in Sect. \ref{s2}), a positive answer to the first  question above is equivalent to the injectivity of
natural map $\Kr_{m-1}( R[ G] , p^{n}) \rightarrow
\Kr_{m-1}( R[ G])$.

Our first result for $\SK_{1}$ concerns the image of $\SK_{1}( R[ G]) $ in $\Kr_{1}( R_{n}[ G])$, which we
denote by $\SK_{1}( R[ G]) _{n}$; for the purposes of
this result we again only need $R$ to be $p$-adically complete:

\begin{theorem}\label{thm5} Assume that $R$ is a $p$-adically complete Noetherian integral domain
 with fraction field of characteristic zero.
If $p^{n}>2$, then $\SK_{1}( R[ G]) $ maps isomorphically
onto $\SK_{1}( R[ G]) _{n}$ under the natural map $\Kr_{1}( R[ G]) \rightarrow \Kr_{1}( R_{n}[ G])$.
\end{theorem}

The key to our study of $\SK_{1}( R[ G] ) $ is the
group logarithm (see Section 3 in [CPT1] for details). Suppose for the
moment that $G $ is a $p$-group. We let $I_G=I( R[ G] ) $
denote the augmentation ideal $ \ker ( R[ G] \rightarrow
R) $; let $\mathcal{A}( R[ G]) =\ker ( R[ G] \rightarrow R[ G^{\rm ab}])$, let $C_{G}$
denote the set of conjugacy classes of $G$ and let $\phi :R[ G]
\rightarrow R[ C_{G}] $ be the $R$-linear map obtained by mapping
each group element to its conjugacy class. We write $\Kr_{1}^{\prime }( R[ G] , I_{G}) $ for the
image of $\Kr_{1}( R[ G] ,I_{G}) $ in $\Kr_{1}( R[ G])$. We
denote the Whitehead group $\Kr_{1}^{\prime }( R[ G]
,I_{G}) /\mbox{Im}( G) $ by $\mathrm{Wh}_{G}( R)$, or $\mathrm{Wh}_{G}$ when $R$ is clear from the context. Using Theorems
3.15 and 3.17 in [CPT1] we obtain the exact sequence:
\begin{equation*}
1\rightarrow \frac{\mathrm{Det}( 1+I_{G}) }{\mathrm{Det}(
G) }\rightarrow \phi ( I_{G}) \overset{\omega }{\rightarrow
}G^{\mathrm{ab}}\otimes _{\Z}\frac{R}{( 1-F) R}%
\rightarrow 1
\end{equation*}%
where $\omega $ is the logarithmic derivative map of Proposition 3.18 loc.
cit.; using Theorem \ref{thm1} (a) we obtain the further exact sequence%
\begin{equation*}
1\rightarrow \SK_{1}( R[ G] ) \rightarrow \mathrm{Wh}%
_{G}( R) \rightarrow \frac{\mathrm{Det}( 1+I_{G}) }{%
\mathrm{Det}( G) }\rightarrow 1.
\end{equation*}
These two exact sequences may then be spliced together to give the four
term exact sequence%
\begin{equation}\label{eq1.3}
1\rightarrow \SK_{1}( R[ G] ) \rightarrow {\rm Wh}_{G}(
R) \rightarrow \phi ( I_{G}) \rightarrow G^{\mathrm{ab}%
}\otimes _{\Z}\frac{R}{( 1-F) R}\rightarrow 1.
\end{equation}%
Let $H_{2}^{\mathrm{ab}}( G, \Z) $ denote the subgroup of
the Schur multiplier $H_{2}( G,\Z) $ generated by the
images under corestriction of the $H_{2}( A,\Z) $ for all
abelian subgroups $A$ of $G,$\ and set $\overline{H}_{2}( G, \Z) =H_{2}( G,\Z) /H_{2}^{\mathrm{ab}}( G,
\Z)$. We may then use Oliver's construction (see Section 3
for details) to construct from this exact sequence a map
\begin{equation}\label{eq1.4}
\Theta _{R[G]}:\SK_1( R[ G]) \rightarrow \frac{R}{( 1-F) R}\otimes \overline{H}_{2}( G,\Z) .
\end{equation}%
We extend Oliver's proof to show:

\begin{theorem}\label{thm6}
For a $p$-group $G$ and for a ring $R$, which satisfies the Standing
Hypotheses, the map $\Theta _{R[G]}$ is an isomorphism.
\end{theorem}

In Sections \ref{s5} and \ref{s6} we extend our results from $p$-groups to arbitrary
finite groups $G.$ The formulation of our general result uses ideas from
[IV]. We let $G_{r}$ denote the set of $p$-regular elements in $G$; we let $R[ G_{r}] $ denote the free $R$-module on the elements of $G_{r}$
and we let $\Psi $ be the $ R$-module endomorphism of $R[ G_{r}]$ which is given by the rule
\begin{equation*}
\Psi ( \sum\nolimits_{g\in G_{r}}a_{g}g) =\sum\nolimits_{g\in
G_{r}}F( a_{g}) g^{p}.
\end{equation*}%
Let $G$ act on $R[ G_{r}] $ by conjugation on $G_{r}$; then the
homology group $H_{2}( G,R[ G_{r}] ) $ is defined. We
let $H_{2}^{\mathrm{ab}}( G,R[ G_{r}]) $ denote the
subgroup of   $H_{2}( G, R[ G_{r}] )$ generated by the images under corestriction of
the $H_{2}( A, R[
A_{r}]) $ for all abelian subgroups $A$ of $G$, and we set $%
\overline{H}_{2}( G, R[ G_{r}] ) =H_{2}( G, R[G_{r}]) /H_{2}^{\mathrm{ab}}( G, R[ G_{r}])$. Then $\Psi $ acts, via its action on $R[ G_{r}]$, on the   groups $H_{2}( G, R[ G_{r}])$, $H_{2}^{\mathrm{ab}}( G,R[ G_{r}] )$,
$\overline{H}_{2}( G, R[ G_{r}] ) $ and we will
write $H_{2}( G, R[ G_{r}]) _{\Psi }$,
$H_{2}^{\rm ab}( G,R[ G_{r}] ) _{\Psi }$,  $\overline{H}%
_{2}( G, R[ G_{r}] ) _{\Psi }$ for their groups of
co-invariants.
\begin{comment}
For any finite non-ramified extension $L$ of $\Q_{p}$, we set $\Delta =%
\mathrm{Gal}( L/\Q_{p})$, we put $R_{L}=R\otimes_{\Z_{p}}\O_{L}$ and $R_{L}$ supports a lift of Frobenius by tensoring
the lift of Frobenius on $R$ with the Frobenius on $\O_{L}$. In order to
state the general result we let $\overline{R}$ denote the integral domain $R/pR$,
we let $\mathbb{F}_{p}^{c}$ denote a chosen algebraic closure
of $\mathbb{F}_{p}$. We define
\begin{equation*}
\M( R, F) =\{u\in R^{\times }\mid F( u) =u^{p}\}
\end{equation*}
and we set
\begin{equation*}
\overline\Lambda ( R, F) =\overline{R}^{\times }/\mbox{Im}(\M(
R, F)) .
\end{equation*}
\end{comment}
Then we have:

\begin{theorem}\label{thm7}
Let  $G$ be an arbitrary finite group and let $R$ be a ring which
satisfies the stated hypotheses, and which in addition is normal. Then there is a
natural isomorphism:
\begin{equation*}
\Theta _{R[G]}:\SK_{1}( R[ G] ) \xrightarrow{\sim} \overline{H}%
_{2}( G, R[ G_{r}] )_{\Psi }.
\end{equation*}
\end{theorem}

To analyze the right-hand side further we let $\{C_{i}\}_{i\in I}$
denote the set of $G$-conjugacy classes in $G_{r}$, let $g_{i}$ be a chosen
group element in $C_{i}$ and set $G_{i}$ denote the centralizer of $g_{i}$
in $G$ so that we have a disjoint union decomposition $G_{r}=\cup
_{i}g_{i}^{G/G_{i}}$. We next consider the action of $\Psi $ on the $\{C_{i}\}_{i\in I}$.
We may view this action as an action on $I$, and we let
$J$ denote   the set of orbits of $\Psi $ on $I$. For $j\in J $ we
let $n_{j}$ denote the cardinality of $j$; we then obtain a further disjoint
union decomposition
\begin{equation*}
G_{r}=\bigcup _{j\in J}\bigcup _{m=1}^{n_{j}}( g_{i_{j}}^{G/G_{i_{j}}})
^{\Psi ^{m}}
\end{equation*}%
where $i_{j}$ denotes a chosen element of the orbit $j$ and the conjugacy
class of $g_{i_{j}}$ lies in $C_{i_{j}}.$

In   Appendix B,  we show that this decomposition affords the explicit
description:

\begin{corollary} \label{cor8} Under the assumptions of the  above theorem we have
\begin{equation*}
\SK_{1}(R[G]) \cong  \overline{H}_{2}(G,R[G_{r}])_{\Psi} = \bigoplus_{j}\left[\overline{H}_{2}(G_{i_{j}}, \Z) \otimes \frac{R}{(F-1)R}\right] .
\end{equation*}
\end{corollary}

\begin{corollary}\label{cor9}
Given two rings $R\subset S$ which satisfy the conditions of Theorem \ref{thm7} and
which have compatible lifts of Frobenius (so that the restriction of $F_{S}$
to $R$ coincides with $F_{R}$), then the natural map $\SK_{1}( R[ G]) $ to $\SK_{1}( S[ G]) $ is surjective,
resp. an isomorphism, if the natural map
\begin{equation*}
\frac{R}{(1-F_{R})R}\rightarrow \frac{S}{(1-F_{S})S}
\end{equation*}%
is surjective, resp. an isomorphism. If this map is injective with torsion
free cokernel, then the map $\SK_{1}( R[ G]) $ to $\SK_{1}( S[ G] ) $ is also injective.
\end{corollary}

Let $W$ denote the ring of integers of a finite non-ramified extension of $\Q_{p}$ and let $F$ denote the Frobenius automorphism of $W$.  An explicit formula for $\SK_{1}(W[G])$ was given by R. ~Oliver in his series of papers [O1-4]; see also
[O5]; thus Theorem \ref{thm7} is a generalization of Oliver's result.

We now extend $F$ to the power series ring $W[[t]]$ in
an indeterminate $t$ by setting $F( t) =t^{p}$.
Since $(1-F)  ( tW [[ t]]) =tW[[t]]$ we see using the above corollary that for any finite group $G$
the inclusion $W\subset W[[t]]$ induces an isomorphism
\begin{equation}\label{skw}
\SK_{1}( W[ G] ) \cong \SK_{1}( W[ [ t] ] [ G] ) .
\end{equation}
(see Proposition 5.3 in [Wi]).
We also consider
 \begin{equation*}
W \langle  \langle t^{-1} \rangle  \rangle
=\varprojlim_{n}(W[ t^{-1}]/p^nW[ t^{-1}] ),\ \
W\{\{t\}\}=\varprojlim_{n}(W((t))/p^nW((t)))
\end{equation*}
with similarly extended Frobenius.
In Appendix B, we also show that Corollary \ref{cor9} implies:

\begin{corollary}\label{le11}
With the above notation   the inclusion
$W \langle \langle t^{-1} \rangle  \rangle \subset
W\{\{t\}\}$ induces an isomorphism $\SK_{1} ( W \langle  \langle
t^{-1} \rangle  \rangle [G]) \cong \SK_{1}(
W\{\{t\}\}[ G]) $.
\end{corollary}

\begin{corollary}\label{le12}
With the above notation  the inclusion $W[ [ t]] \rightarrow W\{\{t\}\}$
induces an injection $\SK_{1}( W[[t]][ G]) \hookrightarrow \SK_{1}( W\{\{t\}\}[G])$.
\end{corollary}

These results play an important role in the proof of the adelic Riemann-Roch theorem of [CPT2].
\smallskip

{\sl Acknowledgement:} The authors would like to thank Otmar Venjakob for useful discussions.

% ----------------------------------------------------------------------
% ----------------------------------------------------------------------
\section{General results for $\Kr_{1}$}\label{s2}
\setcounter{equation}{0}

Let $S$ denote an arbitrary unitary ring. Throughout this section $I$ will
denote a $2$-sided ideal of $S,$ which is contained in the Jacobson radical of
$S$; let $\overline{S}=S/I$. Recall that if $v\in S^{\times }$, $j\in I$,
 $s$, $t\in S$ then $v+sjt\in S^{\times }$. Let ${\rm M}_{n}( S) $
 denote the ring of $n\times n$ matrices with entries in $S$ and let ${\rm M}_{n}( I)$
 denote the $2$-sided ideal of ${\rm M}_{n}( S)$
of matrices with entries in $I$. We set ${\rm M}( I) =\varinjlim_{n}{\rm M}_{n}( I) $, and we write
$1+{\rm M}( I) =\varinjlim_{n}(1+{\rm M}_{n}( I) )$. Note that for a matrix $x\in 1+{\rm M}_{n}(
I) $ since all its diagonal entries are units, by left and right
multiplication by elementary matrices, we can bring $x$ into diagonal form with
unit entries; thus $x$ is invertible and so
\begin{equation*}
1+{\rm M}( I) =\GL( S, I) \overset{\mathrm{defn}}{=}\ker
( \GL( S) \rightarrow \GL( \overline{S})) .
\end{equation*}%
Note that $\GL( S)  $ maps onto $\GL( \overline{S})$,
since given $\overline{x}\in \GL_{n}( \overline{S}) $ with inverse
$\overline{y}$, then for lifts $x,y\in {\rm M}_{n}( S) $ we have $xy\in 1+{\rm M}_{n}( I) $ and so $x $
is invertible.

\begin{lemma}\label{le13}
Suppose we are given two rings $A$, $B$ and a surjection $q:A\rightarrow B$, with the
property that $I={\rm Ker}( q:A\rightarrow B)$ is contained in the
Jacobson radical of $A$. Then the map $q_{\ast }:\Kr_{1}( A)
\rightarrow \Kr_{1}( B) $ is surjective and $\ker q_{\ast }\ $is
generated by the image of $\GL( A,I)$.
\end{lemma}

\begin{proof} The first statement follows from the above fact that the induced map $\GL( A) \rightarrow \GL( B) $ is surjective; the second
part is standard - see for instance Theorem 2.5.3 page 93 in [R].
\end{proof}

We recall the long exact sequence of K-theory (see for instance page 54 of
[M])
\begin{equation}\label{eq2.5}
\Kr_{2}( S) \rightarrow \Kr_{2}( \overline{S}) \rightarrow
\Kr_{1}( S, I) \rightarrow \Kr_{1}( S) \overset{q}{%
\rightarrow }\Kr_{1}( \overline{S}) \rightarrow 1
\end{equation}%
with $q$ surjective since $\GL( S)  $ maps onto $\GL(
\overline{S})$. We let ${\rm E}( S, I) $ denote the smallest
normal subgroup of the group of elementary matrices ${\rm E}( S) $
containing the elementary matrices $e_{ij}( a) $ with $a\in
I$. Then we know (see for instance page 93 in [R])%
\begin{equation*}
\Kr_{1}( S, I) =\frac{\GL( S, I) }{{\rm E}( S,I) }.
\end{equation*}%
We shall write $\Kr_{1}^{\prime }( S, I) $ for the image of $\Kr_{1}( S, I) $ in $\Kr_{1}( S)$. Note that $\Kr_{2}(
S) \rightarrow \Kr_{2}( \overline{S}) $ is surjective, if and
only if, the map $\Kr_{1}( S, I) \rightarrow \Kr_{1}( S) $
is injective; and this is equivalent to the equality%
\begin{equation}\label{eq2.6}
\GL( S, I) \cap {\rm E}( S) ={\rm E}( S, I).
\end{equation}%
Recall also from [R] loc. cit. that $\Er( S, I) $ is a normal subgroup of $\GL( S) $ and
\begin{equation}
[ \GL( S) , \Er( S, I) ] =[\Er(
S) , \Er( S, I) ] =\Er( S, I) .
\end{equation}
We shall write $\kappa _{S}$ for the map from $S^{\times }$ to $\Kr_{1}(S) $ given by mapping $u\in S^{\times}$ to the diagonal matrix which
is $u$ in the first position and 1 in all other diagonal entries and then
composing with the map $\GL( S) \rightarrow \Kr_{1}( S) $.

From Lemma 2.2 in [CPT1] we quote:

\begin{lemma}\label{le14}
(a)\ \ $[ \GL( S) , 1+{\rm M}( I) ] ={\rm E}(
S, I)$.

(b)\ Given $g\in \GL( S, I)$, there exist $e_{1}$, $e_{2}\in \Er(
S, I)$, and $x\in 1+I$ such that $g=e_{1}\delta ( x) e_{2}$
where $\delta ( x) $ is the diagonal matrix with leading term $x$
and with all non-leading diagonal terms equal to 1; hence in particular
\begin{equation*}
\GL( S, I) \subset \left\langle \Er( S, I) , (
1+I) \right\rangle
\end{equation*}%
where we view $1+I\subset S^{\times }\subset \GL( S, I) $ as
previously. Thus, if we write $\kappa _{I}$ for the natural map from $\GL( S, I)$ to $\Kr_{1}( S, I)$, then $\kappa _{I}(
1+I) =\Kr_{1}( S,I)$.
\end{lemma}

From 45.12 on page 142 of [CR2] we know:

\begin{proposition}\label{pro15}
If $S$ is a commutative semi-local ring, then $\SK_{1}( S)
=\{1\}$.
\end{proposition}
% ----------------------------------------------------------------------
\subsection{Proofs of Theorems \ref{thm2}, \ref{thm3} and \ref{thm5}.}

In this subsection we assume only that $R$ is $p$-adically complete. We let
${\rm SL}(R[G])$ denote the subgroup of elements in $\GL(R[G])$ which are trivial under $\rm Det$; more
generally for a two-sided ideal $I$ of $R[G] $ we define%
\begin{equation*}
{\rm SL}(R[ G] , I) ={\rm SL}( R[ G] ) \cap
\GL( R[ G] , I) .
\end{equation*}
We fix an ordering of the elements of the group; $G=\{g_{1},...,g_{q}\}$.
 We start this subsection by showing:

\begin{theorem}\label{thm16}
$\Er( R[ G] ) $ is $p$-adically closed in $\GL( R[ G])$ and for $p^{k}>2$
\begin{equation}
{\rm SL}( R[ G] , p^{k}) ={\rm E}( R[ G], p^{k}) .
\end{equation}
\end{theorem}

\begin{proof} If $m\geq 1$, we let $\overline{\Er_{m}( R[ G] ) }
$ denote the $p$-adic closure of $\Er_{m}( R[ G]) $ in $\GL_{m}( R[ G]) $.
We begin by showing that in order
to prove $\Er_{m}( R[ G] ) =\overline{\Er_{m}( R[
G] ) }$ it will suffice to show that for some integer $k$
with $p^{k}>2$ we have ${\rm SL}_{m}( R[ G] , p^{k})
=\Er_{m}( R[ G] , p^{k}) $. To this end we first observe that
we will have  $\Er_{m}( R[ G] ) =\overline{\Er_{m}( R[
G] ) }$ if we can show
\begin{equation*}
\Er_{m}( R[ G] , p^{k})=\GL_{m}( R[ G]
, p^{k}) \cap \overline{\Er_{m}( R[ G]) };
\end{equation*}
this is because if we are given $x\in \overline{\Er_{m}( R[ G]) }$, then, as $\Er_{m}( R[ G]) $ is dense in
$\overline{\Er_{m}( R[ G]) }$, we can write $x=eg$ with $%
e\in \Er_{m}( R[ G] ) $ and $g\in \GL_{m}( R[ G] ,p^{k}) \cap \overline{\Er_{m}( R[ G]) }$.
Since
\begin{equation*}
{\rm SL}_{m}( R[ G] ,p^{k}) \supset \GL_{m}( R[ G] , p^{k}) \cap \overline{\Er_{m}( R[ G]) }
\supset \Er_{m}( R[G] , p^{k}) ,
\end{equation*}%
by the above we now see that in order to show $\Er_{m}( R[ G]) =\overline{\Er_{m}( R[ G]) }$ it will suffice
to show that ${\rm SL}_{m}( R[ G] ,p^{k}) =\Er_{m}( R[G],p^{k})$. %By Lemma  \ref{le17} it will suffice to show this result for $k>1$.

It is enough to show $%
{\rm SL}_{m}( R[ G] , p^{k}) \subset \Er_{m}( R[ G] , p^{k})$; i.e. given $x\in {\rm SL}_{m}( R[ G]
, p^{k}) $ we want to show we can write $x$ as a finite product of
elements in $\Er_{m}( R[ G] , p^{k}) $. We first note
that as in [CPT1] Lemma 2.2.b given such an $x$ we can find $e$, $e^{\prime
}\in \Er_{m}( R[ G] , p^{k}) $ so that we can write $x=ede^{\prime }$ with $d\in ( 1+p^{k}R[ G]) \cap \ker
( \mathrm{Det})$. The proofs of Proposition 2.4 and Lemma 2.5
in [CPT1] apply under the hypothesis $p^{k}>2$. By the last line in the
proof of Lemma 2.5 in [CPT1] we have:

\begin{lemma}\label{le17}
Suppose that $p^{n}>2$. Given $x\in ( 1+p^{n}R[ G]) \cap \ker ( \mathrm{Det}) $ we can find $\lambda
_{i,n}\in R[ G] $ so that
\begin{equation*}
x\equiv \prod\nolimits_{i=1}^{q}[g_{i},1+p^{n}\lambda _{i,n}]\mod
p^{n+1}R[ G] .
\end{equation*}%
(Here such products, which in the sequel are not generally commutative, are
to be taken in the specified order.)
%\begin{equation*}
%\prod\nolimits_{1}^{q}[g_{i},1+p^{n}\lambda _{i,n}]=[g_{1},1+p^{n}\lambda
%_{1,n}]...[g_{q},1+p^{n}\lambda _{q,n}].
%\end{equation*}
\end{lemma}

\noindent\textbf{Remark.} The hypotheses on $n$ and $p$ imply that for any $\lambda
\in R[ G] $ we have the congruence%
\begin{equation*}
\log ( 1+p^{n}\lambda ) \equiv p^{n}\lambda \mod
p^{n+1}
\end{equation*}
which is used throughout the proof of Lemma 2.5 loc. cit.. Note that we also
need the commutator congruence for $h\in G$, $\lambda$, $\lambda ^{\prime
}\in R[ G] $%
\begin{equation*}
[ h,1+p^{n}\lambda] [ h,1+p^{n}\lambda ^{\prime }]
\equiv [ h,1+p^{n}( \lambda +\lambda ^{\prime }) ]
\mod  p^{n+1} .
\end{equation*}

\begin{lemma}\label{le18}
For $x,y,z\in R[ G]^{\times}$ we have the standard commutator
relation%
\begin{equation*}
[ z,yx] =[ z,y] [ z,x] ^{y^{-1}}.
\end{equation*}
\end{lemma}

\begin{proof} We note that
\begin{equation*}
[ z,y] [ z,x]
^{y^{-1}}=zyz^{-1}y^{-1}yzxz^{-1}x^{-1}y^{-1}=zyxz^{-1}x^{-1}y^{-1}=[
z,yx] .
\end{equation*}
\end{proof}

\begin{lemma}\label{le19}
For $\lambda$, $\mu \in R[ G] $ and $p^{n}>2$, $m\geq 1$ we have
the congruence
\begin{equation*}
( 1+p^{n}\mu )^{-1}( 1+p^{m}\lambda) (
1+p^{n}\mu ) \equiv ( 1+p^{m}\lambda ) \mod p^{n+m}R[ G]
\end{equation*}
and so
\begin{equation*}
( 1+p^{m}\lambda ) ( 1+p^{n}\mu ) (
1+p^{m}\lambda )^{-1}\equiv ( 1+p^{n}\mu) \mod
p^{n+m}R[ G].
\end{equation*}
\end{lemma}

\begin{proof} This follows at once from the congruence
\begin{equation*}
( 1+p^{m}\lambda ) ( 1+p^{n}\mu ) \equiv 1+p^{n}\mu
+p^{m}\lambda \equiv ( 1+p^{n}\mu ) ( 1+p^{m}\lambda )
\mod p^{n+m}R[ G]
\end{equation*}
and the fact that $1+p^{n}\mu $ and $1+p^{m}\lambda $ are units.
\end{proof}
\medskip

\noindent\textit{Proof of Theorem \ref{thm16}.} We have supposed that $p^{k}>2$.
 Suppose that $x$ belongs to
$( 1+p^{k}R[ G] ) \cap \ker ( \mathrm{Det})$. We shall show that  we can write $x=\prod\nolimits_{i=1}^{q}[ g_{i},1+p^{k}\mu
_{i}] $ with $\mu _{i}\in R[ G]$. By Lemma 2.2.a in [CPT1]
we know that each of these commutators is an element of $\Er( R[ G] , p^{k})$, and so we shall be done.

For $p^{n}\geq p^{k}$ we inductively suppose that we have written
\begin{equation*}
x\equiv \prod\nolimits_{i=1}^{q}[ g_{i},1+p^{k}\mu _{i,n}]
\mod p^{n+1}R[ G] \ \text{with\ }\mu _{i,n}\in R[ G] .
\end{equation*}
(Note that we can use Lemma \ref{le17} to start the induction.) By Lemma \ref{le17} above we
can write $x\cdot [ g_{q},1+p^{k}\mu _{q,n}] ^{-1}\cdots [
g_{1},1+p^{k}\mu _{1,n}] ^{-1}$ as
\begin{equation*}
\prod\nolimits_{i=1}^{q}[ g_{i},1+p^{n+1}\lambda _{i,n}]  \mod p^{n+2}R[ G]
\end{equation*}%
and so now we have written
\begin{equation*}
x\equiv \prod\nolimits_{i=1}^{q}[ g_{i},1+p^{k}\mu _{i,n}]
\prod\nolimits_{j=1}^{q}[ g_{j},1+p^{n+1}\lambda _{j,n}]  \mod p^{n+2} R[ G] .
\end{equation*}
  We then apply Lemma \ref{le19} to note that for all $i,j,$ with $ 1\leq
i,j\leq q$ we have
\begin{equation*}
[ g_{i},1+p^{k}\mu _{i,n}] [ g_{j},1+p^{n+1}\lambda _{j,n}] \equiv [ g_{j},1+p^{n+1}\lambda _{j,n}] [
g_{i},1+p^{k}\mu _{i,n}] \mod p^{n+2}R[ G]
\end{equation*}%
and by Lemma \ref{le18} we see that for all $i$ with $1\leq i\leq q $ we have
\begin{eqnarray*}
[ g_{i}, ( 1+p^{k}\mu _{i,n}) ( 1+p^{n+1}\lambda
_{j,n}) ] &=&[ g_{i}, ( 1+p^{k}\mu _{i,n})
][g_{i}, ( 1+p^{n+1}\lambda _{j,n}) ] ^{( 1+p^{k}\mu
_{i,n}) ^{-1}} \\
&\equiv &[ g_{i},1+p^{k}\mu _{i,n}][ g_{i},1+p^{n+1}\lambda
_{i,n}] \mod p^{n+2}R[ G] .
\end{eqnarray*}%
We then define
\begin{equation*}
1+p^{k}\mu _{i,n+1}=( 1+p^{k}\mu _{i,n})( 1+p^{n+1}\lambda
_{j,n})
\end{equation*}%
so that
\begin{equation*}
x\equiv \prod\nolimits_{i=1}^{q}[ g_{i},1+p^{k}\mu _{i,n+1}] \mod p^{n+2}R[ G]
\end{equation*}%
which completes the inductive step. To conclude we note that for each $i$
the sequence $\{\mu _{i,n}\}_{n}$ is a convergent sequence.
\end{proof}
\medskip

\noindent\textit{Proof of Theorem \ref{thm3}.} By the discussion prior to equation (\ref{eq2.6}) we know
that the reduction map $\Kr_{2}( R[ G] ) \rightarrow
\Kr_{2}( R_{n}[ G] ) $ is surjective if
\begin{equation*}
{\rm SL}( R[ G] , p^{n}) \cap \Er(R[ G] )=\Er( R[ G] , p^{n})
\end{equation*}%
and this follows at once from Theorem \ref{thm16} above. \endproof
\medskip

\noindent\textit{Proof of Theorem \ref{thm5}. } We must show that for $p^{n}>2$ the map $\theta
_{n}: \SK_{1}( R[ G] ) \rightarrow \SK_{1}( R[ G] )_{n}$ is injective, since by definition $\theta _{n}$ is
surjective.  If $x\in \ker \theta _{n}$ then by Lemma \ref{le14} (b) we may write $%
x=\kappa ( x^{\prime }) $ with $x^{\prime }\in (
1+p^{n}( R[ G]) ) \cap \ker ( \mathrm{Det}) \subset {\rm SL}_{m}( R[ G] , p^{n}) =\Er_{m}( R[ G] , p^{n})$. \endproof
\medskip

\noindent\textit{Proof of Theorem \ref{thm2}.} By the Bass Stable Range Theorem (see for
instance 41.25 in [CR2]) we can find a sufficiently large integer $m$ with
the property that the map $\GL_{m}( R[ G]) $ to $\Kr_{1}( R[ G]) $ is surjective. Choose a positive
integer $n$ with $p^{n}>2$. Recall that by the remark prior to (\ref{eq2.6}) and by
Theorem \ref{thm16}, we know that the relative K-group $\Kr_{1}( R[ G]
, p^{n})$ maps injectively into $\Kr_{1}( R[ G] )$. We have an exact sequence
\begin{equation*}
1\rightarrow \Er_{m}( R[ G] , p^{n}) \rightarrow
\GL_{m}( R[ G] , p^{n}) \rightarrow \Kr_{1}( R[ G] , p^{n}) \rightarrow 1
\end{equation*}%
which yields (see the exact sequence (\ref{eq2.5}))
\begin{equation*}
1\rightarrow \frac{\Er_{m}( R[ G] ) }{\Er_{m}( R[
G] ,p^{n}) }\rightarrow \frac{\GL_{m}( R[ G]) }{\GL_{m}( R[ G] , p^{n}) }\rightarrow
\Kr_{1}( R_{n}[ G] ) \rightarrow 1
\end{equation*}
and by taking inverse limits, noting that for $n\geq j$ the morphisms
\begin{equation*}
\frac{\Er_{m}( R[ G]) }{\Er_{m}( R[ G]
,p^{n}) }\rightarrow \frac{\Er_{m}( R[ G]) }{%
\Er_{m}( R[ G] , p^{j}) }
\end{equation*}%
are surjective, so that the Mitag Leffler condition is satisfied, we get the
exact sequence
\begin{equation*}
1\rightarrow \varprojlim_n\frac{\Er_{m}( R[ G]) }{%
\Er_{m} ( R [ G ] , p^{n} ) }\rightarrow \GL_{m}( R[
G] ) \rightarrow \widehat{\Kr}_{1}( R[ G])
\rightarrow 1.
\end{equation*}%
But by Theorem \ref{thm16} we know that $\varprojlim_n {\Er_{m}( R[ G]) }/{\Er_{m}( R[ G] , p^{n}) }$ is just the
$p$-adic closure $\overline{\Er_{m}( R[ G]) }$ and that
this, in turn, coincides with $\Er_{m}( R[ G])$. Hence
we have shown
\begin{equation*}
\Kr_{1}( R[ G]) =\frac{\GL_{m}( R[ G]) }{\Er_{m}( R[G]) }=\widehat{\Kr}_{1}( R[G]).
\end{equation*}
\endproof
\medskip

% ----------------------------------------------------------------------
% ----------------------------------------------------------------------
\section{$p$-groups}\label{sec3}
\setcounter{equation}{0}

Throughout this section $G$ denotes a finite $p$-group and the ring $R$
satisfies the Standing Hypotheses.

% ----------------------------------------------------------------------
\subsection{The group logarithm}\label{sslog}

In this subsection we very briefly recall the group logarithm which is a
fundamental tool for studying the determinants of units of $p$-adic group
rings.

Recall that $R$ is a Noetherian integral domain of finite Krull dimension
with field of fractions $N$ of characteristic zero, that $F$ denotes a lift
of Frobenius on $R$, and that $I( R[ G] ) $ denotes
the augmentation ideal of the group ring $R[ G] $ and when $R$ is
clear from the context we shall write $I_{G}$ in place of $I( R[ G] )$. Because $I(\mathbb{F}_{p}[ G] )$ is the
Jacobson radical of the Artinian ring $\mathbb{F}_{p}[ G]$, it
follows that we can find a positive integer $m$ such that $I(\mathbb{F}_{p}[G])^{m}=0$. Since $I(R[G])=R\cdot I(\Z_{p}[G] )$, it therefore follows that
\begin{equation}\label{eq3.9}
I(R[G])^{m}\subset pR[G].
\end{equation}

Define the $F$-semi-linear map $\Psi :R[G] \rightarrow R[G]$ by the rule that $\Psi ( rg) =F( r) g^{p}$.
As in the Introduction $\phi :R[ G] \rightarrow R[ C_{G}] $ denotes the $R$-linear map given by sending each group element to
its conjugacy class. Write $\overline{\Psi }:R[C_{G}]
\rightarrow R[C_{G}]$ for the $F$-semi-linear map induced by $\Psi$.

We define the group logarithm $\mathcal{L}:1+I(R[ G] )\rightarrow
N[ C_{G}] $ by the rule that for $x\in I( R[ G]) $
\begin{eqnarray}\label{eq3.10}
\mathcal{L}( 1-x) &=&( p-\overline{\Psi }) \circ \phi
( \log ( 1-x) ) =\phi ( ( p-\Psi )
( \log ( 1-x) ) ) \\
&=&-\phi ( \sum_{n\geq 1}\frac{px^{n}}{n}-\sum_{n\geq 1}\frac{\Psi
( x^{n}) }{n})  \notag
\end{eqnarray}
which is seen to converge to an element of $N[ C_{G}] $ by (\ref{eq3.9})
above. Note that the map $\mathcal{L}=\mathcal{L}_{F}$ depends on the chosen
lift of Frobenius $F$; we shall therefore have to be particularly careful
when using such group logarithms for different coefficient rings which could
have incompatible lifts of Frobenius.

From Lemma 3.2 in [CPT1] we recall:

\begin{lemma}\label{le20}
For a character $\chi $ of $G$ and for $x\in I( R[ G]) $%
\begin{equation*}
\chi ( \phi ( \log ( 1+x) ) ) =\log (
\mathrm{Det}( 1+x) ( \chi) )
\end{equation*}%
and
\begin{equation*}
\chi ( \mathcal{L}( 1+x) ) =\log [\mathrm{Det}%
( 1+x) ( p\chi )\cdot \mathrm{Det}( 1+F( x)
) ( -\psi ^{p}\chi ) ]
\end{equation*}%
where $\psi ^{p}$ denotes the $p$-th Adams operation on virtual characters of
$G$, which is defined by the rule $\psi ^{p}\chi ( g) =\chi( g^{p}) $ for $g\in G$.
\endproof
\end{lemma}

As per Corollary 3.3 loc. cit. from this we may deduce:

\begin{corollary}\label{cor21}
 If $\mathrm{Det}( 1+x) =1$, then $\phi ( \log( 1+x) ) =0$, and so $\mathcal{L}( 1+x) =1$.
Thus there is a unique map $\nu :\mathrm{Det}( 1+I(R[ G]
)) \rightarrow N[ C_{G}]$ such that $\mathcal{L}=\nu \circ
\mathrm{Det.}$ \endproof
\end{corollary}

The main result for the group logarithm is (see Theorem 3.4.a in [CPT1]):

\begin{theorem}\label{thm22}
We have
$
\mathcal{L}( 1+I_{G}) \subset p\phi ( I_{G}) \subset
pR[ C_{G}]$.
\end{theorem}

We now recall a number of results on the images under $\mathcal{L}$ of
various subgroups of $1+I(R[G] )$. For our first such result we
set $\mathcal{A}(R[G]) ={\rm Ker}( R[ G]
\rightarrow R[G^{\rm ab}] ) $ and note that obviously $%
\mathcal{A}( R[ G] ) \subset I(R[G])$. Then from Theorem 3.5 loc. cit.
we have:

\begin{theorem}\label{thm23}
We have
$\mathcal{L}( 1+\mathcal{A}(R[ G])) =p\phi (
\mathcal{A}(R[G]))$.
\end{theorem}

We now assume $G$ to be non-abelian; then, since $G$ is a $p$-group, by
considering the lower central series of $G$, we may choose a commutator $c=%
[ \gamma ,\delta ] $ which has order $p$ and which lies in the
centre of $G$; it therefore follows that
\begin{equation*}
( 1-c) R[ G] \subset \mathcal{A}(R[ G] ).
\end{equation*}%
From Lemma 3.8 and Lemma 3.10 loc. cit. we recall:

\begin{lemma}\label{le24}
For $n\geq 1$, if $c$ is a commutator as above, we have
$
\mathcal{L}( 1+( 1-c)^{n}R[ G]) =p\phi
( ( 1-c) ^{n}R[G])
$. \endproof
\end{lemma}

If we now assume that $c$ is again a central element order of $p$ but that
now it cannot be written as a commutator in $G$, then from Lemmas 3.12 and
3.13 in [CPT1] we have the following two lemmas:

\begin{lemma}\label{le25}
If $c$ cannot be written as a commutator in $G$, then multiplication by $c$
permutes the elements of $C_{G}$ without fixed points and the kernel of
multiplication by $1-c$ on $R[C_{G}]$ is generated by elements
of the form $\sum\nolimits_{i=0}^{p-1}c^{i}\phi ( g) $ for $g\in
G$. \endproof
\end{lemma}

\begin{lemma}\label{le26}
Suppose again that $c$ cannot be written as a commutator in $G$; then
\begin{equation*}
\mathcal{L}( 1+( 1-c) I( R[ G] )
) \supseteq p\phi ( ( 1-c) I( R[ G]
) ) +p^{2}\phi ( 1-c) R
\end{equation*}%
and
\begin{equation*}
\frac{p\phi ( ( 1-c) R[ G] ) }{\mathcal{L}%
( 1+( 1-c) R[ G] ) }\cong \frac{R}{(
1-F) R+pR}
\end{equation*}
induced by mapping  $p\phi ( ( 1-c) rg) \longmapsto r\,
{\rm  mod}\, pR$ for $r\in R$, $g\in G$. \endproof
\end{lemma}

\medskip
 The relative $K$-group $\Kr_{1}( R[ G] ,I_{G})
$ was defined in Sect. \ref{s2}. Recall $\Kr_{1}^{\prime }( R[G]
, I_{G}) $ denotes the image of $\Kr_{1}( R[ G], I_{G}) $ in $\Kr_{1}(R[G])$. We then define
the Whitehead group ${\rm Wh}( R[ G]) $ to be
$\Kr_{1}^{\prime }( R[ G] , I_{G}) /\mbox{Im}(
G) $ and the determinantal Whitehead group ${\rm Wh}^{\prime }( R[
G])$ is defined to be $\mathrm{Det}( 1+I_{G}) /%
\mathrm{Det}( G)$. From Theorems 3.14 and 3.17 in [CPT1] we
have the exact sequence%
\begin{equation}\label{eq3.11}
1\rightarrow \mathrm{Det}( G) \rightarrow \mathrm{Det}(
1+I_{G}) \overset{\upsilon }{\rightarrow }\phi ( I_{G})
\overset{\omega }{\rightarrow }\frac{R}{( 1-F) R}\otimes
G^{\rm ab}\rightarrow 1.
\end{equation}%
where we write $\upsilon =p^{-1}\nu$. Recall that by hypothesis $\SK_{1}( R) =\{1\}$,
and so by using the commutative diagram with
exact horizontal sequence
\begin{equation*}
\begin{array}{ccccccccc}
1 & \rightarrow & \SK_{1}( R) & \rightarrow & \Kr_{1}( R)
& \rightarrow & \mathrm{Det}( R^{\times }) & \rightarrow & 1 \\
&  & \uparrow &  & \uparrow &  & \uparrow &  &  \\
1 & \rightarrow & \SK_{1}( R[ G]) & \rightarrow &
\Kr_{1}( R[ G]) & \rightarrow & \mathrm{Det}( R[ G] ^{\times }) & \rightarrow & 1
\end{array}%
\end{equation*}%
we see that $\SK_{1}( R[G]) \subseteq \Kr_{1}^{\prime
}( R[ G] ,I_{G})$. Noting that the image of $G$ in $\Kr_{1}^{\prime }( R[ G] ,I_{G}) \subset \Kr_{1}( R[G])$ is isomorphic to $\mathrm{Det}( G)$,
using the snake lemma we get the two exact sequences%
\begin{equation}\label{eq3.12}
1\rightarrow \SK_{1}( R[ G]) \rightarrow {\rm Wh}( R[G]) \rightarrow {\rm Wh}^{\prime }(R[G]) \rightarrow 1
\end{equation}%
\begin{equation}\label{eq3.13}
1\rightarrow \SK_{1}(R[G]) \rightarrow {\rm Wh}( R[G]) \overset{\Gamma }{\rightarrow }\phi (
I_{G}) \rightarrow \frac{R}{(1-F) R}\otimes
G^{\rm ab}\rightarrow 1.
\end{equation}%
where $\Gamma $ is obtained by composing $\mathrm{Det}$ with $p^{-1}\nu$.
Suppose now that we have an exact sequence of $p$-groups
\begin{equation*}
1\rightarrow H\overset{\iota }{\rightarrow }\widetilde{G}\overset{\alpha }{%
\rightarrow }G\rightarrow 1.
\end{equation*}%
Our study of $\SK_{1}$ will be based on the detailed study of the two groups
\begin{eqnarray}
\mathcal{H} &=&\ker ( \alpha _{\ast }) =\ker ( \SK_{1}( R[ \widetilde{G}] ) \rightarrow \SK_{1}( R[ {G}]) ) , \\
\mathcal{C} &=&\mathrm{coker }( \alpha _{\ast }) =\mathrm{coker
}( \SK_{1}( R[ \widetilde{G}] ) \rightarrow
\SK_{1}( R[ G]) ) .
\end{eqnarray}
\smallskip

% ----------------------------------------------------------------------
% ----------------------      BEGIN INSERT       -----------------------
% ----------------------------------------------------------------------

% ----------------------------------------------------------------------
\subsection{The $\mathcal{H}$-subgroup}

\begin{lemma}
If $\ G$ is an abelian $p$-group, then $\mathrm{SK}_{1}( R[ G
] ) =(0).$
\end{lemma}

\begin{proof} As previously, we know that $\mathrm{SK}_{1}( R[ G]
) \subset \mathrm{K}_{1}^{\prime }( R[G]
, I_{G} ) $ and so we may assume that $x\in \mathrm{SL} ( R [ G%
 ]  ) $  has the same image in $\mathrm{K}_{1} ( R [ G 
 ]  ) $ as $x^{\prime }\in  \mathrm{GL} ( R [ G ]
, I_{G} )$. Let $x\in \mathrm{SL} ( R [ G ]  ) $ be
represented by $x^{\prime }\in  \mathrm{GL} ( R [ G ]
, I_{G} )$. By Lemma 2.2(b) we may write $x^{\prime }=e_{1}\delta e_{2}$ with $e_{i}\in \mathrm{E} ( R [ G ] , I_{G} ) ,$
$\delta \in 1+I_{G}.$ Since $G$ is abelian, $\mathrm{Det}$ is an
isomorphism on $R [ G ] ^{\times };$ since $\mathrm{Det} (
\delta ) =1,$ we conclude that $\delta =1,$ and so $x^{\prime }\in
\mathrm{E} ( R [ G ] , I_{G} )$.
\end{proof}
 \medskip

We henceforth assume $G $ to be non-abelian and we recall some results from
[CPT1]. As previously we choose a central element $c $ of $G$ which has
order $p $ and we set $\overline{G}=G/\langle c\rangle ,$ so that
we have the exact sequence of groups
\begin{equation*}
1\rightarrow \langle c\rangle \rightarrow G\rightarrow \overline{G%
}\rightarrow 1.
\end{equation*}%
\ \textbf{Definitions.} We define
\begin{equation*}
\mathcal{B}=\mathcal{B}( R[ G] ) =\ker ( R[ G%
] \rightarrow R[ \overline{G}] ) =( 1-c) R%
[ G]
\end{equation*}%
and we set
\begin{equation*}
\mathcal{H=H}( R[ G] ) =\ker ( \SK_{1}( R%
[ G] ) \rightarrow \SK_{1}( R[ \overline{G}]
) ) .
\end{equation*} 
\medskip

 For the remainder of this subsection we shall suppose that $c$ is a
commutator in $G$; that is to say we can write
\begin{equation*}
c=[ h,g] =hgh^{-1}g^{-1},\ c^{p}=1.
\end{equation*}%
From Lemma 4.2 in [CPT1] we recall the exact sequence%
\begin{equation*}
1\rightarrow \mathrm{Det}( 1+( 1-c) R[ G] )
\rightarrow \mathrm{Det}( 1+I( R[ G] ) )
\rightarrow \mathrm{Det}( 1+I( R[ \overline{G}] )
) \rightarrow 1
\end{equation*}%
and so by Corollary 3.2, (3.3), Lemma 3.5 and noting that
\begin{equation*}
\mathrm{Det}( G) \cap \mathrm{Det}( 1+( 1-c) R%
[ G] ) =\{1\}
\end{equation*}%
(by projecting into $G^{\mathrm{ab}}$ and using $\mathrm{Det}( G)
\cong \mathrm{Det}( G^{\mathrm{ab}}) $ and the fact that $(
1-c) R[ G] $ is contained in $\ker ( R[ G]
\rightarrow R[ G^{\mathrm{ab}}] ) )$, we get the further
exact sequence
\begin{equation}
0\rightarrow p\phi ( \mathcal{B}) \rightarrow \mathrm{Det}(
1+I( R[ G] ) ) \rightarrow \mathrm{Det}(
1+I( R[ \overline{G}] ) ) \rightarrow 1.
\end{equation}%
Using the exact sequence
\begin{equation*}
1\rightarrow \mathrm{Det}( 1+I( R[ G] ) )
\rightarrow \mathrm{Det}( R[ G] ^{\times })
\rightarrow R^{\times }\rightarrow 1
\end{equation*}%
first for $ G$, and then with $G$ replaced by $\overline{G},$ and using the
snake lemma we get
\begin{equation*}
0\rightarrow p\phi ( \mathcal{B}) \rightarrow \mathrm{Det}( R%
[ G] ^{\times }) \rightarrow \mathrm{Det}( R[
\overline{G}] ^{\times }) \rightarrow 1.
\end{equation*}

\begin{proposition}
(a) The quotient map $q_{\ast }:\ \mathrm{SK}_{1}( R[ G]
) \rightarrow   \mathrm{SK}_{1}( R[ \overline{G}]
) $ is surjective;

(b)  for $\mathcal{H}$ as defined above,
\begin{equation*}
\mathcal{H=}\{ \kappa ( x) \mid x\in \mathrm{GL}( R%
[ G] , \mathcal{B} )\  \text{with }\mathrm{Det}(
x) =1\} .
\end{equation*}
\end{proposition}

\begin{proof} This follows at once from the commutative diagram%
\begin{equation*}
\begin{array}{ccccccc}
0\rightarrow & \mathcal{H} & \hookrightarrow & \mathrm{SK}_{1}( R[
G] ) &  \xrightarrow{q_{\ast }}   & \mathrm{SK}_{1}(
R[ \overline{G}] ) & \rightarrow 0 \\
& \downarrow &  & \downarrow &  & \downarrow &  \\
0\rightarrow & \mathrm{K}_{1}^{\prime }(R[ G] ,\mathcal{B)} &
\hookrightarrow & \mathrm{K}_{1}(R[ G] \mathcal{)} & \rightarrow
& \mathrm{K}_{1}(R[ \overline{G}] \mathcal{)} & \rightarrow 0 \\
& \downarrow &  & \downarrow &  & \downarrow &  \\
0\rightarrow & \mathrm{Det}( \mathrm{K}_{1}(R[ G] ,\mathcal{%
B)}) & \hookrightarrow & \mathrm{Det}( \mathrm{K}_{1}(R[ G%
] \mathcal{)}) & \rightarrow & \mathrm{Det}( \mathrm{K}%
_{1}(R[ \overline{G}] \mathcal{)}) & \rightarrow 0.%
\end{array}%
\end{equation*}

Here the middle row is exact: indeed, the right-hand arrow is surjective, as
$\mathcal{B}$ is contained in the Jacobson radical of $R[ G] $;
and the fact that $\mathrm{K}_{1}^{\prime }(R[ G] ,\mathcal{B)}$
is equal to the kernel of this map follows from (2.1). By Lemma 2.2 (b) we
know that $\mathrm{Det}( \mathrm{K}_{1}(R[ G] ,\mathcal{B)}%
) =\mathrm{Det(}1+\mathcal{B)}$ and by (3.8) we can identify $\mathrm{%
Det}( 1+\mathcal{B)}) $ with $p\phi ( \mathcal{B}) ;$
the exactness of the middle row then follows   from the exact sequence given
prior to the statement of the proposition. Therefore, by the snake lemma, we
know that the top row is exact, and in particular $q_{\ast }$ is
surjective. 
\end{proof} 

\bigskip

Recall that $R_{n}=R/p^{n}R$. We now apply the reduction map $%
\mathrm{K}_{1}( R[ G] ) \rightarrow \mathrm{K}%
_{1}( R_{n}[ G] ) $ to the exact sequence
\begin{equation*}
0\rightarrow \mathcal{H}\rightarrow \mathrm{SK}_{1}( R[ G]
) \rightarrow \mathrm{SK}_{1}( R[ \overline{G}]
) \rightarrow 0.
\end{equation*} 
Recall from the Introduction that $\mathrm{SK}_{1}( R[ G]
) _{n}$ denotes the image of $\mathrm{SK}_{1}( R[ G]
) $ in $\mathrm{K}_{1}( R_{n}[ G] )$. Writing $%
\mathcal{H}_{n}$ for the image of $\mathcal{H}$ in $\mathrm{K}_{1}(
R_{n}[ G] )$, we shall now show:

\begin{proposition}
The following sequence is exact:
\begin{equation*}
0\rightarrow \mathcal{H}_{n}\rightarrow \mathrm{SK}_{1}( R[ G%
] ) _{n}\rightarrow \mathrm{SK}_{1}( R[ \overline{G}%
] ) _{n}\rightarrow 0.
\end{equation*}
\end{proposition}

Before proving this proposition we first require a preliminary result:

\begin{lemma}
For brevity we shall write $( ( 1-c) ) $ and $(
p^{n}) $ for the $R[ G] $-ideals $( 1-c) R[
G] $ and $p^{n}R[ G] .$

(a) $\mathrm{Det}( 1+( ( 1-c) ) ) \cong
p\phi ( ( 1-c) ) .$

(b) For $n\geq 1,\ \mathrm{Det}( 1+( p^{n}( 1-c)
) ) \cong p^{n+1}\phi ( ( 1-c) ) .$

(c) For $n\geq 2,\ $\ the image under $\nu $ of $\mathrm{Det}( 1+(
p^{n}) ) \ $is contained$\ $in$\ \phi ( ( p^{n})
) .$

(d) For $n\geq 2$

\begin{equation*}
\mathrm{Det}( 1+( ( 1-c) ) ) \cap \mathrm{%
Det}( 1+( p^{n}) ) =\mathrm{Det}( 1+p^{n}(
1-c) ) .
\end{equation*}
\end{lemma}

\begin{proof} First note again that, because $\psi ^{p}( 1-c)
=1-c^{p}=0,\ $\ it follows that $p\phi \circ \log $ coincides with the group
logarithm on $1+( ( 1-c) ) .$ By Corollary 3.2 we know
that $p.\phi \circ \log =\nu \circ \mathrm{Det}$ on $1+( (
1-c) ) $; $\ $and so by Lemma 3.5%
\begin{equation*}
\mathrm{Det}( 1+( ( 1-c) ) ) \cong p\phi
( ( 1-c) ) ,
\end{equation*}%
and\ for $n\geq 1$ the result (b) now follows from Proposition 2.4 (which
only states the result for $n\geq 2$ although the argument holds for $n=1$)
in [CPT1]
\begin{equation*}
\mathrm{Det}( 1+( p^{n}( 1-c) ) ) \cong
p^{n+1}\phi ( ( 1-c) ) .
\end{equation*}%
To see (c) we note that for $n\geq 2$%
\begin{equation}
\mathcal{L}( 1-p^{n}x) =-\phi ( \sum_{m\geq 1}\frac{%
p^{1+mn}x^{m}}{m}-\sum_{m\geq 1}\frac{\Psi ( p^{nm}x^{m}) }{m}%
) \subset p^{n}R[ C_{G}] .  \notag
\end{equation}

The result (d) then follows since we know that
\begin{equation*}
p\phi ( ( p^{n-1}) ) \cap p\phi ( (
1-c) ) =p\phi ( ( p^{n-1}( 1-c) )
) .\ \ \ \ \
\end{equation*}%
To conclude, if $x\in \mathrm{Det}( 1+( ( 1-c) )
) \cap \mathrm{Det}( 1+( p^{n}) ) $ we have $\nu
( x) =\nu ( y) $ for some $y\in \mathrm{Det}(
1+p^{n-1}( 1-c) r) ,$ for some $r\in R[ G] ;$
the result then follows since $\nu $ is injective on $\mathrm{Det}(
1+( ( 1-c) ) )$.
\end{proof} 
\medskip

\noindent\textit{Proof of Proposition 3.10.} Consider the commutative diagram
\begin{equation*}
\begin{array}{ccccccccc}
0 & \rightarrow & \mathcal{H} & \rightarrow & \mathrm{SK}_{1}( R[ G%
] ) & \xrightarrow{q_{\ast }} & \mathrm{SK}_{1}( R%
[ \overline{G}] ) & \rightarrow & 0 \\
&  & \downarrow \alpha &  & \downarrow \beta &  & \downarrow \gamma &  &  \\
0 & \rightarrow & \mathcal{H}_{n} & \xrightarrow{i} & \mathrm{SK}%
_{1}( R[ G] ) _{n} & \xrightarrow{j} &
\mathrm{SK}_{1}( R[ \overline{G}] ) _{n} & \rightarrow
& 0%
\end{array}%
\end{equation*}%
in which, by Proposition 3.9, we know the top row to be exact. By definition
$i$ is an inclusion map and hence is injective; $j\ $is surjective since $%
\gamma $ and $q_{\ast }$ are surjective; clearly $j\circ i=0.$ Therefore it
remains to show $\ker j\subset \Im i.$ For brevity given $x\in \mathrm{%
GL}( R[ G] ) ,$ we shall write $\overline{x}$ for its
image in $\mathrm{GL}( R[ \overline{G}] )$. Now
consider $x\in \mathrm{SK}_{1}( R[ G] ) $ with $j\circ
\beta ( x) =1$; then $\gamma ( \overline{x}) =\gamma
\circ q_{\ast }( x) =1$ and so we can find $a_{i},b_{i}\in
\mathrm{GL}( R[ G] ) $ such that
\begin{equation*}
\overline{x}=\prod\limits_{i}[ \overline{a}_{i},\overline{b}_{i}]
( 1+p^{n}\overline{\lambda }_{1})
\end{equation*}%
with $\lambda _{1}\in \mathrm{M}( R[ G] ) ;$
furthermore by Lemma 2.2(b) we can find elementary matrices $d_{1},d_{2}\in
\mathrm{E}( R[ G] ,p^{n}) $ and $\lambda _{2}\in R%
[ G] ^{\times } $ so that $1+p^{n}\lambda _{1}=d_{1}(
1+p^{n}\lambda _{2}) d_{2} $,  where we view $ R[ G%
] ^{\times }$ as a subgroup of $ \mathrm{GL}( R[ G]
) $ in the usual way. Thus we can deduce that for some $\mu \in
\mathrm{M}( R[ G] ) $
\begin{equation*}
x=\prod\limits_{i}[ a_{i},b_{i}] d_{1}( 1+p^{n}\lambda
_{2}) d_{2}\cdot ( 1+( 1-c) \mu ) .
\end{equation*}%
Using Lemma 2.2(b) again we can write 
\begin{equation*}
( 1+( 1-c) \mu ) =e_{1}( 1+( 1-c) \mu
^{\prime }) e_{2}
\end{equation*}%
with $e_{i}\in \mathrm{E}( R[ G] , I) $ and $\mu
^{\prime }\in R[ G] $. Therefore, taking determinants  and using
the fact that $\mathrm{Det}( x) =1,$ we get%
\begin{equation*}
\mathrm{Det}( 1+( 1-c) \mu ) =\mathrm{Det}(
1+( 1-c) \mu ^{\prime }) \in \mathrm{Det}( 1+(
p^{n}) ) \cap \mathrm{Det}( 1+( ( 1-c)
) )
\end{equation*}%
and Lemma 3.11 shows that we can write
\begin{equation}
\mathrm{Det}( 1+( 1-c) \mu ) =\mathrm{Det}(
1+p^{n}( 1-c) \xi )
\end{equation}%
for some $\xi \in $ $R[ G] .\ $Therefore
\begin{equation}
( 1+( 1-c) \mu ) =( 1+p^{n}( 1-c) \xi
) \tau
\end{equation}%
for some element $\tau \in \mathrm{GL}( R[ G] ,I) \ $%
with the property that $\mathrm{Det}( \tau ) =1;$ hence $\tau \in
\mathcal{H}$, and so we are done, since $\beta ( x) $ is equal to
the class of $i( \tau ) \in \Im( i) ,$ as
required.\endproof
\medskip

% ----------------------------------------------------------------------
\subsection{Some commutator identities}

We now need to analyze   $\mathcal{H}_{R}=\mathcal{H}( R[ G]
) $  in greater detail. The various groups $ \mathcal{H}_{R}$ will
provide the inductive building blocks for   $\mathrm{SK}_{1}( R[ G%
] ) .$ By Lemma 2.2(b), the elements of $\mathrm{SK}_{1}( R%
[ G] )  $ are represented by  elements of $R[ G]
^{\times }$ which have trivial determinant but which are not themselves
products of commutators (in $\mathrm{GL}( R[ G] ) ).$
This then leads us to establish a number of commutator relations and
congruences, which we then use in subsequent subsections.  All these appear
either explicitly or implicitly in Oliver's work (see for example
[O1]). One important point of this paper is that most of these
identities from Oliver's work (which concerns rings of integers of finite
extensions of $\mathbb{Q}_{p}$) continue to hold when the ring $R$ satisfies
our standing hypotheses. In what follows, we explain this and also
reorganize some of the material from Oliver's papers. \medskip

As in the previous parts of this section we continue to suppose that $G$ is
a $p$-group.  First we suppose $G$ to be non-abelian and, as previously, $c$
denotes a central commutator of order $p,$%
\begin{equation*}
c=[ h,g] =hgh^{-1}g^{-1}
\end{equation*}%
and we put $\overline{G}=G/\langle c\rangle $; hence
\begin{equation}
cg=hgh^{-1}
\end{equation}%
also $c=g^{-1}cg=g^{-1}hgh^{-1}\,$and so $ch=g^{-1}hg. $ We then note that
for $n>0,$ and $\lambda \in R$ we have
\begin{eqnarray*}
[ g^{-1}h,1-\lambda ( g-h) ^{n}] &=&g^{-1}h(
1-\lambda ( g-h) ^{n}) h^{-1}g( 1-\lambda (
g-h) ^{n}) ^{-1} \\
&=&( 1-\lambda c^{n}( g-h) ^{n}) ( 1-\lambda
( g-h) ^{n}) ^{-1}.  \notag
\end{eqnarray*}
Noting the entirely obvious identity in non-commuting indeterminates $X$ and
$Y$
\begin{equation}
( 1-YX) ( 1-X) ^{-1}=1+( 1-Y) X(
1-X) ^{-1},
\end{equation}%
setting $X=\lambda ( g-h) ^{n},\ Y=c^{n},$ and observing that $%
( 1-c^{n}) =( 1+c+\cdots +c^{n-1}) ( 1-c) ,$
we get
\begin{eqnarray*}
[ g^{-1}h,1-\lambda ( g-h) ^{n}] &=&1+\lambda (
1-c^{n}) ( g-h) ^{n}( 1-\lambda ( g-h)
^{n}) ^{-1} \\
&\equiv &1+n( 1-c) [ \lambda ( g-h) ^{n}+\lambda
^{2}( g-h) ^{2n}+ .\\
& &+  \lambda ^{3}( g-h) ^{3n}+\cdots ] \mod ( 1-c)^{2}.  \notag
\end{eqnarray*}

We now require two further such commutator congruences, both of whose proofs
are entirely similar, but which are in fact slightly easier.

The first is for $n>1.$ For such $n$ we have
\begin{equation*}
[ h,1-\lambda ( 1-c) ^{n-1}g] =( 1-\lambda (
1-c) ^{n-1}hgh^{-1}) ( 1-\lambda ( 1-c)
^{n-1}g) ^{-1}
\end{equation*}%
and, setting$\;X=\lambda ( 1-c) ^{n-1}g$ and $\;Y=hgh^{-1}g^{-1}$
in (3.13)$,$ we get
\begin{equation}
[ h,1-\lambda ( 1-c) ^{n-1}g] =1+\lambda (
1-c) ^{n-1}( g-hgh^{-1}) ( 1-\lambda ( 1-c)
^{n-1}g) ^{-1}  \notag
\end{equation}%
\begin{equation}
=1+\lambda ( 1-c) ^{n}g( 1-\lambda ( 1-c)
^{n-1}g) ^{-1}
\end{equation}%
and so%
\begin{equation}
\equiv 1+\lambda ( 1-c) ^{n}g \mod ( 1-c)^{n+1}.
\end{equation}
For the second relation we again take $n>0$ and we no longer insist that $g$
and $h$ be chosen such that $\;c=[ h,g] $ and we still suppose
that $c$ is central. We then have
\begin{equation*}
[ h,1-\lambda ( 1-c) ^{n}g] =( 1-\lambda (
1-c) ^{n}hgh^{-1}) ( 1-\lambda ( 1-c)
^{n}g) ^{-1}
\end{equation*}%
and by the identity (3.13) with $X=\lambda ( 1-c) ^{n}g,\
Y=hgh^{-1}g^{-1},\ $we have
\begin{equation}
[ h,1-\lambda ( 1-c) ^{n}g] =1+\lambda (
1-c) ^{n}( g-hgh^{-1}) ( 1-\lambda ( 1-c)
^{n}g) ^{-1}  \notag
\end{equation}%
\begin{equation}
\equiv 1+\lambda ( 1-c) ^{n}( g-hgh^{-1}) \mod
( 1-c) ^{n+1}.
\end{equation}

% ----------------------------------------------------------------------
\subsection{The groups $\mathcal{I}$ and $\mathcal{J}$}

As always we suppose that$\ R$ satisfies the Standing Hypotheses; note that
it is here that we shall use the hypothesis that $pR$ is a prime ideal of $%
R. $ Initially in this subsection we suppose that $c$ is a central element
of $G$ of order $p$ which is not necessarily a commutator.\smallskip

We start our analysis of $\mathcal{H}( R[ G] ) $ by
using the above commutator identities taken together with the use of logarithmic
methods.\medskip

\noindent\textbf{Definition.}  Recall that $\mathcal{H}_{R}$ denotes $\ker (
\mathrm{SK}_{1}( R[ G] ) \rightarrow \mathrm{SK}%
_{1}( R[ \overline{G}] ) ) .$ We shall now
analyze $\mathcal{H}_{R}$ by means of the following groups (c.f. page 203 in
[O1]):
\begin{equation*}
\mathcal{I}_{R}=\ker [ ( 1+( 1-c) R[ G]
)  \xrightarrow{\alpha}\mathrm{K}_{1}( R[ G]
) ]
\end{equation*} 
\begin{equation*}
\mathcal{J}_{n,R}=\ker [ ( 1+( 1-c) ^{n}R[ G]
)   \xrightarrow{\alpha_n }\mathrm{Det}( R[ G%
] ^{\times }) ]
\end{equation*}%
where the map $\alpha $ is induced by the map $\mathrm{GL}( R[ G%
] ) \rightarrow \mathrm{K}_{1}( R[ G] ) $
and $\alpha _{n}$ is induced by the determinant. We shall frequently write $ 
\mathcal{J}_{R}$ for $\mathcal{J}_{1,R}$. Obviously we have the inclusion $%
\mathcal{I}_{R}\subset \mathcal{J}_{R}$. When the ring $R$ is clear from the
context we shall just write $\mathcal{I}$, $\mathcal{J}$ in place of $\mathcal{I%
}_{R}$, $\mathcal{J}_{R}.$ \medskip

Recall that $N$ denotes the field of fractions of $R$ and $\phi :N[ G%
] \rightarrow N[ C_{G}] $ is the $N$-linear map induced by
mapping each element of $G$ to its conjugacy class.

\begin{lemma}
In this lemma we do not suppose that $R$ satisfies all the Standing
Hypotheses but instead we only suppose that $R$ is an integral domain
whose field of fractions has characteritic zero. Define  
\begin{equation*}
\mathcal{S}_{G}=\{ g\in G\mid \{ c^{m}g\} \text{ are
all conjugate to }g\text{ for all }0\leq m<p\}
\end{equation*}%
and let $D_{G}$ denote the set of conjugacy classes $\phi ( \mathcal{S}%
_{G})$. Then, for $n\geq 1,$ the kernel of $\phi :( 1-c)
^{n}R[ G] \rightarrow R[ C_{G}] $ consists of the $R$%
-linear span of the following two kinds of element:

(\textbf{Type 1}) $( 1-c) ^{n}g$ for$\ g\in \mathcal{S}_{G};$

(\textbf{Type 2})$\;( 1-c) ^{n}( g-hgh^{-1}) $ for $%
g,h\in G.$\
\end{lemma}

\begin{proof} Suppose that $n\geq 1,$ as in the statement of the lemma.The
central subgroup $\langle c\rangle $\ of order $p$ acts naturally
by multiplication on the conjugacy classes $C_{G}.$  Thus $R[ C_{G}%
] $ is a permutation $R\langle c\rangle $-module, and as
such it is isomorphic to a direct sum of copies of $R\langle
c\rangle $ (on which $( 1-c) ^{n}$ acts faithfully, since $%
R $ has characteristic\ zero) and a sum of copies of $R$ on which $(
1-c) ^{n}$ acts as 0. The latter summands form exactly the $R$-linear
span $R[ D_{G}] $ of $D_{G} $. We conclude that for every $\mu
\in R[ G] $ there is an element $\lambda \in R[ \mathcal{S}%
_{G}] $ such that the image $\phi ( \mu -\lambda ) $ of $%
\mu -\lambda $ in $R[ C_{G}] $ is annihilated by $(
1-c) ^{n}\ $if and only if $\phi ( \mu -\lambda ) =0 $. We
have
\begin{equation*}
( 1-c) ^{n}\mu =( 1-c) ^{n}( \mu -\lambda )
+( 1-c) ^{n}\lambda
\end{equation*}%
where $\phi ( ( 1-c) ^{n}\lambda ) =0.$ Thus $\phi
( ( 1-c) ^{n}\mu ) =0$ if and only if $\phi (
( 1-c) ^{n}( \mu -\lambda ) ) =( 1-c)
^{n}\phi ( ( \mu -\lambda ) ) =0,$ and this forces $%
\phi ( \mu -\lambda ) =0.$ Thus $\mu -\lambda $ is an $R$-linear
combination of elements of the form $g-ghg^{-1}$ with $g,h\in G,$ and this
implies the lemma.
\end{proof}
\medskip

\noindent \textbf{Definition.} We define $T_{R,i}$ for $i=1,2$ to be the following $R$%
-submodules of $R[ G] $%
\begin{equation*}
T_{R,1}=\sum\limits_{g\in \mathcal{S}_{G}}Rg,\quad
T_{R,2}=\sum\limits_{g,h}R( g-hgh^{-1}) \ .\
\end{equation*}%
We refer to elements of $T_{R,i}$ as elements of Type $i$.
\medskip

\begin{proposition}
If $u\in \mathcal{J}_{n}$, then $\log ( u) $ is of the form $%
( 1-c) ^{n}\xi $ with $\xi $ an $R$-linear sum of terms of types
1 and 2.
\end{proposition}

\begin{proof} Let $\chi $ denote an arbitrary character of $G.\ $The result then
follows at once from the equality
\begin{equation*}
\chi ( \phi \circ \log ( u) ) =\log ( \mathrm{Det}%
_{\chi }( u) ) =0
\end{equation*}%
since, by varying $\chi ,$ this implies $\phi \circ \log ( u) =0$.
\end{proof}

\bigskip

The following result explains the key role of $\mathcal{I}$ and $\mathcal{J\medskip }$ in understanding $\mathcal{H}.$

\begin{lemma} We have
$
\mathcal{H}\cong \mathcal{J}/\mathcal{I}.
$
\end{lemma}

\begin{proof} From (2.1) we have the exact sequence%
\begin{equation*}
1\rightarrow \Kr_{1}^{\prime }( R[ G] ,( (
1-c) ) ) \rightarrow \Kr_{1}( R[ G] )
\rightarrow \Kr_{1}( R[ \overline{G}] ) \rightarrow 1
\end{equation*}%
where as previously $\Kr_{1}^{\prime }( R[ G] ,( (
1-c) ) ) $ denotes the image of $\Kr_{1}( R[ G%
] ,( ( 1-c) ) ) $ in $\Kr_{1}( R[ G%
] ) .$ Clearly%
\begin{equation*}
\ker ( \SK_{1}( R[ G] ) \rightarrow \SK_{1}( R%
[ \overline{G}] ) ) =\ker ( \Kr_{1}( R[ G%
] ) \rightarrow \Kr_{1}( R[ \overline{G}] )
\oplus \mathrm{Det}( R[ G] ) )
\end{equation*}%
and by the above we can write this as
\begin{eqnarray*}
\ker ( \SK_{1}( R[ G] ) \rightarrow \SK_{1}( R%
[ \overline{G}] ) ) &=&\ker ( \Kr_{1}( R[
G] ,( ( 1-c) ) ) \rightarrow \mathrm{Det}%
( R[ G] ) ) \\
&=&\ker ( \kappa _{R[ G] }) :1+( (
1-c) ) \rightarrow \mathrm{Det}( R[ G] ) \\
&=&\mathcal{J}/\mathcal{I}. 
\end{eqnarray*}
\end{proof}
\medskip

We now analyze the groups $\mathcal{I}$\ and $\mathcal{J}$ by means of
filtrations. We will show:

\begin{lemma}
For $n\geq 2$,
\begin{equation}
\mathcal{I\cap }( 1+( ( 1-c) ^{n}) ) =%
\mathcal{J\cap }( 1+( ( 1-c) ^{n}) ) \mod ( ( 1-c) ) ^{n+1}.
\end{equation}
\end{lemma}

\begin{proof} (Note that this proof is very similar to that of Theorem
1.5 with the role of $p^{n}$ replaced by $( 1-c) ^{n}.)$

We show (3.16) for all $n\geq 2.$ The inclusion $\subset $ results from the
fact that $\mathcal{I}\subset \mathcal{J}.$ Conversely, to show the
inclusion $\supset ,$ we consider a typical generator $u$ of $\mathcal{J\cap
}( 1+( ( 1-c) ^{n}) ) .$ Note that since $%
n\geq 2$ we know%
\begin{equation*}
\log ( u) \equiv u-1\mod ( 1-c) ^{n+1}.
\end{equation*}%
We wish to show that $u\in \mathcal{I}\mod (( 1-c)
) ^{n+1}.\ $Since $\mathrm{Det}( u) =1,$ by Proposition
3.13 we know that we can write
\begin{equation*}
\log ( u) =\sum\nolimits_{i}n_{i}( 1-c)
^{n}g_{i}+\sum\nolimits_{j}m_{j}( 1-c) ^{n}(
h_{j}-h_{j}^{l_{j}})
\end{equation*}%
where $n_{i},\ m_{j}\in R$, the $g_{i}$ are of Type 1 and the $l_{j}\in G$
(so that the latter right hand terms are all of Type 2). Thus for some $%
k_{i}\in G$ we have
\begin{equation*}
[ k_{i},g_{i}] =c.
\end{equation*}%
We therefore have shown that%
\begin{equation*}
u\equiv 1+\sum\nolimits_{i}n_{i}( 1-c)
^{n}g_{i}+\sum\nolimits_{j}m_{j}( 1-c) ^{n}(
h_{j}-h_{j}^{l_{j}}) \mod( 1-c) ^{n+1}
\end{equation*}%
and so we have shown that we have the congruence $\mod(
1-c) ^{n+1}$
\begin{equation*}
u\equiv \prod\nolimits_{i}( 1+n_{i}( 1-c) ^{n}g_{i})
\prod\nolimits_{j}( 1+m_{j}( 1-c) ^{n}(
h_{j}-h_{j}^{l_{j}}) )
\end{equation*}

Since
\begin{equation*}
[ k_{i},g_{i}^{k_{i}}] =[ k_{i},k_{i}^{-1}g_{i}k_{i}] =%
[ k_{i},g_{i}c^{-1}] =[ k_{i},g_{i}] =c
\end{equation*}%
$\ $by (3.14), we know that%
\begin{equation*}
[ k_{i},( 1-n_{i}( 1-c) ^{n-1}g_{i}) ]
\equiv 1+n_{i}( 1-c) ^{n}g_{i}\mod( 1-c) ^{n+1}
\end{equation*}%
and of course%
\begin{eqnarray*}
[ ( 1+m_{j}( 1-c) ^{n}h_{j}) ,l_{j}^{-1}]
&=&( 1+m_{j}( 1-c) ^{n}h_{j}) l_{j}^{-1}(
1+m_{j}( 1-c) ^{n}h_{j}) ^{-1}l_{j} \\
&\equiv &( 1+m_{j}( 1-c) ^{n}h_{j}) (
1-m_{j}( 1-c) ^{n}h_{j}^{l_{j}}) \mod(
1-c) ^{n+1} \\
&\equiv &( 1+m_{j}( 1-c) ^{n}(
h_{j}-h_{j}^{l_{j}}) ) \mod( 1-c) ^{n+1}
\end{eqnarray*}%
and so 
\begin{equation*}
u\equiv \prod\nolimits_{i}[ k_{i},1-n_{i}( 1-c) ^{n-1}g_{i}%
] \prod\nolimits_{j}[ 1+m_{j}( 1-c) ^{n}h_{j},\
l_{j}^{-1}] \mod( 1-c) ^{n+1}.
\end{equation*}%
This then shows that $u\in \mathcal{I}\mod( ( 1-c)
) ^{n+1}$ as required.  
\end{proof} 
\medskip

\noindent \textbf{Definition.} We let $\overline{\mathcal{I}}$, resp.\ $%
\overline{\mathcal{J}}$, denote the image of $\mathcal{I}$, resp. $\mathcal{J},
$   in  $( 1+( ( 1-c) ) ) \mod(
( 1-c) ^{2}) .$ We note that by Lemma 3.15, the above shows
that
\begin{equation}
\mathcal{H}\cong \overline{\mathcal{J}}/\overline{\mathcal{I}}.
\end{equation}%

% ----------------------------------------------------------------------
\subsection{Logarithmic methods}

Next we analyze the groups $\overline{\mathcal{I}}$  and $\overline{\mathcal{%
J}}$  via logarithmic methods in order to evaluate the right-hand term in
(3.17). In the remainder of this subsection we analyze $\overline{\mathcal{J}%
}$; we\ shall then analyze $\overline{\mathcal{I}}$ in the next subsection.

We write $\overline{R}=R/{pR}$. 
Note that for any $r\in R[ G
] $, if $( 1-c) r=0,$ then $r$ is a multiple of $%
\sum_{n=0}^{p-1}c^{n}.$ From Lemma 3.6 in [CPT1] we quote the elementary congruence:

\begin{lemma}
We have $( 1-c)^{p}\equiv -p( 1-c) \mod p(
1-c) ^{2}.$
\end{lemma}

\begin{lemma}
There is a natural isomorphism%
\begin{equation*}
\frac{1+( 1-c) R[ G] }{1+( 1-c) ^{2}R[
G] }\cong \frac{( 1-c) R[ G] }{( 1-c)
^{2}R[ G] }\cong \overline{R}[ G] /((
1-c)) ,
\end{equation*}
\end{lemma}

\begin{proof} First we consider the commutative diagram with exact rows where the
vertical arrows are multiplication by $( 1-c) $
\begin{equation*}
\begin{array}{ccccccccc}
0 & \rightarrow & ( 1-c) \overline{R}[ G] & \rightarrow
& \overline{R}[ G] & \rightarrow & \overline{R}[ G]
/( 1-c) & \rightarrow & 0 \\
&  & \downarrow &  & \downarrow &  & \downarrow &  &  \\
0 & \rightarrow & ( 1-c) ^{2}\overline{R}[ G] &
\rightarrow & ( 1-c) \overline{R}[ G] & \rightarrow &
( 1-c) \overline{R}[ G] /( 1-c) ^{2}%
\overline{R}[ G] & \rightarrow & 0.%
\end{array}%
\end{equation*}%
Since the left and central downward arrows both have the same kernel, namely
$( 1-c) ^{p-1}\overline{R}[ G] ,$ by the snake lemma
we get%
\begin{equation}
\overline{R}[ G] /( 1-c) \cong ( 1-c)
\overline{R}[ G] /( 1-c) ^{2}\overline{R}[ G%
] .
\end{equation}%
Next we consider the commutative diagram where the vertical arrows are
induced by reduction $\mod {p}$%
\begin{equation*}
\begin{array}{ccccccccc}
0 & \rightarrow & ( 1-c) ^{2}R[ G] & \rightarrow &
( 1-c) R[ G] & \rightarrow & ( 1-c) R [
G] /( 1-c) ^{2}R[ G] & \rightarrow & 0 \\
&  & \downarrow &  & \downarrow &  & \downarrow &  &  \\
0 & \rightarrow & ( 1-c) ^{2}\overline{R}[ G] &
\rightarrow & ( 1-c) \overline{R}[ G] & \rightarrow &
( 1-c) \overline{R}[ G] /( 1-c) ^{2}%
\overline{R}[ G] & \rightarrow & 0%
\end{array}%
\end{equation*}%
and note that the left and central vertical maps both have kernel $p(
1-c) R[ G] =( 1-c) ^{p}R[ G] $. This
then gives the isomorphism%
\begin{equation*}
( 1-c) R[ G] /( 1-c) ^{2}R[ G]
\cong ( 1-c) \overline{R}[ G] /( 1-c) ^{2}%
\overline{R}[ G]
\end{equation*}%
as required. 
\end{proof}
\medskip

\noindent\textbf{Definition. } We define $L$ to be the composite map%
\begin{equation*}
L:\frac{1+( 1-c) R[ G] }{1+( 1-c) ^{2}R%
[ G] }\cong \frac{( 1-c) R[ G] }{(
1-c) ^{2}R[ G] }\cong \overline{R}[ G] /
( 1-c)  \xrightarrow{\phi }\overline{R}[ C_{G}%
] /(( 1-c)) .
\end{equation*}

\begin{lemma}
The logarithm induces a map $\log :1+( 1-c) R[ G]
\rightarrow ( 1-c) R[ G] $ and for $r\in R[ G%
] ,$ using Lemma 3.16,  we have
\begin{equation*}
\log ( 1+( 1-c) r) \equiv ( 1-c) (
r-r^{p}) \mod( 1-c) ^{2}R[ G] .
\end{equation*}
\end{lemma}
\medskip

\noindent\textbf{Definition. } Writing $( 1-c) ^{-1}$ for the isomorphism
in (3.18) we define
\begin{equation*}
\lambda :\frac{1+( 1-c) R[ G] }{1+( 1-c)
^{2}R[ G] }\rightarrow \overline{R}[ C_{G}] /
(( 1-c))
\end{equation*}%
to be the composite
\begin{equation*}
\lambda =\phi \circ ( 1-c) ^{-1}\circ \log \mod(
( 1-c) ,p) =( 1-c) ^{-1}\phi \circ \log \mod ( ( 1-c) ,p)
\end{equation*}%
$\ $\ where%
\begin{equation*}
\lambda ( 1+( 1-c) x) =\phi \circ ( 1-c)
^{-1}\circ \{\log ( 1+( 1-c) x) \mod(
( 1-c) ,p) \}.
\end{equation*}

\bigskip

In summary, we see that, using the map $L$ defined above, we can view $%
\lambda $ as the composite%
\begin{equation}
\lambda :\frac{1+( 1-c) R[ G] }{1+( 1-c)
^{2}R[ G] }  \xrightarrow{ \ L \ }\overline{R}[ C_{G}%
] /(( 1-c))   \xrightarrow{1-\Psi }\overline{%
R}[ C_{G}] /( (1-c)) .
\end{equation}
\medskip

\begin{lemma}
Using the above Lemma 3.18 we obtain
\begin{equation}
\lambda ( 1+( 1-c) r) \equiv \phi ( r)
-\phi ( r^{p}) \ \equiv \phi ( r) -\Psi \circ \phi
( r) \mod( ( 1-c) ,p)
\end{equation}
\end{lemma}

\begin{proof} We need to show $\phi ( r^{p}) \ \equiv \Psi \circ \phi
( r) \mod pR[ C_{G}] $ and this comes from
Proposition 2.2 in Ch. 60 of [T].
\end{proof}
\medskip

\begin{lemma} We have:

(a) \ \ 
$
\ker ( L) =1+( 1-c) T_{\overline{R},2}%
\mod( 1-c) ^{2};
$

(b) \ \ $
\ker ( \lambda ) =( 1+( 1-c) T_{\overline{R}%
,2}) \langle c\rangle \mod( 1-c) ^{2}$.
\end{lemma}

\begin{proof}By the definition of $L$ we see that $\ker(L)$ is isomorphic, under the map in Lemma 3.17, to the kernel of the map from $\overline{R}[G]$ to $\overline{R}[C_{G}]$; that is to say $T_{\overline{R},2}$.
This proves part (a).

We now prove (b).  Note that, since $1-\Psi $ is the identity on $(
1-c) \overline{R}[ C_{G}] ,$ we have the equalities%
\begin{eqnarray*}
\ker ( 1-\Psi :\overline{R}[ C_{G}] /
1-c) \rightarrow \overline{R}[ C_{G}]/(
1-c) ) &=&\ker ( 1-\Psi :\overline{R}[ C_{G}]
\rightarrow \overline{R}[ C_{G}] ) \\
&=&\ker ( 1-\Psi :\overline{R}\rightarrow \overline{R}) =\mathbb{F%
}_{p}
\end{eqnarray*}%
with the penultimate equality holding because $\Psi $ is  nilpotent on the
augmentation ideal $I( \overline{R}[ G] ) $ and with
the final equality holding because, by  hypothesis,  $\overline{R}$ is an
integral domain of characteristic $p $. Hence
\begin{equation*}
\ker ( \lambda ) =L^{-1}( \mathbb{F}_{p}) =\ker (
L) \cdot \langle c\rangle . 
\end{equation*}
\end{proof}

\begin{proposition}
For brevity we write $T_{1}$ and $T_{2}$ for $T_{R,1}$ and $T_{R,2}  $ which 
were defined after Lemma 3.12, and let $\overline{T}_{i}$ denote $T_{R,i}%
\mod {p}.$ Again we assume that $c$ is a central element of order $p$
which is not necessarily a commutator of $G$; then

(a)
\begin{equation}
\mathcal{H}\cong \overline{\mathcal{J}}/\overline{\mathcal{I}}\cong \lambda
( \overline{\mathcal{J}}) /\lambda ( \overline{\mathcal{I}}%
) ;
\end{equation}%
(b)
\begin{equation*}
\lambda ( \overline{\mathcal{J}}) =\phi \{ \Ima(
1-\Psi ) \overline{R}[ G] \cap ( \overline{T}_{1}+%
\overline{T}_{2}) \} ;
\end{equation*}%
(c)%
\begin{equation}
\lambda ( \overline{\mathcal{J}}) =\Ima( 1-\Psi )
\overline{R}[ C_{G}] \cap \widetilde{R}[ D_{G}] .
\end{equation}
\end{proposition}

\begin{proof} Using the identity%
\begin{equation*}
1+( 1-c) ( g-hgh^{-1}) \equiv [ 1+(
1-c) g,h] \mod( 1-c) ^{2}
\end{equation*}
we see that $1+( 1-c) T_{\overline{R},2} \subset \overline{\mathcal{I}}\ $.

We first prove (a). By (3.16) it will suffice to show
\begin{equation*}
\ker ( \lambda ) \cap \overline{\mathcal{J}}=\ker ( \lambda
) \cap \overline{\mathcal{I}}
\end{equation*}%
so that $\lambda ~$induces the required second isomorphism (a). We argue by
cases according as $c$ lies in $[ G,G] \ $or not.\ Firstly, if$\
c\in $ $[ G,G] ,\ $\ then obviously $c\in \overline{\mathcal{I}}$
and so  by Lemma 3.20 and the above $\ker ( \lambda )\subset \overline{\mathcal{I}}$.
Secondly,
if\ $c\notin $ $[ G,G] ,$ then $\mathrm{Det}( c) \neq 1;$ and so $%
c\notin \mathcal{J}$; and therefore, as is easily seen by projecting onto $G^{ab}$, it follows that  $%
c\notin \overline{\mathcal{J}}$; therefore by Lemma 3.20 and the above we are done.

We now prove (b).  Recall that in Lemma 3.12, $D_{G}$ was defined as the
subset of conjugacy classes in\ $C_{G}$ which are fixed under multiplication
by $c$.\ Consider $1+u( 1-c) \in \mathcal{J}$ \ as in Proposition
3.13 and, using the fact that $\Psi ( 1-c) =1-c^{p}=1,\ $we know
that%
\begin{equation*}
0=\nu \circ \mathrm{Det}( 1+u( 1-c) ) =p\log (
1+u( 1-c) )
\end{equation*}%
with
\begin{equation*}
\log ( 1+u( 1-c) ) \equiv ( 1-c) ( \xi
_{1}+\xi _{2})
\end{equation*}%
for $\xi _{i}\in T_{R,i};$ while by\ Lemma 3.19 we know%
\begin{equation*}
\log ( 1+u( 1-c) ) \equiv ( 1-c) (
u-u^{p}) \mod( 1-c) ^{2};
\end{equation*}%
hence we have now shown:%
\begin{equation*}
\lambda ( \overline{\mathcal{J}}) \subset \phi \{ (
1-\Psi ) \overline{R}[ G] \cap ( \overline{T}_{1}+%
\overline{T}_{2}) \} .
\end{equation*}%
To complete the proof of (b) it remains to prove the opposite inclusion $%
\supseteq $. For an element $x\in R[ G] $ we again write $%
\overline{x}$ for $x\mod{p}$. Suppose now we are given$\ \ u\in R[
G] ,\ n\in T_{1}+T_{2}$ with%
\begin{equation*}
\overline{u}-\Psi ( \overline{u}) =\overline{n}\in \Ima%
( 1-\Psi ) \overline{R}[ G] \cap ( \overline{T}%
_{1}+\overline{T}_{2}) .
\end{equation*}%
Writing $u-\Psi ( u) =n+pm$ with $m\in R[ G] ,\ $it
follows from Lemma 3.18 that
\begin{eqnarray*}
\phi \circ \log ( 1+( 1-c) u) &=&\phi \circ \lbrack
(1-c)( u-u^{p}) +( 1-c) ^{2}e^{\prime }] \\
&\equiv &\phi \circ \lbrack (1-c)( u-\Psi ( u) )
+( 1-c) ^{2}e^{\prime }]\mod( 1-c) p \\
&\equiv &\phi \circ \lbrack (1-c)n+( 1-c) ^{2}e^{\prime }]\mod ( 1-c) p
\end{eqnarray*}%
and so $\phi \circ \log ( 1+( 1-c) u) =\phi \circ
\lbrack (1-c)n+( 1-c) ^{2}e]$ for some $e\in R[ G] .$
As we have seen previously, $\mathcal{L}$ and $p\phi \circ \log \ $\
coincide on $1+( 1-c) R[ G] \ ($because$\Psi (
( 1-c) ) =0),$ and so by Lemma 3.5 we know that we can find
$f\in R[ G] $ such that
\begin{equation*}
\phi \circ \log ( 1+( 1-c) ^{2}f) =\phi ( (
1-c) ^{2}e) .
\end{equation*}%
If we set $x=( 1+( 1-c) u) ( 1+( 1-c)
^{2}f) ^{-1};$ then, firstly, we note%
\begin{equation*}
\phi \circ \log ( x) =\phi ( ( 1-c) n)
\end{equation*}%
and so $\lambda ( x) =\phi ( n) \mod(
1-c) .$ Secondly, as in Lemma 3.12, we know that as $n\in
T_{1}+T_{2},$ it follows that $\phi ( ( 1-c) n) =0,\ $%
and so we deduce that
\begin{equation*}
\phi \circ \log ( x) =\phi ( ( 1-c) n) =0.
\end{equation*}%
We claim that we have now shown that $\mathrm{Det}( x) =1\ $so
that $x\in \mathcal{J}$, as required. Again we argue by cases. Firstly, if $%
\ c$ is a commutator, then we note that we have seen previously in the
proof of Lemma 3.11 that
\begin{equation*}
\ker ( \phi \circ \log ) =\ker ( \mathrm{Det}) \ \
\text{on\ \ }1+( 1-c) R[ G] .
\end{equation*}%
and so $\mathrm{Det}( x) =1$. Secondly, if $c$ is not a
commutator, then $D_{G}$ is empty and therefore $T_{1}=(0)$ and so $n\in
( 1-c) T_{2}$ and again using the identity%
\begin{equation*}
( 1+(1-c)g) h^{-1}( 1+(1-c)g) ^{-1}h\equiv 1+(
1-c) ( g-g^{h}) \mod( 1-c) ^{2}
\end{equation*}%
we easily see there is a commutator $y\in 1+( 1-c) R[ G%
] $ with the property the%
\begin{equation*}
\log ( y) \equiv ( 1-c) n\mod( 1-c)
^{2}
\end{equation*}%
and so $( 1-c) n$ $\mod( 1-c) ^{2}\in \lambda (%
\mathcal{I)\ }\subset \lambda ( \mathcal{J}) ,$ as required.

Finally, in order to prove part (c), we must show that the right-hand
expressions in (b) and (c) coincide. The inclusion of the right-hand side of
(b) in the right-hand side of (c) is obvious. Conversely, given $x\in (
1-\Psi ) \overline{R}[ C_{G}] \cap \overline{R}[ D_{G}%
] $ we may write it as
\begin{equation*}
x=( 1-\Psi ) \sum\limits_{\gamma \in C_{G}}m_{\gamma }\gamma
=\sum\limits_{\delta \in D_{G}}n_{\delta }\delta
\end{equation*}%
with $m_{\gamma },n_{\delta }\in \overline{R}.$ We then choose group
elements $\gamma ^{\prime },\delta ^{\prime }\in G$ with the property
that $\phi ( \gamma ^{\prime }) =\gamma ,\ \phi ( \delta
^{\prime }) =\delta ,$ and observe that if we define $y$ as
\begin{equation*}
y=( 1-\Psi ) \sum m_{\gamma }\gamma ^{\prime }=\sum n_{\delta
}\delta ^{\prime }+\sum\limits_{g,h}l_{g,h}( g-hgh^{-1}) ,
\end{equation*}
for some $l_{g,h}\in \overline{R}$, then $y\in ( 1-\Psi )
\overline{R}[ G] \cap ( \overline{T}_{1}+\overline{T}%
_{2}) ;$ and, so by construction, we have $\phi ( y) =x$.
\end{proof} 
\medskip

\begin{corollary}
If $  c$ is not a commutator, $\lambda ( \overline{\mathcal{J}})
=0,$ and so in this case the map $  \mathrm{SK}_{1}( R[ G]
) \rightarrow \mathrm{SK}_{1}( R[ \overline{G}]
) $ is injective. (c.f. Proposition 7 in [O1].)
\end{corollary}

\begin{proof}  If $c$ is not a commutator, then $D_{G}$ is empty and so by Prop.
3.21 (c) $\mathcal{H}$ is trivial. By the very definition of $\mathcal{H}$,
the map $\mathrm{SK}_{1}( R[ G] ) \rightarrow \mathrm{%
SK}_{1}( R[\overline{G}])$ is therefore injective. 
\end{proof} 

\bigskip

% ----------------------------------------------------------------------
\subsection{The logarithmic image of $\mathcal{I}$}

We begin with the following straightforward generalization of Lemma 8 in
[O1]:

\begin{lemma}
If $c$ is a commutator, so that we can write $c=[ g,h] $ for $%
g,h\in G$, then for any $r\in \overline{R}=R\mod{p}:$

(i) \ $r( g-g^{k}) \in \lambda ( \overline{\mathcal{I}})
$ for any integer $k$ prime to $p;$

(ii) \ $( r-r^{p}) g\in \lambda ( \overline{\mathcal{I}}%
) ;$

(iii) \ for $g,h$ as given with $h\in \mathcal{S}_{G}$, $r( g-h)
\in \lambda ( \overline{\mathcal{I}}) .$
\end{lemma}

\begin{proof} First we note that by restricting to the subgroup generated by $g,h$
 we may assume that $\overline{G}=G/\langle c\rangle $ is
abelian; note that in this case the set $\mathcal{S=S}_{G}$ contains no $p$-th powers.

Let $\Lambda $ denote the $\mathbb{F}_{p}$-vector space
\begin{equation}
\Lambda =\lambda ( \overline{\mathcal{I}})
+\sum\nolimits_{k\notin \mathcal{S}}\overline{R}\phi ( k)
\end{equation}%
and note that of course this is an internal direct sum, as $\lambda (
\overline{\mathcal{I}}) \subset \lambda ( \overline{\mathcal{J}}%
) \subset \overline{R}[ D_{G}] $.

In order to prove the lemma we claim that it suffices to show that the three
types of term, stated in the lemma, all have trivial projection into the
right-hand factor of the direct sum (3.23). We demonstrate this is indeed
the case by showing that each of the three terms lies in both $\Lambda $ and
$R[ D_{G}] .  $

By hypothesis $c=[g,h],$ and so $c=ghg^{-1}h^{-1}$ and hence $%
c^{-1}=hgh^{-1}g^{-1}$ and therefore
\begin{equation*}
c^{-1}g=\ ^{h}g\text{ and }c^{-k}g^{k}=\ ^{h}g^{k};
\end{equation*}%
therefore, for $k$ coprime to $p,$ we know that $g^{k}\in \mathcal{S}.\ $%
This then shows that the terms listed in (i) and (ii) lie in $\overline{R}%
[ D_{G}] ,$ and then same is true for the terms in (iii) since we
are given $h,g\in \mathcal{S}$.

To conclude we must now show that each of the three types of element listed
in the lemma belongs to $\Lambda .$ For an element $g\in G$ we write $%
\overline{g}$ for the image of $g$ in $\overline{G}.$ (Note that this is
different to the notational convention used by Oliver in [O1].) Thus, as
above, we know that the subgroup of $G$ generated by $g^{p}$, $h^{p}$ is
abelian and central and as such do not lie in $  \Lambda .$ More generally as
in Lemma 3.19 we know that for $\sum_{g}r_{g}g\in R[ G] $ we have
\begin{equation*}
\phi ( ( \sum_{g}r_{g}g) ^{p}) \equiv \phi (
\sum_{g}r_{g}g^{p}) \mod pR[ C_{G}]
\end{equation*}%
and so we have shown more generally that the image under $\phi $ of the $p$%
th power of an element of $\overline{R}[ G] $  has trivial
intersection with $\Lambda .$

Proof of (iii): From identity (3.14) we know that for any $n>0,\ \mu \in R$
\begin{equation}
[ g^{-1}h,1-\mu ( g-h) ^{n}] \equiv 1+n(
1-c) \{\mu ( g-h) ^{n}+\mu ^{2}( g-h)
^{2n}+\cdots \}\mod( 1-c) ^{2}.
\end{equation}%
As we have seen above, we may drop $p$th~powers and so consider only those $%
n $ coprime to $p$ and thereby get that
\begin{equation*}
\{\mu ( g-h) ^{n}+\mu ^{2}( g-h) ^{2n}+\cdots \}\in
\Lambda .
\end{equation*}%
Next, since $( g-h) ^{N}=0$ in $\overline{R}[ G] $ for
$N>>0,$ we may argue by downwards induction to get that $\mu (
g-h) ^{n}\in \Lambda $ for all $n>0$  which are coprime to $p.$

Proof of (i): As above we know that for $k$ coprime to $p$ we can find an
integer $m$ so that we have the equality
\begin{equation}
[ h^{m},g^{p+k}] =[ h^{m},g^{k}] =c.
\end{equation}%
Using part (iii) above, we know that for any $k>0$ we have%
\begin{equation}
\mu g^{k}( 1-g^{p}) =\mu ( g^{k}-h^{m}) -\mu (
g^{p+k}-h^{m}) \in \Lambda
\end{equation}%
and this implies that for any $r\geq p$ we have
\begin{equation}
\mu g( 1-g) ^{r}=\mu g( 1-g) ^{r-p}( 1-g)
^{p}=\mu g( 1-g) ^{r-p}( 1-g^{p}) \in \Lambda .
\end{equation}%
So for any $r>1$%
\begin{eqnarray*}
[ h,1+\mu ( 1-g) ^{r}] &=&( 1+\mu h(
1-g) ^{r}h^{-1}) ( 1+\mu ( 1-g) ^{r}) ^{-1}
\\
&=&( 1+\mu ( 1-cg) ^{r}) ( 1+\mu (
1-g) ^{r}) ^{-1} \\
&=&( 1+\mu ( 1-g) ^{r}-\mu ( 1-g) ^{r}+\mu (
1-cg) ^{r}) ( 1+\mu ( 1-g) ^{r}) ^{-1} \\
&=&1+( \mu ( 1-cg) ^{r}-\mu ( 1-g) ^{r})
( 1+\mu ( 1-g) ^{r}) ^{-1}.
\end{eqnarray*}%
Writing
\begin{equation*}
( 1-cg) ^{r}=( ( 1-g) +g( 1-c) )
^{r}=\sum\nolimits_{i=0}^{r}(
\begin{array}{c}
r \\
i%
\end{array}%
) ( 1-g) ^{r-i}g^{i}( 1-c) ^{i}
\end{equation*}%
we get
\begin{equation*}
[ h,1+\mu ( 1-g) ^{r}] \equiv 1+\mu r( 1-c)
( 1-g) ^{r-1}g[ 1-\mu ( 1-g) ^{r}] ^{-1}%
\mod( 1-c) ^{2}.
\end{equation*}%
Omitting $p$-th powers, as previously, we suppose that $r $ is coprime to $p$ and get that $\Lambda $ contains
\begin{equation}
\mu g( 1-g) ^{r-1}-\mu ^{2}g( 1-g) ^{2r-1}+\mu
^{3}g( 1-g) ^{3r-1}+\cdots  \in \Lambda .
\end{equation}%
Now by (3.26) above we already know that for $ir-1\geq p$ we have $\mu
^{i}g( 1-g) ^{ir-1}\in \Lambda ,\ $and so, again by downwards
induction, we get%
\begin{equation}
\mu rg( 1-g) ^{r-1}=\mu rg( 1-g) ^{r-2}(
1-g) \in \Lambda
\end{equation}%
for all $2\leq r\leq p-1$ and hence
\begin{equation*}
\mu g\equiv \mu g^{2}\equiv\cdots \equiv \mu g^{p-1}\mod\Lambda .
\end{equation*}%
The result then follows by (3.26) and so $\mu ( g-g^{k}) \in
\Lambda $ for any $k$ which is coprime to $p.$

Proof of (ii):\ When we take $r=1$ in (3.28) above  and the line that
follows we get
\begin{equation*}
\mu g-\mu ^{2}g( 1-g) +\mu ^{3}g( 1-g) ^{2}+\cdots +\mu
^{p}g( 1-g) ^{p-1}\in \Lambda
\end{equation*}%
Applying (3.27) we deduce that $\Lambda $ contains
\begin{eqnarray*}
\mu g+\mu ^{p}g( 1-g) ^{p-1} &\equiv &\mu g+\mu ^{p}(
g+g^{2}+\cdots +g^{p-1}) \\
&=&\mu g-\mu ^{p}g+\mu ^{p}( g^{2}-g) +\cdots +(
g^{p-1}-g) .
\end{eqnarray*}
By part (i) we know $\mu ^{p}(g-g^{k})\in \Lambda $ for each $1\leq k\leq
p-1 $ and so $( \mu -\mu ^{p}) g\in \Lambda$, as required.
\end{proof} 

The following is a   straightforward generalization of Theorem 1 in
[O1]:

\begin{theorem}
With the above notation $\mathcal{H}( R[ G] ) =\ker
( \mathrm{SK}_{1}( R[ G] ) \rightarrow \mathrm{SK}%
_{1}( R[ \overline{G}] ) ) $ is generated by the
elements $\exp ( r( 1-c) ( g-g^{\prime }) )
$ for elements $g,g^{\prime }\in S_{G}.$
\end{theorem}

The following follows easily from the above theorem by arguing by induction
of the order of the group $G$ (see Proposition 9 in [O1] for details):

\begin{proposition}
If $A$ a normal abelian subgroup of $G$ with the property that $G/A$ is
cyclic then $\mathrm{SK}_{1}( R[ G] ) =\{1\} $.
\end{proposition}

% ----------------------------------------------------------------------
\subsection{Oliver's map $\Theta _{R[ G] }$}\label{ss3c}

Recall that $\mathrm{Wh}( R[ G] ) $ and $\mathrm{Wh}%
^{\prime }( R[ G] ) $ were defined in 3.1 and were
shown to sit in exact sequences (see (3.4)) and (3.5):%
\begin{equation}
1\rightarrow \mathrm{SK}_{1}( R[ G] ) \rightarrow
\mathrm{Wh}( R[ G] ) \rightarrow \mathrm{Wh}^{\prime
}( R[ G] ) \rightarrow 1
\end{equation}
\begin{equation}
1\rightarrow \mathrm{SK}_{1}( R[ G] ) \rightarrow
\mathrm{Wh}( R[ G] )  {\xrightarrow
{\Gamma _{G}}}\phi ( I_{G})  {\xrightarrow{\omega _{G}} }R_{F}\otimes
G^{ab}\rightarrow 1
\end{equation}%
where as previously $R_{F}=R/( 1-F) R.$

Suppose now that we have an exact sequence of finite $p$-groups
\begin{equation}
1\rightarrow H  \xrightarrow{\iota }  \widetilde{G}\xrightarrow{\alpha}
 G\rightarrow 1.
\end{equation}%
Our study of $\mathrm{SK}_{1}$ will be based on the detailed study of the
two groups
\begin{eqnarray}
\mathcal{K} &=&\ker ( \alpha _{\ast }) =\ker ( \alpha _{\ast
}:\mathrm{SK}_{1}( R[ H] ) \rightarrow \mathrm{SK}%
_{1}( R[ \widetilde{G}] ) ) , \\
\mathcal{C} &=&\text{\textrm{coker}}( \alpha _{\ast }) =\text{%
\textrm{coker}}( \alpha _{\ast }:\mathrm{SK}_{1}( R[
\widetilde{G}] ) \rightarrow \mathrm{SK}_{1}( R[ G%
] ) ) .
\end{eqnarray}%
The above exact sequence (3.31) affords a commutative diagram with exact
rows:%
\begin{equation}
\begin{array}{ccccccccccc}
1 & \rightarrow & \mathrm{SK}_{1}( R[ H] ) &
\rightarrow & \mathrm{Wh}( R[ H] ) & \xrightarrow{\Gamma
_{H}}  & \phi ( I_{H}) & \xrightarrow{\omega _{H}}  & R_{F}\otimes H^{\rm ab} & \rightarrow & 1 \\
&  & \downarrow &  & \downarrow &  & \downarrow \phi ( \iota ) &
& \downarrow 1\otimes \iota ^{\rm ab} &  &  \\
1 & \rightarrow & \mathrm{SK}_{1}( R[ \widetilde{G}] )
& \rightarrow & \mathrm{Wh}( R[ \widetilde{G}] ) &
\xrightarrow{\Gamma _{\widetilde{G}}}  & \phi ( I_{\widetilde{G%
}}) & \xrightarrow{\omega _{\widetilde{G}}} & R_{F}\otimes
\widetilde{G}^{\rm ab} & \rightarrow & 1 \\
&  & \downarrow \alpha _{\ast } &  & \ \ \ \ \downarrow \mathrm{Wh}( \alpha
) &  & \downarrow &  & \downarrow 1\otimes \pi ^{\rm ab} &  &  \\
1 & \rightarrow & \mathrm{SK}_{1}( R[ G] ) &
\rightarrow & \mathrm{Wh}( R[ G] ) & \xrightarrow{\Gamma
_{G}}  & \phi ( I_{G}) & \xrightarrow{\omega _{G}} & R_{F}\otimes G^{\rm ab} & \rightarrow & 1.%
\end{array}%
\end{equation}%
Reasoning as in the snake lemma, we obtain a map
\begin{equation}
\Delta =``\mathrm{Wh}( \alpha ) \circ \Gamma _{\widetilde{G}%
}^{-1}\circ \phi ( \iota ) \circ \omega _{H}^{-1}":\ker (1\otimes
\iota ^{ab})\rightarrow \text{ }\mathcal{C}\text{.}
\end{equation}%
Next we consider the following two subgroups of $\widetilde{G}:$%
\begin{eqnarray*}
H_{0} &=&H\cap [ \widetilde{G},\widetilde{G}] \\
H_{1} &=&\langle h\in H\mid h=[ \widetilde{g}_{1},\widetilde{g}_{2}%
] \text{ for }\widetilde{g}_{1},\widetilde{g}_{2}\in \widetilde{G}%
\text{ }\rangle
\end{eqnarray*}%
that is to say $  H_{1}$ is the subgroup of $H$ which is generated by $%
\widetilde{G}$ commutators  which lie in $H$; so that obviously $[ H,H%
] \subset H_{1}\subset H_{0}$.  We use the exact sequence%
\begin{equation*}
1\rightarrow H_{0}/[ H,H] \rightarrow H^{\mathrm{ab}}\xrightarrow{i^{ 
\mathrm{ab}}} G^{\mathrm{ab}}
\end{equation*}%
to obtain the composite map $\widetilde{\kappa }_{\alpha }$
\begin{equation*}
\widetilde{\kappa }_{\alpha }:R_{F}\otimes H_{0}\twoheadrightarrow \ker
( 1\otimes \iota ^{\rm ab}) \xrightarrow{\Delta }  
\mathcal{C}
\end{equation*} 
  noting that the left-hand map is surjective since $R_{F}$ is
torsion free.

\begin{proposition}
\bigskip The map $\widetilde{\kappa }_{\alpha }~$induces an isomorphism,
denoted $\kappa _{\alpha },$
\begin{equation*}
\kappa _{\alpha }:\frac{R}{( 1-F) R}\otimes \frac{H_{0}}{H_{1}}%
\xrightarrow { \ \ }\mathcal{C}
\end{equation*}%
which is natural with respect to maps between group extensions.
\end{proposition}

\begin{proof} See Proposition 16 in [O1]. Consider the maps
\begin{equation}
\mathrm{SK}_{1}( R[ \widetilde{G}] ) \xrightarrow{\alpha
_{1\ast }}\mathrm{SK}_{1}( R[ \widetilde{G}/H_{1}%
] ) \xrightarrow{\alpha _{2\ast }} \mathrm{SK}%
_{1}( R[ G] ) .
\end{equation}%
From the proof of Proposition 3.9, because $H_{1}$ is generated by
commutators lying in $H$, we know that $\alpha _{1\ast }$ is surjective. So
we may reduce consideration to to the case where $H$ is replaced by $%
H/H_{1}.\ $By repeated use of the argument used to show Corollary 3.22, we
see that $\alpha _{2\ast }$ is injective, because $H/H_{1}$ contains no
commutators.

We now prove the result by taking the case when $H=\langle
c\rangle $ is of order $p$ and contains no commutators. First we
recall from Lemma 3.7 that%
\begin{equation}
\frac{p\phi ( ( 1-c) R[ \widetilde{G}] ) }{%
\mathcal{L}( 1+( 1-c) R[ \widetilde{G}] ) }%
\cong \frac{R}{( 1-F) R+pR}
\end{equation}%
induced by $\phi ( ( 1-c) rg) \mapsto r\mod {p}$. Using the map
\begin{equation*}
\omega _{\widetilde{G}}:\frac{1+I_{R[ \widetilde{G}] }}{1+I_{R%
[ \widetilde{G}] }^{2}}\cong \frac{I_{R[ \widetilde{G}]
}}{I_{R[ \widetilde{G}] }^{2}}\cong R\otimes \widetilde{G}^{%
\mathrm{ab}}
\end{equation*}%
the above isomorphism (3.38) may be recast as
\begin{equation}
\tau _{\widetilde{G}}:\frac{p\phi ( ( 1-c) R[
\widetilde{G}] ) }{\mathcal{L}( 1+( 1-c) R[
\widetilde{G}] ) }\cong R_{F}\otimes \langle c\rangle
\end{equation}%
We know that
\begin{equation*}
\Gamma _{\widetilde{G}}( \ker \mathrm{Wh}( \alpha ) )
=\nu _{\widetilde{G}}( \mathrm{Det}( 1+( 1-c) R[
\widetilde{G}] ) )
\end{equation*}%
so that we obtain the exact sequence
\begin{equation}
\ker \mathrm{Wh}( \alpha ) \xrightarrow{\Gamma _{\widetilde{G}}}p( 1-c) \phi ( R[ \widetilde{G}]
) \xrightarrow{\tau _{\widetilde{G}}}R_{F}\otimes
\langle c\rangle \rightarrow 0.
\end{equation}%
We use the exact sequence (3.30) to form the diagram with exact rows:%
\begin{equation}
\begin{array}{ccccccccc}
1 & \rightarrow & \mathrm{SK}_{1}( R[ \widetilde{G}] )
& \rightarrow & \mathrm{Wh}( R[ \widetilde{G}] ) &
\rightarrow & \mathrm{Wh}^{\prime }( R[ \widetilde{G}]
) & \rightarrow & 1 \\
&  & \downarrow \alpha _{\ast } &  & \downarrow \mathrm{Wh}( \alpha
) &  & \downarrow \mathrm{Wh}^{\prime }( \alpha ) &  &  \\
1 & \rightarrow & \mathrm{SK}_{1}( R[ G] ) &
\rightarrow & \mathrm{Wh}( R[ G] ) & \rightarrow &
\mathrm{Wh}^{\prime }( R[ G] ) & \rightarrow & 1.%
\end{array}%
\end{equation}%
We consider the further diagram with exact rows:%
\begin{equation}
\begin{array}{ccccccccc}
&  & \ker \mathrm{Wh}( \alpha ) & \xrightarrow{\Gamma _{\widetilde{G}}%
}& ( 1-c) \phi ( R[ \widetilde{G}]
) & \xrightarrow{\tau _{\widetilde{G}}}& R_{F}\otimes
\langle c\rangle & \rightarrow & 0 \\
&  & \downarrow \beta &  & \downarrow p &  & \downarrow &  &  \\
0 &  \rightarrow & \ker \mathrm{Wh}^{\prime }( \alpha ) & \xrightarrow{%
\varpi }& p( 1-c) \phi ( R[ \widetilde{G}%
] ) & \xrightarrow{\omega _{\widetilde{G}}} & \ker
( R_{F}\otimes \widetilde{G}^{\rm ab}\rightarrow R_{F}\otimes G^{\rm ab})
& \rightarrow & 0 \\
&  & \downarrow &  & \downarrow &  &  &  &  \\
&  & {\rm coker}(\alpha _{\ast })&  & 0 &  &  &  &  \\
&  & \downarrow &  &  &  &  &  &  \\
&  & 0 &  &  &  &  &  &
\end{array}%
\end{equation}%
where $\tau _{\widetilde{G}}$ is as in (3.39) above \ and $\omega _{%
\widetilde{G}}$ is as in 3.3; here the top row is exact from (3.40) above;
the middle horizontal row is exact by applying the snake lemma to the diagram%
\begin{equation}
\begin{array}{ccccccccc}
1 & \rightarrow & \mathrm{Wh}^{\prime }( R[ \widetilde{G}]
) & \rightarrow & \phi ( I_{\widetilde{G}}) & \rightarrow &
R_{F}\otimes \widetilde{G}^{\rm ab} & \rightarrow & 1 \\
&  & \downarrow &  & \downarrow &  & \downarrow &  &  \\
1 & \rightarrow & \mathrm{Wh}^{\prime }( R[ G] ) &
\rightarrow & \phi ( I_{G}) & \rightarrow & R_{F}\otimes G^{\rm ab} &
\rightarrow & 1%
\end{array}%
\end{equation}%
the left-hand vertical column is exact by applying the snake lemma to the
diagram (3.41) above.

Recall that by hypothesis $H=\langle c\rangle $  contains no
commutators, so that $H_{1}=(1)$; we now consider separately: Case 1, when $%
H\nsubseteq [ G,G] ;$ and Case 2, when $H\subset [ G,G]$. In all cases, since $R_{F}$ is torsion free we may identify
\begin{equation*}
\ker ( R_{F}\otimes \widetilde{G}^{\mathrm{ab}}\rightarrow R_{F}\otimes
G^{\mathrm{ab}}) =R_{F}\otimes \ker ( \widetilde{G}^{\mathrm{ab}%
}\rightarrow G^{\mathrm{ab}}) =R_{F}\otimes H/H_{0}.
\end{equation*}

\textbf{Case 1.} Then $H_{0}=H_{1}=(1);$ in this case $H=\ker (
\widetilde{G}^{ab}\rightarrow G^{ab}) \ $and using diagram (3.42) we
see that $\ker \tau _{\widetilde{G}}=\ker \omega _{\widetilde{G}};$ hence $%
\beta $ is onto and we have shown:%
\begin{equation*}
H_{0}/H_{1}=(1)=\text{\textrm{coker}}( \alpha _{\ast }) =\text{%
\textrm{coker}}( \alpha _{\ast }:\mathrm{SK}_{1}( R[
\widetilde{G}] ) \rightarrow \mathrm{SK}_{1}( R[ G%
] ) ) .
\end{equation*}%

\textbf{Case 2.} Then $H=H_{0}=\langle c\rangle $ and again $%
H_{1}=(1)$; in this case we have $\widetilde{G}^{\mathrm{ab}}\cong G^{%
\mathrm{ab}}$ and so by diagram (3.42)
\begin{equation*}
\text{\textrm{coker}}( \alpha _{\ast }:\SK_{1}( R[ \widetilde{G%
}] ) \rightarrow \SK_{1}( R[ G] ) ) =%
\text{\textrm{coker}}( \alpha _{\ast }) =R_{F}\otimes
\langle c\rangle . 
\end{equation*}
\end{proof}
\medskip

If we write $G=F/\mathcal{R},$ with $F $ a free group, then by Hopf's
theorem (see Chapter II, Theorem 5.3 and also Exercise 6 in Chapter II of
[B]), we have an isomorphism
\begin{equation*}
H_{2}( G,\mathbb{Z}) =\frac{\mathcal{R}\cap [ F,F] }{%
[ \mathcal{R},F] }.
\end{equation*}%
Given  a surjection $\theta :F\rightarrow \widetilde{G}$ and setting $%
\mathcal{R}=\ker ( F\xrightarrow{\theta }\widetilde{G}\xrightarrow
{\alpha }G) ,$ so that $\theta ( \mathcal{R})
=H,$ and we have a surjective map
\begin{equation*}
\delta ^{\alpha }:H_{2}( G,\mathbb{Z}) \rightarrow H\cap [
\widetilde{G},\widetilde{G}] /[ H,G] \twoheadrightarrow
H_{0}/H_{1}.
\end{equation*}

From Lemma 17 in [O1] we have the following entirely group theoretic result
(in which the ring $R$ plays no role):

\begin{lemma}
For any $p$-group $G$ there are central extensions of $p$-groups
\begin{equation*}
1\rightarrow H_{1}\rightarrow \widetilde{G}_{1}\xrightarrow{\alpha _{1}}G\rightarrow 1, \text{ and } \ 1\rightarrow H_{2}\rightarrow
\widetilde{G}_{2}\xrightarrow{\alpha _{2}}G\rightarrow 1
\end{equation*}%
such that: (i) $\delta ^{\alpha _{1}}$ is an isomorphism; and (ii) any
automorphism of $G$ lifts to an automorphism of $\widetilde{G}_{2},$ and for
any normal subgroup $K$ of $G$, if we write $\widetilde{K}=\alpha
_{2}^{-1}( K)$, then the sub-extension
\begin{equation*}
1\rightarrow H_{2}\rightarrow \widetilde{K}\xrightarrow{\alpha }
K\rightarrow 1
\end{equation*}%
has the property that $\delta ^{\alpha }$ is an injection.
\end{lemma}

In the Introduction we defined the subgroup $H_{2}^{ab}( G,\mathbb{Z}%
) $ of $H_{2}( G,\mathbb{Z}) .$ Since for an abelian group $%
A$ we know that $H_{2}( A,\mathbb{Z}) =A\wedge A$ we see easily
that $\delta ^{\alpha }( H_{2}^{ab}( G,\mathbb{Z}) ) $
is equal to the subgroup $H$ generated by $ h\in H$ where $h$ is a
commutator in $\widetilde{G}.$ Note that for any two elements of $\widetilde{%
G}$ whose commutator lies in $H $, their images in $G$ commute and so lie in
an abelian subgroup of $G$.

We now suppose that
\begin{equation*}
1\rightarrow H\rightarrow \widetilde{G}\xrightarrow{\alpha }
G_{1}\rightarrow 1
\end{equation*}%
is an \textbf{arbitrary} extension of finite $p$-groups, and we again define
$H_{0}$ and $H_{1}\ $as above, so that $H_{0}/H_{1}$ is central in $%
\widetilde{G}/H_{1}.$ Using Proposition 3.26 and assembling together the
above we have maps%
\begin{eqnarray*}
\mathrm{SK}_{1}( R[ G] ) &\twoheadrightarrow &\text{%
\textrm{coker}}( \mathrm{SK}_{1}( R[ \widetilde{G}]
) \rightarrow \mathrm{SK}_{1}( R[ G] ) ) \\
&&\overset{\kappa _{\alpha }}{\cong }R_{F}\otimes H_{0}/H_{1}\xrightarrow{%
1\otimes \delta ^{\alpha }}R_{F}\otimes \overline{H}_{2}(
G,\mathbb{Z})
\end{eqnarray*}%
with $1\otimes \delta ^{\alpha }$ surjective as above. Recall here that, as
in the Introduction, by definition, $\overline{H}_{2}( G,\mathbb{Z})
=H_{2}( G,\mathbb{Z}) /H_{2}^{\rm ab}( G,\mathbb{Z}) .$

For any such group extension where $\delta ^{\alpha }$ is an injection (for
instance when the group extension satisfies condition (i) of Lemma 3.27), $%
\delta ^{\alpha }$ is then obviously an isomorphism; and hence $1\otimes
\delta ^{\alpha }$ is an isomorphism. In these circumstances we denote the
composite map as
\begin{equation}
\Theta _{R[ G] }:\mathrm{SK}_{1}( R[ G] )
\rightarrow \frac{R}{( 1-F) R}\otimes \overline{H}_{2}( G,%
\mathbb{Z})
\end{equation}%
and we note that by the naturality of the maps involved, $\Theta _{R[ G%
] }$ is in fact independent of the particular choice of extension used
(where $\delta ^{\alpha }$ is an injection).

\begin{theorem}\label{thm43}
For any ring $R$ which satisfies the Standing Hypotheses the map
\begin{equation*}
\Theta _{R[ G] }:\mathrm{SK}_{1}( R[ G] )
\rightarrow \frac{R}{( 1-F) R}\otimes \overline{H}_{2}( G,%
\mathbb{Z}) .
\end{equation*}%
is an isomorphism.
\end{theorem}

The proof of Theorem 3.28 is exactly the same as the proof of Theorem 3 in
[O1] with $R$ now replacing $\mathbb{Z}_{p}.$ We highlight the remaining
steps in the proof in order to provide an overview for the reader's
convenience. \ As a first step towards proving the theorem, we note that
proof of Proposition 18 in [O1] now extends to give:

\begin{proposition}
For any ring $R$ satisfying the Standing Hypotheses, given a surjection of $%
p $-groups $\alpha :\widetilde{G}\rightarrow G,$ if $\Theta _{R[
\widetilde{G}] }$ is an isomorphism, then $\Theta _{R[ G] }$
is also an isomorphism.
\end{proposition}

We now quote the following two lemmas (see Lemmas 19 and 20 in [O1]):

\begin{lemma}
Suppose $H$ is a normal subgroup of the $p$-group $G$ with the property that
$G/H$ is cyclic; then the boundary map%
\begin{equation*}
\partial :\ker ( \mathrm{Wh}^{\prime }( R[ H] )
\rightarrow \mathrm{Wh}^{\prime }( R[ G] ) )
\rightarrow \mathrm{coker}( \mathrm{SK}_{1}( R[ H]
) \rightarrow \mathrm{SK}_{1}( R[ G] ) )
\end{equation*}%
is surjective.
\end{lemma}

\begin{lemma}
Suppose $H$ is a normal subgroup of the $p$-group $G$ with the property that
$G/H$ is cyclic; then
\begin{equation*}
\mathrm{coker}( \mathrm{SK}_{1}( R[ H] )
\rightarrow \mathrm{SK}_{1}( R[ G] ) ) \cong
\frac{R}{( 1-F) R}\otimes \mathrm{coker}( \overline{H}%
_{2}( H,\mathbb{Z}) \rightarrow \overline{H}_{2}( G,\mathbb{Z%
}) ) .
\end{equation*}%
In particular, if $\mathrm{SK}_{1}( R[ H] ) =\{1\},$
then $\Theta _{R[ G] }$ is an isomorphism.
\end{lemma}

In order to prove Theorem 3.28 we can then piece Proposition 3.29 and Lemmas
3.30 and 3.31 together in an entirely group theoretical way, by studying a
suitable category of central extensions of the group $G$. The details are
exactly the same as those for the proof of Theorem 3 in [O1].

% ----------------------------------------------------------------------
% ----------------------       END INSERT        -----------------------
% ----------------------------------------------------------------------

% ----------------------------------------------------------------------
% ----------------------------------------------------------------------
\section{Character action and reduction to elementary groups}\label{s4}
\setcounter{equation}{0}

In this section we do not need to assume that $R$ supports a lift of
Frobenius.

% ----------------------------------------------------------------------
\subsection{Character action on $\SK_{1}$}\label{ss4a}

Let $\G_{0}( \Z_{p}[G] ) $ denote the
Grothendieck group of finitely generated $\Z_{p}[G] $%
-modules and let $\G_{0}^{\Z_{p}}( \Z_{p}[G]
) $ denote the Grothendieck group of finitely generated $\Z_{p}%
[G] $-modules which are projective over $\Z_{p}$. From
38.42 and 39.9 in [CR2] we have:

\begin{proposition} We have
\begin{equation*}
\G_{0}^{\Z_{p}}( \Z_{p}[G] ) \xrightarrow{\simeq}
\G_{0}( \Z_{p}[G] )  \xrightarrow{\simeq}\G_{0}( \Q_{p}[G] )
\end{equation*}%
with the first isomorphism induced by the natural embedding of categories and the second
isomorphism induced by the extension of scalars map $\otimes _{\Z_{p}}\Q_{p}$ .
\end{proposition}

\begin{proposition}
Let $S$ denote an integral domain containing $\Z_{p}$. Then $\G_{0}^{%
\Z_{p}}( \Z_{p}[G] ) $ and hence, by
the previous proposition, $\G_{0}( \Q_{p}[G] )$, acts naturally on $\Kr_{1}( S[G] ) $ via the
following rule: for a $\Z_{p}[G] $-lattice $L$ and for
an element  of $\Kr_{1}( S[G] ) $ represented by a pair
$( P,\alpha ) $ (where $P$ is a projective $S[G] $-module and $\alpha $ is an $S[G] $-automorphism of $P$), then $L $ sends $( P,\alpha ) $ to the pair
\begin{equation*}
( L\otimes P, \ ( 1\otimes \alpha ) ) .
\end{equation*}%
Note also for future reference that the functor $ G\rightarrow
\Kr_{1}( S[G] ) $ is a Frobenius module for the
Frobenius functor $G\rightarrow \G_{0}^{\Z_{p}}( \Z_{p}%
[G] )  $ (see page 4 in [CR2] and also [L]).  Also,  $G\rightarrow
\SK_{1}( S[G] ) $ is a Frobenius submodule of $G\rightarrow
\Kr_{1}( S[G] ) $ and therefore the action of $%
\G_{0}( \Q_{p}[G] ) $ on $\Kr_{1}( S[G%
] ) $ induces an action on $\mathrm{Det}( \GL(S[G%
] )) $ (see Ullom's Theorem in 2.1 of [T], and see also below
for his explicit description of this action).
\end{proposition}

Proof. See Proposition 5.2 in [CPT1].

% ----------------------------------------------------------------------
\subsection{Brauer Induction}\label{ss4b}

Let $N^{c}$ denote our chosen algebraic closure of the field $N$ and, for a
positive integer $m$, $\mu _{m}$ denotes the group of roots  of unity of
order $m$ in $N^{c}$. We then identify ${\rm Gal}( N( \mu
_{m}) /N) $ as a subgroup of $(\Z/m\Z) ^{\times }$ in the usual way. Recall that a
semi-direct product of a cyclic group $C$ (of order $m$, say, which is
coprime to $p$) by a $p$-group $P$, $C\rtimes P$, is called $N$\textbf{-}$p$%
\textbf{-elementary} (see page 112 in [S]) if for given $\pi \in P$ there
exists
\begin{equation*}
t=t( \pi ) \in {\rm Gal}( N( \mu _{m})
/N) \subseteq (\Z/m\Z) ^{\times }
\end{equation*}%
such that for all $c\in C$, we have
$
\pi c\pi ^{-1}=c^{t}
$.
The direct product $C\times P$ is called a $p$%
\textbf{-elementary group}. We say that a group is $N$-\textbf{elementary} if it is
$N$\textbf{-}$p$%
\textbf{-elementary} for some $p$. We will denote by $\E_p(N)$
the class of $N$-$p$-elementary groups.

\begin{theorem}\label{thm49}
For a given field $N$ of characteristic zero and for a given finite group $G$ there exists an integer $l$ coprime to $p$ such that
\begin{equation*}
l\cdot \G_{0}( N[G] ) \subseteq \sum_{J}{\rm Ind}%
_{J}^{G}( G_{0}( N[J] ) )
\end{equation*}%
where $J$ ranges over the $N$-$p$-elementary subgroups of $G$.
\end{theorem}

\begin{proof} See Theorem 28 in [S].\end{proof}

% ----------------------------------------------------------------------
\subsection{Mackey functors and Green rings}\label{ss4c}

We now briefly introduce the notions of a Mackey functor, a Frobenius
functor and a Green ring (see 38.4 in [CR2] and [Bo] for details).

A Frobenius functor consists of a collection of commutative rings $F(
G)$, one for each finite group $G$ together with induction and
restriction maps of rings for each pair of groups $K\subset H$
\begin{equation*}
\mathrm{Ind}_{K}^{H}:F( K) \rightarrow F( H) ,\ \ \
\text{\textrm{Res}}_{K}^{H}:F( H) \rightarrow F( K)
\end{equation*}%
with the properties:

(1) for $K\subset H\subset J$ we have%
\begin{equation*}
\mathrm{Ind}_{H}^{J}\circ \mathrm{Ind}_{K}^{H}=\mathrm{Ind}_{K}^{J},\ \
\text{\textrm{Res}}_{K}^{H}\circ \text{\textrm{Res}}_{H}^{J}=\text{\textrm{%
Res}}_{K}^{J};
\end{equation*}

(2) for $x\in F( H)$, $y\in F( K)  $ we have
\begin{equation*}
x\cdot \mathrm{Ind}_{K}^{H}( y) =\mathrm{Ind}_{K}^{H}( \text{%
\textrm{Res}}_{K}^{H}( x) \cdot y) .
\end{equation*}

A Frobenius module $N$ over the Frobenius functor $F$ is  a collection of $F( G) $-modules $N( G)$, one for each finite group $%
G$, together with induction and restriction maps of rings for each pair of
groups $K\subset H$
\begin{equation*}
\mathrm{Ind}_{K}^{H}:N( K) \rightarrow N( H) ,\ \ \
\text{\textrm{Res}}_{K}^{H}:N( H) \rightarrow N( K)
\end{equation*}%
satisfying the transitivity properties as in (1) above together with the
following three properties:

(i) for $x\in F( H) ,\ m\in N( H) $%
\begin{equation*}
\text{\textrm{Res}}_{K}^{H}( x) \cdot \text{\textrm{Res}}%
_{K}^{H}( m) =\text{\textrm{Res}}_{K}^{H}( x\cdot m) ;
\end{equation*}

(ii) and for $m^{\prime }\in N( K) $%
\begin{equation*}
x\cdot \mathrm{Ind}_{K}^{H}( m^{\prime }) =\mathrm{Ind}_{K}^{H}(
\text{\textrm{Res}}_{K}^{H}( x) \cdot m^{\prime }) ;
\end{equation*}

(iii) and for $x^{\prime }\in F( K) $%
\begin{equation*}
\mathrm{Ind}_{K}^{H}( x^{\prime }) \cdot m=\mathrm{Ind}_{K}^{H}(
x^{\prime }\cdot \text{\textrm{Res}}_{K}^{H}( m) ) .
\end{equation*}

A Mackey functor with values in the category of abelian groups consists of a
collection of abelian groups $M( G)$, one for each finite group $
G $, together with induction and restriction maps for each pair of groups $K\subset H$
\begin{equation*}
\mathrm{Ind}_{K}^{H}: M( K) \rightarrow M( H) ,\quad
\text{\textrm{Res}}_{K}^{H}: M( H) \rightarrow M( K)
\end{equation*}
and conjugation maps $c_{x,H}:M( H) \rightarrow M(
^{x}H) $ for $x\in G$, $H\subset G$ where we write $^{x}H=xHx^{-1}$,
$H^{x}=x^{-1}Hx$. These maps are then required to satisfy the following
properties:

1. For $L\subset K\subset H$ we have $ {\rm Ind}_{K}^{H}\circ {\rm Ind}%
_{L}^{K}={\rm Ind}_{L}^{H}$ and ${\rm Res}_{L}^{K}\circ  {\rm Res}_{K}^{H}={\rm Res}_{L}^{H}$.

2. For $x,y\in G$ and $H\subset G$ we have $c_{y,^{x}H}\circ
c_{x,H}=c_{yx,H}$.

3. For $x\in G$ and $K\subset H\subset G$ we have
\begin{equation*}
c_{x,H}\circ {\rm Ind}_{K}^{H}={\rm Ind}_{^{x}K}^{^{x}H}\circ c_{x,K}%
\text{ \ and \ }c_{x,H}\circ {\rm Res}_{K}^{H}={\rm Res}
_{^{x}K}^{^{x}H}\circ c_{x,K}.
\end{equation*}

4. For $x\in H$ we have $c_{x,H}={\rm id}_{H}$.

5. (The Mackey axiom.) For $L\subset H\supset K$ we have
\begin{equation*}
\text{\textrm{Res}}_{L}^{H}\circ {\rm Ind}_{K}^{H}=\sum\nolimits_{x\in
L\backslash H/K}{\rm Ind}_{L\cap ^{x}K}^{L}\circ c_{x},_{L^{x}\cap
K}\circ\,  {\rm Res}_{L^{x}\cap K}^{K}.
\end{equation*}

A Green ring is a commutative ring-valued Mackey functor $\mathcal{R}$ which is a
Frobenius functor; and  a Green module for the Green ring $\mathcal{R}$ is a
Frobenius module over ${\mathcal R}$ which is also a Mackey functor.
 (See [D] and page 246
in [O5] for details.)\bigskip

\noindent\textbf{Examples.}
(i) For a given prime $p$, the functor $G\rightarrow \G_{0}( \Z_{p}[G] ) $ is a   Green ring and the functors $G\rightarrow  \Kr_{1}( R[G] )$, $\mathrm{Det}(
\Kr_{1}( R[G] ) )$, $\SK_{1}( R[G]
) $ are Green modules for the Green ring  $G\rightarrow \G_{0}( \Z_{p}[G] ) $.

(ii) For a given prime $p$, the functor $G\rightarrow H^{0}( G,\Z_p) $ is a Green ring and the functors $G\rightarrow H_{2}( G,%
\Z_{p})$, $H_{2}^{\rm ab}( G,\Z_{p})$, $\overline{%
H}_{2}( G,\Z_{p}) $ are Green modules for this   Green
ring.  (See page 281 in [O5].)\medskip

Let $\mathcal{C}$ be an arbitrary class of finite groups
which is closed with respect to subgroups and isomorphic images.
 For a given  finite group $G$ we write $\mathcal{C}( G) $ for the
set of subgroups of $G$ which belong to the class $\mathcal{C}$.

 We say that the Mackey functor $M$ is \textbf{generated} by $
\mathcal{C}$ if for any  $G$
\begin{equation*}
\bigoplus _{H\in \mathcal{C}( G) } M( H)\xrightarrow{\Sigma \,
\mathrm{Ind}_{H}^{G}} M( G)
\end{equation*}
is surjective.
%\smallskip

  We say that the Mackey functor $M $ is $\mathcal{C}$-\textbf{computable with
respect to induction} if for any $G$
\begin{equation*}
M( G) \cong \varinjlim_{{H\in \mathcal{C}( G)
}}M( H)
\end{equation*}%
where the direct limit is taken with respect to the induction and
conjugation maps.
%\smallskip

 We say that the Mackey functor $M $ is $\mathcal{C}$-\textbf{computable with
respect to restriction} if for any $G$
\begin{equation*}
M( G) \cong \varprojlim_{{H\in \mathcal{C}( G)
}}M( H)
\end{equation*}%
where the inverse limit is taken with respect to  the restriction and
conjugation maps.

 \medskip

From Theorem 11.1 in [O5] we have:

\begin{theorem}\label{thm50}
  (Dress)  Let $\mathcal{R}$ be a Green ring, $ M$ be a Green module
for $\mathcal{R}$ and let $\mathcal{C}$ be a class of finite groups with the
property that $\mathcal{R}$ is $\mathcal{C}$-generated; then $M$ is $%
\mathcal{C}$-computable with respect to both induction and restriction.
\end{theorem}

\begin{definition}\label{def51}
A Mackey functor is called $p$\textbf{-local} if it is a $\Z_{p}$%
-module valued functor.
\end{definition}

Using Brauer induction as in \ref{ss4b} above it can also be shown that (see
Theorem 11.2 in [O1]):

\begin{theorem}\label{thm52}
Let $\mathcal{E}$, resp. $\mathcal{E}_{p}$, denote the class of $\Q$-elementary, resp. of $\Q_{p}$-p-elementary, groups.

(i)  As previously, suppose that $R$ has field of fractions $N$ and let $X$
be an additive functor from the category of $R$-orders in semi-simple $N$-algebras with bimodule morphisms to the category of abelian groups. Then $M( G) =X( R[G] ) $ is a Mackey functor and
is in fact a Green module for the Green ring $G\rightarrow  \G_{0}( \Z[G%
] ) $ and $M$ is $\mathcal{E}$-computable with respect to both induction and restriction.

(ii)   Suppose further that $X$ is $p$-local; then $M( G)
=X( R[G] ) $ is a Green module for the Green ring $G\rightarrow \G_{0}( \Z_{p}[G] )$ and $M$ is $\mathcal{E}%
_{p}$-computable with respect to both induction and restriction.
\end{theorem}

From Theorem 11.9 in [O5] we have:

\begin{theorem}\label{thm53}
  Let $p$ denote a chosen prime number, let $R$ be a $p$-adically
complete integrally closed integral domain of characteristic zero with field
of fractions $N$ and let $X$ be a $p$-local additive functor on the category
of $R$-orders with bimodule morphisms. For any positive integer $n$ which
is prime to $p$, let $\zeta _{n}$ denote a primitive $n$-th root of unity in $\Q_{p}^{c} $
and set $R[\zeta _{n}] =R\otimes _{\Z_{p}}\Z_{p}[\zeta _{n}] $. Then:

(i) $X( R[G] ) $ is computable with respect to induction from $p$-elementary groups if, and only if, for any $n$ as above, any $p$-group $G$
and any homomorphism $t:G\rightarrow{\rm Gal}( \Q_{p}(
\zeta _{n}) /\Q_{p})$ with $H=\ker(t)$, the
induction map
\begin{equation*}
\mathrm{Ind}_{H}^{G}:H_{0}( G/H,X( R[\zeta _{n}]  [
H] ) ) \rightarrow X( R[\zeta _{n}] \circ
G)
\end{equation*}%
is bijective.

(ii) $X( R[G] ) $ is computable with respect to restriction to $p$-elementary
 groups if, and only if, for any $n$, as above, any $p$-group $%
p$ and any homomorphism $t:G\rightarrow{\rm Gal}( \Q_{p}(
\zeta _{n}) /\Q_{p})$ with $H=\ker(t)$, the
  restriction map
\begin{equation*}
{\rm Res}_{H}^{G}:X( R[\zeta _{n}] \circ G)
\rightarrow H^{0}( G/H, X( R[\zeta _{n}] [H]
) )
\end{equation*}
is bijective.
\end{theorem}

To obtain this theorem from Theorem 11.9 in [O5] it is helpful to note that
(see Sect. 6 in [CPT1]) we can decompose the ring $R[ \zeta _{n}] $
into a product of integral domains.

% ----------------------------------------------------------------------
% ----------------------------------------------------------------------
\section{$\SK_{1}$ for $p$-elementary groups}\label{s5}
\setcounter{equation}{0}

In this section we assume that in addition $R$ is a normal ring. Here $G$ is
a $\Q_{p}$-$p$-elementary group, written $G=C\rtimes P$ as above,
with $P$ a $p$-group and $C$ a cyclic group of order prime to $p$. We
therefore have decompositions of algebras:
\begin{equation*}
\Z_{p}[C] =\prod_{m}\Z_{p}[m] \text{ \ where \ }\Z_{p}[m] =\Z[\zeta
_{m}] \otimes _{\Z}\Z_{p}
\end{equation*}%
\begin{equation*}
R[C] =\prod_{m}R[m] \text{ \ where \ }R[m%
] =\Z_{p}[m] \otimes _{\Z_{p}}R
\end{equation*}%
and where $m$ runs over the divisors of the order of $C$.

We fix a surjective abelian character $\chi :C\rightarrow  \langle \zeta
_{m} \rangle$. We let $A_{m}$ denote the image of $P$ in $\mathrm{Aut}\langle \zeta _{m}\rangle $
given by conjugation; thus, by the
definition of a $\Q_{p}$-$p$-elementary group, $A_{m}$ identifies as
a subgroup of the cyclic group $\mathrm{Gal}( \Q_{p}(\zeta _{m}) /\Q_{p})$ and we let $H_{m}$ denote the
kernel of the map from $P$ to $A_{m} $. We endow $R[m] $ with
the lift of Frobenius given by the tensor product of the lift of Frobenius
on $R$ and the Frobenius automorphism of $\Z_{p}[m]$.
From Lemma 6.1 in [CPT1], we know that each $R[m]$ decomposes
as a product of integral domains each of which satisfies the Standing
Hypotheses.

Recall that each twisted group ring $R[m] \circ P$ contains the
standard group ring $R[m] [H_{m}]$, and, as per 6.a
in [CPT1], we have the induction and restriction maps
\begin{equation*}
\Kr_{1}( R[m] [H_{m}] )\, \overset{r_{\chi }}{%
\underset{i_{\ast }}{\leftrightarrows }}\, \Kr_{1}( R[m] \circ
P) .
\end{equation*}

In most of the remaining of this section we fix $m$ and drop the index $m$ where possible; so, in
particular, we write $H$, $A$, $r$ for $H_{m}$, $A_{m}$, $r_{m}$ and set $B=R[m%
] $. We write $I_{H}$ for the augmentation ideal $I( B[H%
] ) $ and we let $I_{P}$ denote the two-sided $B\circ P$ ideal
generated by $I_{H}$. Note for future reference that, since $R[P]
I_{H}$ is contained in the Jacobson radical of $R[P] $ and since
$B$ commutes with $H$, by the definition of $H$, it follows that $I_{P}$ is
contained in the Jacobson radical of $B\circ P$. In the usual way we have
the relative K-groups $\Kr_{1}( B[H] ,I_{H})$, and $\Kr_{1}( B\circ P,I_{P})$.

Recall that from Theorem 6.2 and Proposition 6.3 in [CPT1]
we have

\begin{theorem}\label{thm54}
We have
\begin{equation*}
r( \mathrm{Det}( (B\circ P)^{\times }) ) =\mathrm{Det}%
( B[H] ^{\times }) ^{A}\supseteq r\circ i_{\ast
}( \mathrm{Det}( B[H] ^{\times }) ) .
\end{equation*}
Using the decomposition $\mathrm{Det}( B[H] ^{\times
}) =\mathrm{Det}( 1+I_{H}) \times B^{\times },$ we obtain
inclusions
\begin{equation}\label{eq5.52}
r( \mathrm{Det}( 1+I_{P}) ) =\mathrm{Det}(
1+I_{H}) ^{A}\supseteq r\circ i_{\ast }( \mathrm{Det}(
1+I_{H}) )
\end{equation}
and
\begin{equation*}
\mathrm{Det}( 1+I_{P}) =i_{\ast }( \mathrm{Det}(
1+I_{H}) ) .
\end{equation*}
\end{theorem}

We shall now generalize a result of R. Oliver by establishing a
corresponding result for $\Kr_{1}^{\prime }$ by showing:

\begin{proposition}\label{pro55} We have
$
i_{\ast }(\Kr_{1}^{\prime }( B[H] ,I_{H})) =\Kr_{1}^{\prime
}( B\circ P,I_{P})$.
\end{proposition}

\begin{proof} Note that the result is immediate if $H=\{1\}$ and so we may
henceforth assume that $H$ contains an element $c$ of order $p$ which is
central in $P$. We write $t:P/H\rightarrow A$ for the isomorphism induced
by conjugation. We let $c$ denote a central element of order $p$ which lies
in $H$; note that here we are not necessarily assuming that $c\ $is a
commutator in $H$. We let $\overline{H}=H/\langle c\rangle$,
$\overline{P}=P/\langle c\rangle $ and we write $I_{\overline{H}}$
for $I( B[\overline{H}] ) $ and $I_{\overline{P}}$
for $I( B\circ \overline{P})$. Our proof proceeds in two
steps.\smallskip

\textbf{Step 1.} With the above notation we consider the diagram with exact
rows, where the downward arrows are induced by $i$
\begin{equation*}
\begin{array}{ccccccccc}
0 & \rightarrow & \mathcal{K}_{1} & \overset{\mathrm{defn}}{\hookrightarrow }
& \Kr_{1}^{\prime }( B[H] ,I_{H}) & \rightarrow &
\Kr_{1}^{\prime }( B[\overline{H}] ,I_{\overline{H}}) &
\rightarrow & 0 \\
&  & \alpha \downarrow\ \ &  & \beta \downarrow &  & \gamma \downarrow &  &  \\
0 & \rightarrow & \mathcal{K}_{2} & \overset{\mathrm{defn}}{\hookrightarrow }
& \Kr_{1}^{\prime }( B\circ P,I_{P}) & \rightarrow & \Kr_{1}^{\prime
}( B\circ \overline{P},I_{\overline{P}}) & \rightarrow & 0%
\end{array}%
\end{equation*}

and recall the exact sequence%
\begin{equation*}
0\rightarrow ( 1-c) B[ H] \rightarrow I_{H}\rightarrow
I_{\overline{H}}\rightarrow 0.
\end{equation*}%
We rewrite the top row of above diagram as%
\begin{equation*}
\begin{array}{ccccccccc}
&  &  &  & 1 & \rightarrow & 1 &  &  \\
&  &  &  & \uparrow &  & \uparrow &  &  \\
1 & \rightarrow & \mathcal{K}_{1} & \rightarrow & \Kr_{1}^{\prime }( B
[ H] ,I_{H}) & \rightarrow & \Kr_{1}^{\prime }( B[
\overline{H}] ,I_{\overline{H}}) & \rightarrow & 1 \\
&  & \uparrow \sigma &  & \uparrow \sigma ^{\prime } &  &     \uparrow  && \\
1 & \rightarrow & \GL( B[ H] ,(1-c)) & \rightarrow &
\GL( B[ H] ,I_{H}) & \rightarrow & \GL( B[
\overline{H}] ,I_{\overline{H}}) & \rightarrow & 1%
\end{array}%
\end{equation*}
We claim that $\sigma $ is surjective: to see this, first note that the map $%
\ker ( \sigma ) \rightarrow \ker ( \sigma ^{\prime }) $
induced by the map $\GL( B[ H] ,I_{H}) \rightarrow
\GL( B[ \overline{H}] ,I_{\overline{H}}) $ is
\begin{equation*}
\rho : \GL( B[ H] ,I_{H}) \cap E( B[ H]
) \rightarrow \GL( B[ \overline{H}] , I_{\overline{H}%
}) \cap E( B[ \overline{H}] ) .
\end{equation*}%
By the snake lemma, it will suffice to show that the map$\ \rho $ is
surjective; to see this, given $\overline{e}=\overline{e}_{1}\cdots
\overline{e}_{n}\in \GL( B[ \overline{H}] , I_{\overline{H}%
}) \cap E( B[ \overline{H}] ) ,$ we choose $%
e_{i}\in E( B[ H] ) $ with image $\overline{e}_{i}\in
E( B[ \overline{H}] ) $ and set $e=e_{1}\cdots e_{n};$
then, as the augmentation map $\varepsilon _{H}$ factors through $%
\varepsilon _{\overline{H}},$ we see that $\varepsilon _{H}( e)
=1 $ because $\varepsilon _{\overline{H}}( \overline{e}) =1$ and
so $e\in \GL( B[ H] ,I_{H})$. A similar argument shows
that $\GL( B\circ P,( 1-c) I_{P}) $ maps onto $\mathcal{%
K}_{2}$ and we have shown:
\begin{equation}
\GL( B[ H] ,(1-c)) \twoheadrightarrow \mathcal{K}_{1},\
\ \ \GL( B\circ P,( 1-c) ) \twoheadrightarrow \mathcal{K%
}_{2}.
\end{equation}%
We now show that $\beta $ is surjective by arguing by induction on the
order of the group $H$. Note that if $\overline{H}=\left\{ 1\right\}$,
then $I_{\overline{H}}=0=I_{\overline{P}},\ $and so in this case we
trivially have%
\begin{equation*}
\Kr_{1}^{\prime }( B[ H] ,\overline{I}) =0=\Kr_{1}^{\prime
}( B\circ P,\overline{I}_{P})
\end{equation*}%
and this starts the induction. By the inductive hypothesis we may assume
that $\gamma $ is surjective. Therefore, by the snake lemma, it will now
suffice to show that $\alpha $ is surjective.\medskip

\textbf{Step 2.}  As $( 1-c) B[ H] $ resp. $%
( 1-c) B\circ P$ lies in the radical of $B[ H] $ resp.
$B\circ P,$ we know from Lemma \ref{le14} (b) that there are surjections%
\begin{equation}
\kappa _{H}:( 1+( 1-c)) B[ H] )
\twoheadrightarrow \mathcal{K}_{1},\ \ \ \kappa _{P}:( 1+(
1-c) B\circ P) \twoheadrightarrow \mathcal{K}_{2}
\end{equation}
where we abbreviate $\kappa _{I_{H}}$ to $\kappa _{H}$ and $\kappa _{I_{P}}$
to $\kappa _{P}$. In order to show that the map $\alpha $ is surjective,
we shall show that for each $k\geq 1$ the map
\begin{equation*}
i_{\ast k}:\frac{\kappa _{H}( 1+( 1-c) ^{k}B[H]
) }{\kappa _{H}( 1+( 1-c) ^{k+1}B[H]
) }\rightarrow \frac{\kappa _{P}( 1+( 1-c) ^{k}B\circ
P) }{\kappa _{P}( 1+( 1-c) ^{k+1}B\circ P) }
\end{equation*}%
is surjective.  This will then prove the proposition because the $(
1-c) $-adic and $p$-adic topologies are cofinal in both $(
1-c) B[H] $ and $( 1-c) B\circ P$.

Recall that $B=R[m] $ and $A\subset \mathrm{Gal}( R[m%
] /R) $. Let $\overline{B}=B/pB$, $\overline{R}=R/pR$
 and recall that by the Standing Hypothesis $\overline{R}$ is an
integral domain  and $\overline{B}=\overline{R}\otimes \mathbb{F}_{p}[m] $.
To establish the surjectivity of $i_{\ast k}$, we choose
a normal basis generator $\widetilde{\mu }$  of $\mathbb{F}_{p}[m%
] $ over $\mathbb{F}_{p}[m] ^{A}$ as follows: let $\zeta $
denote a primitive $m$-th root of unity in the algebraic closure of $%
\mathbb{F}_{p}$; then $A$ identifies as a subgroup of $\mathrm{Gal}(
\mathbb{F}_{p}( \zeta ) /\mathbb{F}_{p}) $ and $\mathbb{F}%
_{p}[m] $ identifies as a product of copies of $\mathbb{F}_{p}( \zeta )$. We then take $\widetilde{\mu }$ to be the direct
product of copies of a normal integral basis of $\mathbb{F}_{p}( \zeta
) /\mathbb{F}_{p}( \zeta ) ^{A}$ and so we observe that $%
\widetilde{\mu }$ is an invertible element of $\mathbb{F}_{p}[m]$.
This then affords a normal basis generator (also denoted $\widetilde{\mu
})$ of $\ov{B}$ over the group ring $\ov{B}^{A}[A]$.
We then choose a lift $\mu \in B$ of $\widetilde{\mu}$. Then for $g\in
P-H$, $g( \widetilde{\mu }) \neq \widetilde{\mu }$ and so $\mu
^{-1}( g( \mu ) ) -1\in B^{\times }.$ So for any $%
\lambda \in B$ we can find $\lambda ^{\prime }$ such that
\begin{equation*}
\lambda =\lambda ^{\prime }( \mu ^{-1}( g( \mu )
) -1)
\end{equation*}
and we note that ${\rm coker }( i_{\ast k})$ is generated
by elements of the form $1+( 1-c) ^{k}\lambda g^{-1}$ with $g\in
P-H$.

Using the identity $g( \mu ) =g\cdot \mu \cdot g^{-1}$, we then note the
identity
\begin{eqnarray*}
\kappa _{P}( 1+( 1-c) ^{k}\lambda g) &=&\kappa
_{P}( 1+( 1-c) ^{k}\lambda ^{\prime }( \mu ^{-1}(
g( \mu ) ) -1) g) \\
&=&\kappa _{P}( 1+\mu ^{-1}( 1-c) ^{k}\lambda ^{\prime }g\mu
-( 1-c) ^{k}\lambda ^{\prime }g) \\
&\equiv &\kappa _{P}( ( 1+( 1-c) ^{k}\mu ^{-1}\lambda
^{\prime }g\mu ) ( 1-( 1-c) ^{k}\lambda ^{\prime
}g) ) \mod ( 1-c) ^{k+1}.
\end{eqnarray*}%
Using the fact that, because $c\in H,$ it commutes with $B$, we have shown
\begin{equation*}
\kappa _{P}( 1+( 1-c) ^{k}\lambda g^{-1}) \equiv
\kappa _{P}( \mu ^{-1}( 1+( 1-c) ^{k}\lambda ^{\prime
}g^{-1}) \mu ( 1+( 1-c) ^{k}\lambda ^{\prime
}g^{-1}) ^{-1}) \mod ( 1-c) ^{k+1}
\end{equation*}%
which, being a commutator, has trivial image in $\Kr'_{1}( B\circ
P,I_{P}) $.
\end{proof}

Recall we write $I_{H}$ for $I_{B[H] }$.

\begin{proposition}\label{pro56}
 ${\rm Det}( 1+I_{H}) $ is a cohomologically trivial $A$-module; hence
\begin{equation*}
\mathrm{Det}( 1+I_{H}) ^{A}=H^{0}( A,\mathrm{Det}(
1+I_{H}) ) =H_{0}( A,\mathrm{Det}( 1+I_{H})
)
\end{equation*}%
and
\begin{equation*}
H_{1}( A,\mathrm{Det}( 1+I_{H}) ) =\{ 1\} =%
\widehat{H}^{-1}( A,\mathrm{Det}( 1+I_{H}) ) =\left\{
1\right\} .
\end{equation*}
\end{proposition}

\begin{proof} Recall that we write $\mathcal{A}_{H}=\ker ( B[H%
] \rightarrow B[H^{\rm ab}] )$. With the notation of
\ref{sslog} (using the fact that $B$ is a direct product of algebras which satisfy
the Standing Hypotheses), we have an exact sequence%
\begin{equation*}
0\rightarrow \phi ( \mathcal{A}_{H}) \rightarrow \mathrm{Det}%
( 1+I_{H}) \rightarrow 1+I( B[H^{\rm ab}] )
\rightarrow 1.
\end{equation*}%
As seen previously, the $B$-lattice $\phi ( \mathcal{A}_{H}) $ is
$R[A] $-free and hence is a cohomologically trivial $A$-module.
On the other hand $1+I( B[H^{\rm ab}] ) $ can be filtered
with subquotients all of which are isomorphic to $\overline {B}$.
The result then follows since all these terms are also all cohomologically
trivial $A$-modules.
\end{proof}

\begin{corollary}\label{cor57}
We have
 $
i_{\ast }(\SK_{1}( B[H]) ) =\SK_{1}( B\circ
P,I_{P})=\SK_1(B\circ P)$.
\end{corollary}

\begin{proof} By Theorem \ref{thm1} we know that $\mathrm{Det}( \Kr_{1}( R%
[G] ) ) =\mathrm{Det}( R[G]
^{\times })$. Consider the augmentation map $B[H]
\rightarrow B$ and the induced map
\begin{equation*}
B\circ P\rightarrow B\circ A\cong M_{\vert A\vert }(
B^{A})
\end{equation*}
with the latter isomorphism coming from the proof of Theorem 6.2 in [CPT1].

This then leads us to consider the commutative diagram with exact rows,
where the downward arrows are all induced by $i_{\ast }$:
\begin{equation}\label{eq5.55}
\begin{array}{ccccccccc}
0 & \rightarrow & \SK_{1}( B[H] ) & \rightarrow &
\Kr_{1}^{\prime }( B[H] , I_{H}) & \rightarrow & \mathrm{Det}( 1+I_{H}) & \rightarrow & 1 \\
&  & \downarrow \rho &  & \downarrow i &  & \downarrow \tau &  &  \\
0 & \rightarrow & \SK_{1}( B\circ P,I_{P}) & \rightarrow &
\Kr_{1}^{\prime }( B\circ P,I_{P}) & \rightarrow & \mathrm{Det}%
( 1+I_{P}) & \rightarrow & 1.%
\end{array}%
\end{equation}%
By Proposition \ref{pro55} we know that the middle downward arrow $i$ is
surjective. In order to show that $\rho$ is surjective, by the snake
lemma, it will suffice to show that $\ker i$ maps onto $\ker \tau$.
Suppose therefore that $\mathrm{Det}( x) \in \ker \tau$. We
have seen that the restriction map $r$ is injective and $r\circ i_{\ast }$
coincides with the norm map $N_{A};$ so by Proposition \ref{pro56} we deduce that we
can write
\begin{equation*}
\mathrm{Det}( x) =\prod\limits_{a\in A}\mathrm{Det}(
y_{a}) ^{a-1}
\end{equation*}%
with $y_{a}\in 1+I_{H}$. Now consider the image of the element  $z=\prod
y_{a}^{a-1}\in\Kr_{1}^{\prime }( B[H] ,I_{H}) $
in $\Kr_{1}( B\circ P)$: clearly this maps to $\mathrm{Det}( x) $ under ${\rm Det}$; moreover, each term $%
y_{a}^{a-1}=ay_{a}a^{-1}y_{a}^{-1}$ becomes a commutator in $(B\circ
P)^{\times }$ and so vanishes in $\Kr_{1}( B\circ P) $, as
required. To conclude we show that $\SK_{1}( B\circ P) =\SK_{1}( B\circ
P, I_{P}) $ and this follows from the exact sequence%
\begin{equation*}
1\rightarrow I_{P}\rightarrow B\circ P\rightarrow B\circ A_{m}\rightarrow 1
\end{equation*}%
and the equalities
\begin{equation*}
\SK_{1}( B\circ A_{m}) =\SK_{1}( M_{\left\vert
A_{m}\right\vert }( B^{A_{m}}) ) =\SK_{1}(
B^{A_{m}}) =\SK_{1}( R\otimes \mathbb{Z}[ m] )
=\{1\}.
\end{equation*}
\end{proof}

We conclude this section by showing:

\begin{theorem}\label{thm58}
We have
$$
i_{\ast }(\SK_{1}( B[H]) ) =H_{0}(
A,\SK_{1}( B [H] ) )$$ in
$\SK_{1}( B\circ P)$.
\end{theorem}

\begin{proof}
We start by noting as previously that for $x\in \Kr_{1}^{\prime }( B\circ
P,I_{P}) $ and $a\in A$ we have $x^{a}x^{-1}=1$ and so $I_{A}\Kr_{1}^{\prime }( B\circ P,I_{P}) =0$. Next we note that by
taking $A$-homology of the top exact row in  (\ref{eq5.55}) and
using the above Proposition \ref{pro56}  we get
\begin{equation*}
\begin{array}{cccccccc}
  H_{0}( A,SK_{1}( B[ H] ) )
& \hookrightarrow & H_{0}( A,K_{1}^{\prime }( B[ H]
,I_{H}) ) & \twoheadrightarrow & H_{0}( A,\mathrm{Det}(
1+I( B[ H] ) ) )   \\
    \downarrow \rho &  & \downarrow i &  & \downarrow \tau &     \\
  \SK_{1}( B\circ P,I_{P}) & \hookrightarrow &
\Kr_{1}^{\prime }( B\circ P,I_{P}) & \twoheadrightarrow & \mathrm{Det}%
( 1+I( B\circ P) ) %
\end{array}%
\end{equation*}%
Now from the above we know that $\tau $ is an isomorphism; by Corollary \ref{cor57} $%
\rho $ is surjective; by Proposition \ref{pro55} $i$ is surjective; hence, by the
snake lemma, it will now suffice to show that $i$ is an isomorphism.

By Lemma 8.3.ii and Lemma 8.9 in [O5] (see also the discussion on page 278
of [O5]) we know that we can find a finite $p$-group $\widetilde{H}$ with a
central subgroup $\Sigma $ so that we have a diagram
\begin{equation*}
\begin{array}{ccccccccc}
1 & \rightarrow & \Sigma & \rightarrow & \widetilde{P} & \rightarrow & P &
\rightarrow & 1 \\
&  & \ \ \uparrow = &  & \uparrow &  & \uparrow &  &  \\
1 & \rightarrow & \Sigma & \rightarrow & \widetilde{H} & \overset{\alpha _{0}%
}{\rightarrow } & H & \rightarrow & 1%
\end{array}%
\end{equation*}%
with $H_{2}( \alpha _{0}) =0,$ where $H_{2}( \alpha
_{0}) $ is the map $H_{2}( \widetilde{H},\mathbb{Z}_{p})
\rightarrow H_{2}( H,\mathbb{Z}_{p}) $ induced by $\alpha _{0};$
hence by Lemma 8.9 in [O5] we know that we have $\SK_{1}( B[ H%
] ) =(1);$ and hence
\begin{equation}\label{eq5.5new}
\Kr_{1}^{\prime }( B[ H] ,I_{H}) =\mathrm{Det}(
1+I( B[ H] ) ) ;
\end{equation}%
and therefore, by Proposition \ref{pro56}, $\Kr_{1}^{\prime }( B[ H]
,I_{H}) $ is $A$-cohomologically trivial.

To conclude we consider the exact sequence which defines $\mathcal{J}$
\begin{equation*}
1\rightarrow \mathcal{J}\rightarrow B[ \widetilde{H}] \rightarrow
B[ H] \rightarrow 1
\end{equation*}%
together with the diagram
\begin{equation*}
\begin{array}{ccccccccc}
&  & \Kr_{1}^{\prime }( B[ \widetilde{H}] ,\mathcal{J})
& \rightarrow & H_{0}( A,\Kr_{1}^{\prime }( B[ \widetilde{H}%
] ,I_{\widetilde{H}}) ) & \rightarrow & H_{0}(
A, \Kr_{1}^{\prime }( B[ H] ,I_{H}) ) & \rightarrow
& 1 \\
&  & \downarrow a &  & \downarrow b &  & \downarrow c &  &  \\
1 & \rightarrow & \Kr_{1}^{\prime }( B\circ \widetilde{P},\mathcal{J}%
) & \rightarrow & \Kr_{1}^{\prime }( B\circ \widetilde{P},I_{%
\widetilde{P}}) & \rightarrow & \Kr_{1}^{\prime }( B\circ
P,I_{P}) & \rightarrow & 1%
\end{array}%
\end{equation*}%
where the top row is obtained by using the map $\Kr_{1}^{\prime }( B[
\widetilde{H}] ,\mathcal{J}) \rightarrow H_{0}(
A,K_{1}^{\prime }( B[ \widetilde{H}] ,\mathcal{J})
) $ by taking the $A$-homology of the exact sequence

\begin{equation*}
1\rightarrow \Kr_{1}^{\prime }( B[ \widetilde{H}] ,\mathcal{J}%
) \rightarrow \Kr_{1}^{\prime }( B[ \widetilde{H}] ,I_{%
\widetilde{H}}) \rightarrow \Kr_{1}^{\prime }( B[ H]
,I_{H}) \rightarrow 1
\end{equation*}%
which follows from the snake lemma applied to the diagram%
\begin{equation*}
\begin{array}{ccccccccc}
&  &  &  & \Kr_{1}( B[ \widetilde{H}] ) & = &
\Kr_{1}( B[ \widetilde{H}] ) &  &  \\
&  &  &  & \downarrow &  & \downarrow &  &  \\
1 & \rightarrow & \Kr_{1}^{\prime }( B[ \widetilde{H}] /%
\mathcal{J},I_{\widetilde{H}}/\mathcal{J}) & \rightarrow & \Kr_{1}(
B[ \widetilde{H}] /\mathcal{J}) & \rightarrow & \Kr_{1}(
B[ \widetilde{H}] /I_{\widetilde{H}}) & \rightarrow & 1.%
\end{array}
\end{equation*}%
The map $a$ is induced by $i$ and so is surjective; the map $b$ is an
isomorphism by (\ref{eq5.5new}) and Theorem \ref{thm54}; hence the map $c$ is an isomorphism
as required.
\end{proof}

% ----------------------------------------------------------------------
% ----------------------------------------------------------------------
\section{Arbitrary finite groups.}\label{s6}
\setcounter{equation}{0}

Throughout this section we shall always suppose that $R$ satisfies the
Standing Hypotheses and is normal and that $G$ is an arbitrary finite group.

We start by proving Theorem \ref{thm7}: we do this by using the induction theorems
of Section \ref{s4}, using the isomorphism $\Theta _{R[ G] }$
for $p$-groups, and the work in Section \ref{s5} for $\mathbb{Q}_{p}$-$p$%
-elementary groups, in order to produce an isomorphism  $\Theta _{R [ G ] }$
  for arbitrary finite groups $G$.

In subsection \ref{ss6c} we construct Adams operations on the determinantal groups $%
\mathrm{Det}( R[ G] ^{\times })$. We then use these
Adams operations to construct a group logarithm $\upsilon _{G}$, which
naturally extends the group logarithm for $p$-groups described in
\ref{sslog}.This enables us to construct a long exact sequence (see Proposition
\ref{pro75}) which generalizes the exact sequence (\ref{eq1.3})  (which
was valid only for $p$-groups); this result provides a deep understanding
of the determinantal groups $\mathrm{Det} ( R  [G ] ^{\times
}) $. \textit{Inter alia} this sequence, together with other
constructions, allows us to provide a more constructive definition of the
map $\Theta _{R[ G] }$.

% ----------------------------------------------------------------------
\subsection{Proof of Theorem \ref{thm7}.}\label{ss6a}

In this subsection we suppose that $R$ satisfies the Standing Hypotheses.
In \ref{ss3c} we constructed the map $\Theta _{R[ G] }$ in the case when
$G$ is a $p$-group. Indeed, in that case, by Theorem \ref{thm43} we have the
natural isomorphism%
\begin{equation}
\Theta _{R[ G] }:\SK_{1}( R[ G])
\rightarrow \overline{H}_{2} ( G, R ) _{\Psi }.
\end{equation}

Suppose now that $G=C\rtimes P$ is a $\mathbb{Q}_{p}$-$p$-elementary group
and, with the notation of Section 5, we then have decompositions
\[
R [ G ] =\oplus _{m}R[ H_{m} ] \circ P,\ R [ G_{r}] =R [ C ] =\oplus _{m}R [ m ]
\]%
(so that $H_{m}=\ker ( \mathrm{conj}:P\rightarrow \mathrm{Aut}\langle \zeta _{m}\rangle ) $ and $A_{m}=\mbox{Im}
( \mathrm{conj}:P\rightarrow \mathrm{Aut}\langle \zeta _{m}\rangle ) $. Thus we obtain the further decompositions
\begin{eqnarray*}
H_{2}( G,R[ G_{r}]) &=&\oplus _{m}H_{2} ( G,R[ m] ) =\oplus _{m}H_{2} ( H_{m},R [ m ]
^{A_{m}} ) \\
&=&\oplus _{m}H_{2} ( H_{m},R [ m ] _{A_{m}} ) =\oplus
_{m}H_{0} ( A_{m},H_{2} ( H_{m},R [ m ]  )  )
\end{eqnarray*}%
by using the fact that $R [ m ]$ is $A_{m}$-free. We then have
similar decompositions for $H_{2}^{\mathrm{ab}} ( G,R [ G_{r} ]
 )$, $\overline{H}_{2} ( G,R [ G_{r} ]  )$,
$\overline{H}_{2} ( G,R [ G_{r} ] ) _{\Psi } $.

By Theorem \ref{thm58} there is a natural isomorphism%
\begin{equation*}
\SK_{1}( R[ m] [ G] ) =\SK_{1}( \oplus
_{m}R[ m] [ H_{m}] \circ P) =\oplus
_{m}H_{0}( A_{m}, \SK_{1}( R[ m] [ H_{m}]
) )
\end{equation*}%
and so applying the functor $H_{0}( A_{m},\, -) $ to the isomorphism
\begin{equation*}
\Theta _{R[ m] [ H_{m}] }:\SK_{1}( R[ m]
[ H_{m}] ) \rightarrow \overline{H}_{2}( H_{m},R[
m] ) _{\Psi }
\end{equation*}%
we get the isomorphisms for each $m$
\begin{equation*}
H_{0}( A_{m},\Theta _{R[ m] [ H_{m}] })
:H_{0}( A_{m}, \SK_{1}( R[ m] [ H_{m}] )
) \rightarrow H_{0}( A_{m},\overline{H}_{2}( H_{m},R[ m%
] ) ) _{\Psi }.
\end{equation*}%
For a $\mathbb{Q}_{p}$-$p$-elementary group $G$ these isomorphisms add
together to give the desired isomorphism
\begin{equation}
\Theta _{R[ G] }:\SK_{1}( R[ G] )
\rightarrow \overline{H}_{2}( G,R[ G_{r}] ) _{\Psi }.
\end{equation}%
In conclusion we note that for arbitrary $G$ by Theorem \ref{thm53} (i)
\begin{equation*}
\SK_{1}( R[ G] ) =\varinjlim_{ H\in \mathcal{E}%
_{p} ( G ) } \SK_{1}( R[ H] ) .
\end{equation*}

 To complete the proof of Theorem \ref{thm7} we show:

\begin{lemma}\label{LemmaNss6} We have
\[
\overline{H}_{2} ( G,R [ G_{r} ] ) =\varinjlim_{ H\in \mathcal{E}%
_{p} ( G ) } \overline{H}_{2} ( H,R[ H_{r}]).
\]
\end{lemma}

\begin{proof} Let $E_{p} ( G ) $ denote the set of subgroups of $G$ which
are $p$-elementary, so that of course $E_{p} ( G ) \subset
\mathcal{E}_{p} ( G )$. We start by showing
\[
H_{2}( G,R[ G_{r}] ) =\varinjlim_{ H\in E_{p} ( G ) } H_{2}( H,R [ H_{r} ] ) ;
\]
then, since every $p$-elementary subgroup of $G$ is trivially $\mathbb{Q}%
_{p} $-$p$-elementary, by Theorem \ref{thm50} it will follow that%
\[
H_{2} ( G,R [ G_{r} ]  ) =\varinjlim_{ H\in \mathcal{E}%
_{p} ( G ) }   H_{2} ( H,R [ H_{r} ]  ) .
\]
Let $G_{p}$ denote a $p$-Sylow subgroup of $G$. Since the index $ (
G:G_{p} )  $ is a unit of $R$, we have the equality
\[
\mathrm{Cor}_{G_{p}}^{G} ( H_{2} ( G_{p},R [ G_{r} ])) =H_{2} ( G,R [ G_{r} ] ) .
\]%
Next we decompose $G_{r}$ into disjoint cycles under $G_{p}$-conjugation%
\[
G_{r}=\bigcup _{i\in I}g_{i}^{G_{p}},
\]%
we let $H_{i}$ denote the subgroup of $G_{p}$ that centralizes $g_{i}$. Then
it follows that
\[
H_{2} ( G_{p},R [ G_{r}])
=\sum\nolimits_{i}H_{2} ( G_{p},\sum\nolimits_{\gamma \in
H_{i}\backslash G_{p}}Rg_{i}^{\gamma } ) =\sum\nolimits_{i}H_{2} (
H_{i},Rg_{i} )
\]%
and the result follows because $ \langle g_{i} \rangle \times
H_{i}$ is $p$-elementary.

The proof that
\[
H_{2}^{\mathrm{ab}} ( G,R [ G_{r} ]  ) =\varinjlim_{ H\in \mathcal{E}%
_{p} ( G ) } H_{2}^{\mathrm{ab}} ( H,R [ H_{r} ]
 )
\]%
is very similar (see the proof of Corollary \ref{cor8} in Appendix B).\ \ \
\end{proof}

\bigskip

% ----------------------------------------------------------------------
\subsection{Extending the group logarithm}\label{ss6b}

In this subsection we consider the extension of the group logarithm in the trivial
case when $G=\{1\}$ and define
\begin{equation*}
\mathcal{L}_{R}( R^{\times }) =\left(\log \circ \frac{p}{F}\right)(
R^{\times }) \subset \log ( 1+pR) \subset pR
\end{equation*}
where abusively $p/F$ denotes the map on $R^\times$ given by $p/F(u)=u^p/F(u)$.

We define
\begin{equation*}
\mathcal{M}( R, F) =\{u\in R^{\times }\mid \ F( u)
=u^{p}\},\ \ \Lambda ( R, F) =R^{\times }/ \mathcal{M}( R, F) .
\end{equation*}%
When $F$ is fixed we shall write $\Lambda ( R) $ and $\mathcal{M}( R)$ in place of $\Lambda ( R, F) $ and $\mathcal{M}( R, F)$. We claim that the following sequence is exact:
\begin{equation}\label{newExseq6.3}
R^{\times }\xrightarrow{\frac{1}{p}\mathcal{L}_{R}\oplus \theta }
R\oplus \Lambda ( R) \xrightarrow{T}R\rightarrow 0
\end{equation}%
where $\theta ( u) =u\mod\mathcal{M}( R) $ and
where $T( x\oplus y ) =x-\frac{1}{p}\mathcal{L}_R( y)$.
Indeed: $ T$ is surjective, since obviously $R$ maps onto $R$; clearly
\begin{equation*}
T \left( (\frac{1}{p}\mathcal{L}_{R}\oplus \theta) ( u)\right)
=\frac{1}{p}\mathcal{L}_{R}(u) -\frac{1}{p}\mathcal{L}_{R}(
u ) =0;
\end{equation*}%
and, if $  T ( x\oplus y ) =0$, then $ x=\frac{1}{p}\mathcal{L}%
_{R} ( y )$,   and so $(\frac{1}{p}\mathcal{L}_{R}\oplus \theta)
 ( y ) =x\oplus y.$

\begin{lemma}\label{le60}
(a) $\mathcal{M} ( R ) \cap  ( 1+pR ) =\{1\}.$

(b) $\ker  ( \mathcal{L}_{R} ) =\left\{
\begin{array}{c}
 \mathcal{M}( R) \ \ \ \ \ \ \  \ \ \ \ \ \text{if }p>2, \\
\mathcal{M}( R) \oplus \left\langle \pm 1\right\rangle \ \ \
\text{if }p=2.%
\end{array}%
\right. $
\end{lemma}

\begin{proof} (a) For the sake of contradiction we consider $1+p^{n}x\in \mathcal{M}(R) $ with $n>0$ and $x\notin pR$. Then we have congruences modulo $p^{n+1}R$
\begin{equation*}
1+p^{n}x^{p}\equiv 1+p^{n}F(x) \equiv F( 1+p^{n}x)
\equiv ( 1+p^{n}x) ^{p}\equiv 1
\end{equation*}%
which implies $x\in pR$.

For (b) we consider the factorization of $\mathcal{L}_{R}$ given by
\begin{equation*}
R^{\times } \xrightarrow{p/F}  1+pR \xrightarrow{\log } pR
\end{equation*}
where $p/F ( u ) =u^{p}F ( u ) ^{-1}$. The result then
follows since, because $pR$ is a prime ideal, we have:
\begin{equation*}
\ker \left( \log :1+pR\rightarrow pR\right) =\left\{
\begin{array}{cc}
\{1\}\ \ \ \  \ \text{ if }p>2, \\
\ \ \left\langle \pm 1\right\rangle \ \ \ \ \  \ \text{if }p=2.\ \ \
\end{array}%
\right.
\end{equation*}
\end{proof}

\begin{example}
{\rm We can use the structure of $R^{\times }$ to work out $\mathcal{M} (
R ) $\ and $\Lambda  ( R ) $ in the following cases:

1. If $W$ denotes the valuation ring of a finite non-ramified extension of $%
\Q_{p};$ then $\mathcal{M} ( W ) = \mu_{W}^{\prime }$,
the group of roots of unity of $W$ of order prime to $p$ and $\Lambda( W ) \cong pW$.

2. Using $k[t]^\times=k^\times$ we can see that   $\mathcal{M} ( W\langle t\rangle ) = \mu
_{W}^{\prime }$ and $\Lambda ( W \langle t \rangle )
\cong 1+p W \langle t \rangle$.

3. Also $\mathcal{M} ( W[[t]  ] ) =$ $\mu
_{W}^{\prime }$ and $\Lambda (W[[t]])
\cong 1+ ( p, t ) W[[t]]$.

4. Any unit $u$ of the d.v.r $W\{\{t\}\}$ can be written as $u=t^{n}v (
1+pr)$ with $r\in W\{\{t\}\}$, $n\in \Z$ and $v\in W%
[[t]]] ^{\times };$ therefore $\mathcal{M} (
W\{\{t\}\} ) =$ $\mu _{W}^{\prime }\times t^{\Z}$ and we have
an exact sequence
\begin{equation*}
1\rightarrow 1+pW\{\{t\}\}\rightarrow \Lambda  ( W\{\{t\}\} )
\rightarrow 1+tk[[t]]\rightarrow 1
\end{equation*}%
where we write $k$ for the residue class field of $W$.}
\end{example}

  For a $p$-group $G$ we splice the above exact sequence  (\ref{newExseq6.3})
together with the exact sequence
\begin{equation*}
\mathrm{Det} ( 1+I_{R [ G ] } ) \xrightarrow{\nu _{G}} p\phi  ( I_{R [ G ] } ) \rightarrow G^{
\mathrm{ab}}\otimes \frac{R}{\ ( 1-F ) R}\rightarrow 0
\end{equation*}%
from Theorem 3.17 in [CPT1], and we get the exact sequence for the whole
group $\mathrm{Det}R[G]^{\times } )$:
\begin{eqnarray}\label{neweq6.4}
\mathrm{Det}\left( R [ G ] ^{\times }\right)  &=&\mathrm{Det} (
1+I_{R [ G ] } ) \oplus R^{\times }\xrightarrow{{\nu _{G}^{\prime
}\oplus \frac{1}{p}\mathcal{L}_{R}\oplus \theta _{G}}} \phi
 ( I_{R [ G ] } ) \oplus R\oplus \Lambda (R)
\\
&=&pR [ C_{G} ] \oplus \Lambda  ( R ) \rightarrow G^{%
\mathrm{ab}}\otimes \frac{R}{ ( 1-F ) R}\oplus R  \notag
\end{eqnarray}%
where for brevity we set $\nu _{G}^{\prime }=\frac{1}{p}\nu _{G}$ and $%
\theta _{G} $ is the composition of augmentation and reduction modulo $\mathcal{M}( R)$.

% ----------------------------------------------------------------------
\subsection{Adams operations and norm maps.}\label{ss6c}

From here on until the end of the paper we suppose that, in addition to the Standing Hypotheses, $R$
also satisfies the following two additional hypotheses:

(1) \ The $\mathbb{F}_{p}$-algebra $\overline{R}\otimes_{\mathbb{F}_p}{\mathbb{F}}^c_{p}$ contains only finitely many orthogonal idempotents.

(2) \ For any non-ramified extension $L$ of $\Q_{p}$ we set $\Delta =%
\mathrm{Gal} ( L/\Q_{p} )$; then the group $\Lambda  (
R_{L} ) $ is a $\Z_{p}$-module which is $\Delta $%
-cohomologically trivial and for any $p$-subgroup $\Gamma $ of $\Delta $, $%
H^{1}( \Gamma , \mathcal{M} ( R_{L} )  ) =\{1\}$.

\smallskip

For an integer $n$ and a virtual character $\chi $ of $G$, recall that we
define $\psi ^{n}\chi $ by the rule that for $g\in G$ we have $\psi ^{n}\chi
( g) =\chi(g^{n})$. Clearly $ \psi ^{n}\chi  $ is a central function on
$G$; in fact one sees easily by Newton's formulas that $\psi ^{n}\chi $ is a
virtual character of $G$. In this subsection we use these Adams operations
on characters to define Adams operations on the group $\mathrm{Det}( R%
[G] ^{\times })$.

For $\mathrm{Det}( x) \in \mathrm{Det}( R[G]
^{\times }) $ and for an integer $n\ $we define $\psi ^{n}\mathrm{Det}%
( x) $ to be the character function given by the rule that for a
virtual character $\chi $ of $G$
\begin{equation*}
\psi ^{n}\mathrm{Det}( x) ( \chi ) =\mathrm{Det}(
x) ( \psi ^{n}\chi ) .
\end{equation*}
Then, as per Theorem 1 in [CNT] (see also Sect. 9 in [T]), we have:

\begin{theorem}\label{thm62}
For any integer $n$,  $\psi^{n}\mathrm{Det}( R[G]
^{\times }) \subset \mathrm{Det}( R[G] ^{\times
}) .$
\end{theorem}

The proof of this result is in many ways similar to the proof given in [CNT]
and Sect. 9 of [T] (when $R$ is the valuation ring of a finite non-ramified
extension of $\Q_{p})$. The details are provided in an Appendix.
However, note that the proof in [CNT] and [T] requires a number of modifications. In
particular, one issue that arises at the conclusion of the proof (see pages
114-115 in [T]) is the following: for a finite non-ramified extension $L$ of
$\Q_{p}$, if we set $R_{L}=R\otimes \O_{L}$ then we do not
necessarily know that:
\begin{equation*}
\mathcal{N}_{L/\Q_{p}}( \mathrm{Det}( R_{L}[G]
^{\times }) ) =\mathrm{Det}( R[G] ^{\times
}) ;
\end{equation*}
and this is because we do not necessarily know that the group of norms $N_{L/\Q_{p}}( R_{L}^{\times }) $ coincides with $R^{\times }$. However,
we do have from Theorem 4.3 of [CPT1]:

\begin{theorem}\label{thm63}
Let $G$ be a $p$-group. With the above notation we have the equality
\begin{equation*}
\mathcal{N}_{L/\Q_{p}}( \mathrm{Det}( 1+I( R_{L}
[G] ) ) ) =\mathrm{Det}( 1+I( R[G%
] ) ) .
\end{equation*}
\end{theorem}

This is a key-result, as is explained in the Appendix.

The proof of Theorem \ref{thm62} then follows very closely the proof given in [CNT]
and [T]  by using the above result
together with the decomposition
\begin{equation*}
\mathrm{Det}( R[G] ^{\times }) =\mathrm{Det}(
1+I( R[G] ) ) \times R^{\times };
\end{equation*}%
However, the special case where $G$ is a $p$-group is now considerably
more involved. The details are provided in the Appendix. \bigskip

For future reference note that if $H$ is a subgroup of $G$, then, by
definition, for $\mathrm{Det}( y) \in  \mathrm{Det}( R%
[H] ^{\times }) $ we have,
\begin{equation*}
\mathrm{Ind}_{H}^{G}( \mathrm{Det}( y) ) ( \chi
) =\mathrm{Det}( y) ( \text{\textrm{Res}}%
_{H}^{G}( \chi ) )
\end{equation*}%
and so
\begin{eqnarray*}
( \psi ^{n}\mathrm{Ind}_{H}^{G}( \mathrm{Det}( y)
) ) ( \chi ) &=&\mathrm{Ind}_{H}^{G}( \mathrm{Det%
}( y) ) ( \psi ^{n}\chi ) \\
&=&\mathrm{Det}( y) ( \text{\textrm{Res}}_{H}^{G}( \psi
^{n}\chi ) ) =\mathrm{Det}( y) ( \psi ^{n}\text{%
\textrm{Res}}_{H}^{G}( \chi ) ) \\
&=&( \psi ^{n}\mathrm{Det}( y) ) ( \text{\textrm{%
Res}}_{H}^{G}( \chi ) ) =( \text{\textrm{Ind}}%
_{H}^{G}\psi ^{n}\mathrm{Det}( y) ) ( \chi ) .
\end{eqnarray*}%
Thus we have shown that

\begin{lemma}\label{le64}
For $H< G$, $\psi ^{n}$ commutes with the map $\mathrm{Ind}_{H}^{G}$ on
determinants.
\end{lemma}

% ----------------------------------------------------------------------
\subsection{The group logarithm for arbitrary finite groups}\label{ss6d}

We define
$$
\Psi :\mathrm{Det}( R[G] ^{\times })
\rightarrow \mathrm{Det}( R[G] ^{\times })
$$
 by the
rule $\Psi ( \mathrm{Det}( x) ) =\psi ^{p}\mathrm{Det}%
( F( x) ) $ and we define
\begin{equation*}
\upsilon_{G} ( \mathrm{Det}( x) ) =\phi \circ
\log \left( \frac{\mathrm{Det}( x) ^{p}}{ \Psi ( \mathrm{Det}(
x) ) }\right).
\end{equation*}%
 We begin this subsection by showing the following
generalization of Theorem \ref{thm22} in \ref{sslog}:

\begin{theorem}\label{thm65} We have \
$\upsilon_{G}( \mathrm{Det}( R[G] ^{\times
}) ) \subset pR[C_{G}] $.
\end{theorem}

  In the sequel we shall write $\upsilon_{G}^{\prime }$ for $p^{-1}\upsilon
_{G}$.

\begin{proof} Suppose first that $G$ is a $p$-group; the result then follows
immediately from Theorem \ref{thm22}  and \ref{ss6b}.

Suppose next that $G$ is a $\Q_{p}$-$p$-elementary group. We adopt
the notation of \S \ref{s5}: We write $G=C\rtimes P$ and we have the
decomposition
\begin{equation*}
R[G] =\oplus_{m}R[m] \circ P.
\end{equation*}%
We consider the restriction map $\mathrm{res}:R[m] \circ
P\rightarrow M_{ \vert A_{m} \vert }( R[m] [
H_{m}] ) $ and form the composite
  \begin{equation*}
R[m] \circ P\rightarrow M_{ \vert A_{m} \vert }( R%
[m] [H_{m}] ) \xrightarrow{\rm Tr} R[m] [H_{m}] ^{A_{m}}\rightarrow (
R[m] [H_{m}] /I_{H_{m}}) ^{A_{m}}=R[m%
] [C_{H_m}] ^{A_{m}}
\end{equation*}
where the map $R[ m] [ H_{m}] ^{A_{m}}\rightarrow
( R[ m] [ H_{m}] /I_{H_{m}}) ^{A_{m}}$ is
induced by the natural map $R[ m] [ H_{m}] \rightarrow
R[ m] [ H_{m}] /I_{H_{m}}$, where of course $A_{m}$
acts on both $H_{m}$ and $R[ m]$.
Recall that here $P$ acts on $R[m] \circ P$ by conjugation on $P$
and its natural action on $R[m]$.
Since $A_{m}$ is cyclic, we
know $[P,P]\subset H_{m}$, so for any $g$, $k\in P$ we have $g^{k}-g$ lies in
the $R[G] $-ideal generated by $I_{H_{m}};$ hence $(\mathrm{Tr}
\circ \mathrm{res})( I_{P}) \subset I_{H_{m}}$ and we have
constructed a map
\begin{equation*}
H_{0}( Tr) :H_{0}( P,R[m] \circ P)
\rightarrow H_{0}( H_{m},R[m] [H_{m}] )
^{A_{m}}.
\end{equation*}
Observe that $R[m] [C_{H_{m}}] ^{A_{m}}$is spanned
over $R[m] ^{A_{m}}$ by the elements $(\mathrm{Tr}\circ {\rm res} )
( \zeta h) $ for $h\in H_{m}$ and $\zeta $ a primitive $m$-th
root of unity; this shows that $H_{0}( Tr) $ is surjective.

\begin{lemma}\label{le66}
We have the commutative diagram:%
\begin{equation*}
\begin{array}{ccccc}
\mathrm{Det}( R[G] ^{\times }) & \xrightarrow{\sim} & \oplus _{m}%
\mathrm{Det}( R[m] \circ P^{\times }) &  \xrightarrow{\sim}  &
\oplus_{m}\mathrm{Det}( R[m] [H_{m}] ^{\times
}) ^{A_{m}} \\
\downarrow \upsilon _{G}^{\prime } &  & \downarrow &  & \downarrow \oplus
\upsilon _{m}^{A_{m}} \\
H_{0}( G,R[G] ) &  \xrightarrow{\sim} & \oplus _{m}H_{0}( G,R
[m] \circ P) &
\xrightarrow{H_0(Tr)} & \oplus _{m}H_{0}( H_{m},R[m] [H_{m}%
] ) ^{A_{m}}%
\end{array}%
\end{equation*}%
and $H_{0}( Tr) $ is an isomorphism.
\end{lemma}

\begin{proof} The diagram comes from Theorem \ref{thm54}.
 We have shown $H_{0}( Tr) $ to be surjective, and
the terms in the lower row are $R$-torsion free with the same rank (as is
seen by tensoring by $N$), and so $H_{0}( Tr)$ is injective.
\end{proof}

\begin{lemma}\label{le67} We have
\begin{equation*}
\oplus _{m}H_{0}( H_{m},R[m] [H_{m}] )
^{A_{m}}\cong \oplus _{m}R[m] [C_{H_{m}}] ^{A_{m}}\cong R[
C_{G}] .
\end{equation*}
\end{lemma}

\begin{proof} The result is clear if $C=\{1\}$, so suppose that $C$ is non-trivial
and choose a prime number $l$, that divides $|C| $ and
write $|C| =l^{r}q$ with $q$ coprime to $l$. Let $c\in C$
denote an element in $C$ of order $l$ and set $\overline{G}=G/C$.  We can
then write $R[C_{G}] =R[C_{\overline{G}}] \oplus
( 1-c) R[C_{G}] $ where $( 1-c) R[
C_{G}] $ denotes $\ker ( R[C_{G}] \rightarrow R[
C_{\overline{G}}] )$. We then have the decompositions
\begin{equation*}
( 1-c) R[C] =\oplus _{m}R[m] ,\quad (
1-c) R[G] =\oplus _{m}R[m] \circ P
\end{equation*}%
where$\ m$ runs through all the divisors of $|C|$ which
are divisible by $l^{r}$.
\begin{eqnarray*}
H_{0}( G,( 1-c) R[G] ) &=&H_{0}(
G,\oplus _{m}R[m] \circ P) =\oplus _{m}H_{0}( G,R%
[m] \circ P) \\
&=&\oplus _{m}H_{0}( P,R[m] \circ P)  \overset{H_0(Tr)}{\cong}\oplus _{m}H_{0}( H_{m},R[m]
[H_{m}] ) ^{A_{m}}.
\end{eqnarray*}
The result then follows by induction on the group order (which gives the result for $\overline{G}$).
\end{proof}
 \smallskip

To complete the proof of Theorem \ref{thm65} for arbitrary $G$ we recall from
Theorem \ref{thm49} that we can write $m\cdot 1_{G}=\sum\nolimits_{H}n_{H}\mathrm{Ind}%
_{H}^{G}( \theta _{H}) $ where $n_{H}$ are integers, the $H$ are $\Q_{p}$-$p$-elementary subgroups of $G$ and the $\theta _{H}$ are $\Q_{p}$-characters of $H$ and $m$ is not divisible by $p$.

Let $\mathrm{Det}( x) \in  \mathrm{Det}( R[G]
^{\times })$. Using Lemma \ref{le64} we have
\begin{eqnarray*}
m\cdot \upsilon_{G}^{\prime }( \mathrm{Det}( x) )
&=&\sum\nolimits_{H}n_{H}\cdot \upsilon_{G}^{\prime }( \mathrm{Ind}%
_{H}^{G}( ( \theta _{H}) \mathrm{Det}( x) )
) \\
&=&\sum\nolimits_{H}n_{H}\cdot \upsilon_{G}^{\prime }( \mathrm{Ind}%
_{H}^{G}( \theta _{H}( \text{\textrm{Res}}_{H}^{G}( \mathrm{%
Det}( x) ) ) ) ) \\
&=&\sum\nolimits_{H}n_{H}\cdot \mathrm{Ind}_{H}^{G}( \upsilon_{H}^{\prime
}( \theta _{H}( \text{\textrm{Res}}_{H}^{G}( \mathrm{Det}%
( x) ) ) ) )
\end{eqnarray*}%
and so by the result for $\Q_{p}$-$p$-elementary groups we know that
$m\cdot \upsilon_{G}^{\prime }( \mathrm{Det}( x) ) \in pR[
C_{G}]$. The result is then shown since $m$ acts as an automorphism
on $R[C_{G}]$.
\end{proof}

% ----------------------------------------------------------------------
\subsection{Oliver's map $\Theta _{R[G]} $ for arbitrary finite groups}\label{6e}

In this section, among other results, we provide a more direct construction of Oliver's isomorphism
$\Theta_{R[G]}$ of Theorem \ref{thm7}.
(Recall that $R$  satisfies the additional
hypotheses stated in the beginning of \S \ref{ss6c}.)

For a $\Z[G] $-module $M$ we recall the bar
resolution%
\begin{equation*}
\cdots \rightarrow \Z[G] \otimes _{\Z}%
\Z[G] \otimes _{\Z}M\overset{\partial
_{2}}{\rightarrow }\Z[G] \otimes _{\Z}M%
\overset{\partial _{1}}{\rightarrow }M
\end{equation*}%
which computes the homology groups $H_{\ast }( G,M)$. (See for
instance 13.2 in [HS].) Explicitly for $m\in M$, and $g$, $h\in G$ we have
\begin{eqnarray*}
\partial _{1}( g\otimes m) &=&( 1-g) m \\
\partial _{2}( g\otimes h\otimes m) &=&h\otimes m-gh\otimes
m+g\otimes hm.
\end{eqnarray*}%
In what follows an element of $H_{1}( G,M) $ will be represented
as an element of the quotient $\ker \partial _{1}/{\rm Im}\, \partial _{2}$
without further reference.

\medskip

Let $G_{r}$ again denote the subset of $p$-regular elements of $G$; i.e. the
subset of elements of $G$ whose order is coprime to $p.$ We shall be
particularly interested in the $R[G] $-module $R[G_{r}]$ where $G$ acts on the left on $G_{r}$ by conjugation; that is to
say $^{g}h=ghg^{-1}$ for $g\in G,\ h\in G_{r}.$

We shall view $R[G_{r}] $ as a $\Psi $-module (as in the
Introduction) by the rule
\begin{equation*}
\Psi ( \sum\nolimits_{g}s_{g}g) =\sum\nolimits_{g}F(
s_{g}) g^{p}
\end{equation*}%
for $s_{g}\in R$, $g\in G_{r} $. Note that the actions of $G$ and $\Psi $
commute, so that we may view $R[G_{r}] $ as a $G\times \Psi $%
-module. We write $H_{1}( G,R[G_{r}] ) ^{\Psi }$ and $%
H_{1}( G,R[G_{r}] ) _{\Psi }$ for the groups of $\Psi
$-invariants and $\Psi $-covariants of $H_{1}( G,R[G_{r}]
) $ in the usual way.

\subsubsection{The map $\omega _{G}$}\label{sss6e1}

See Theorem 12.9 in [O5].   We also view $R[G] $ as a left $G$-module
by conjugation; so that $^{g}( \sum\nolimits_{h}s_{h}h)
=\sum\nolimits_{h}s_{h}.ghg^{-1}$. We may then identify $H_{0}( G,R%
[G] ) $ with $R[C_{G}]$. Recall that we have
previously defined the $R$-linear map $\phi :R[G] \rightarrow R%
[C_{G}] $ by mapping each group element to its conjugacy class.
Recall also that each group element $g\in G$ may be written uniquely as $%
g=g_{r}g_{p}$ where $g_{r}$ has order coprime to $p$, $g_{p}$ has $p$-power
order, and $g_{r}$ and $g_{p}$ commute.

\begin{proposition}\label{pro68}
The map
\begin{equation*}
\omega_G ( \sum\nolimits_{g}s_{g}g) =\sum\nolimits_{g}g\otimes
s_{g}g_{r}
\end{equation*}%
induces a homomorphism
\begin{equation*}
\omega_G :H_{0}( G,R[G] ) \rightarrow H_{1}( G,R%
[G_{r}] ) .
\end{equation*}
\end{proposition}

\begin{proof} For simplicity, write $\omega$ instead of $\omega_G$.
 Note that $\partial _{1}( \omega ( g) )
=\partial _{1}( g\otimes g_{r}) =g_{r}-gg_{r}g^{-1}=0,$ since $g$
and $g_{r}$ commute. $\ $To see that we get a well-defined homomorphism
(that is to say which is independent of choices), we note that
\begin{equation*}
\omega ( ^{g}h) =\ ^{g}h\otimes {}^{g}h_{r}=\ ^{g}\omega (
h)
\end{equation*}
and $^{g}\omega ( h) $ and $\omega ( h) $ are
homologous  since the functorial action of $G$ on homology groups is
trivial (see for instance Ch. III Proposition 8.3 in [B]).
\end{proof}

\begin{remark} {\rm For future reference note that if $G$ is $p$-group, then
$G_{r}=\{1\}$, $H_{1} ( G,R [ G_{r} ]  ) =G^{\mathrm{ab}}\otimes R$, and $\omega _{G}$ is the $R$-linear map induced by
$G\rightarrow G^{\mathrm{ab}}\otimes 1$.}
\end{remark}

\subsubsection{The map $ \xi _{G}$}\label{sss6e2}

\begin{proposition}
For $u=\sum\nolimits_{g\in G}{}s_{g}g\in \GL_{n}( R[G]
) ,$ with $s_{g}\in {\rm M}_{n}( R)$, we write $u^{-1}=\sum\nolimits_{h\in G}{}t_{h}h$ and set
\begin{equation*}
{\xi'_G} ( u) =\sum\nolimits_{g,h}{}g\otimes \mathrm{Tr}(
t_{h}s_{g}) ( hg) _{r}\in R[G_{r}] .
\end{equation*}%
This induces a homomorphism
$$
\xi_G:\Kr_{1}( R[G] )
\rightarrow H_{1}( G,R[G_{r}] )
$$
 and the element
\begin{equation*}
{ \xi_{G}} ( u) =\sum\nolimits_{g,h}{}( g-1)
\otimes \mathrm{Tr}( t_{h}s_{g}) ( hg) _{r}\in R[
G_{r}]
\end{equation*}
lies in the same homology class as $\xi'_G ( u) $ in $H_{1}( G,R%
[G_{r}] ) $.
\end{proposition}

\begin{proof} (See Theorem 12.9.i in [O5].) We first show that $\xi'=\xi'_G$ is a
homomorphism from $\GL_{n}( R[G] ) $ into $H_{1}(
G,R[G_{r}] )$. To this end we first note that
\begin{equation*}
1=\sum\nolimits_{g,h}{}t_{h}s_{g}hg=\sum\nolimits_{g,h}{}s_{g}t_{h}gh
\end{equation*}%
and so
\begin{equation*}
n=\sum\nolimits_{g,h}{}{\rm Tr}( t_{h}s_{g})
hg=\sum\nolimits_{g,h}{}{\rm Tr}( s_{g}t_{h}) gh
\end{equation*}
since for given $x\in G$, $\sum\nolimits_{hg=x}{}\mathrm{Tr}
( t_{h}s_{g})=0$ unless $x=1$, in which case of course we have $h=g^{-1}$.
We then see that
\begin{equation*}
\partial _{1}\circ \xi'( u) =\sum\nolimits_{g,h}{}\mathrm{Tr}%
( t_{h}s_{g}) ( hg) _{r}-\mathrm{Tr}(
t_{h}s_{g}) ( ghgg^{-1}) _{r}=0\text{ .}
\end{equation*}
To see that $\xi' $ is a homomorphism, consider $v=\sum\nolimits_{a\in
G}{}p_{a}a\in \GL_{n}( R[G] ) $ and write $v^{-1}=\sum\nolimits_{b\in G}{}q_{b}b$. We then see that
\begin{equation*}
\xi' ( uv) =\sum ga\otimes \mathrm{Tr}(
q_{b}t_{h}s_{g}p_{a}) ( bhga) _{r}.
\end{equation*}%
Using the fact that $gh\otimes x-h\otimes x-g\otimes hxh^{-1}$ lies in $%
\mbox{Im}\partial _{2},\ $we get the congruence modulo ${\rm Im}\,\partial
_{2}$
\begin{eqnarray*}
\xi '( uv) &\equiv &\sum_{a,b}a\otimes \sum_{h,g}\mathrm{Tr}%
( q_{b}t_{h}s_{g}p_{a}) ( bhga) _{r}+ \\
&&\sum_{h,g}g\otimes \sum_{a,b}\mathrm{Tr}(
q_{b}t_{h}s_{g}p_{a}) ( abhg) _{r}
\end{eqnarray*}%
and using the equalities
\begin{eqnarray*}
\sum_{h,g}\mathrm{Tr}( q_{b}t_{h}s_{g}p_{a}) ( bhga)
_{r} &=&\mathrm{Tr}( q_{b}p_{a}) ( ba) _{r} \\
\sum_{a,b}\mathrm{Tr}( p_{a}q_{b}t_{h}s_{g}) ( abhg)
_{r} &=&\mathrm{Tr}( s_{h}r_{g}) ( hg) _{r}
\end{eqnarray*}
we find $\xi' ( uv) =\xi'( u) +\xi'(v)$ in
$H_{1}( G,R[G_{r}] )$. Next observe that, since $\xi' $ is homomorphism into an abelian group, it factors through $\Kr_{1}( R [
G] ) $.

To conclude we note that setting $h=1$ in $gh\otimes x=h\otimes x+g\otimes
hxh^{-1}$ we get that $\xi ( u) -\xi ^{\prime }(u)$ is in ${\rm Im}\,\partial _{2}$.
\end{proof}
\medskip

The following result shows how the above map $\xi_G $ relates to the $p$-group
logarithmic differential map in 3.3 of [CPT1].

\begin{lemma}\label{le70}
  Suppose that $ G$ is a $p$-group, so that $G_{r}=\{1\}$. Then, under the
identification
\begin{equation*}
\Omega _{R[G] /R}^{1}=G^{\rm ab}\otimes _{\Z}R=\frac{I_{G}}{%
I_{G}^{2}}\otimes _{\Z}R=H_{1}( G,R) ,
\end{equation*}
we have the equality $\xi_G ( u) =d\log ( u) $ for $u\in
R[G] ^{\times }$.
Let $\Psi $ denote the $F$-semi-linear map on $\Omega _{R [ G ]
/R}^{1}$ (so that for $x\in I_{G}/I_{G}^{2},$ $r\in R$, we have $\Psi
 ( r dx ) =F ( r ) \Psi ( dx ) ) $ and which has
the property that $p\Psi \circ d=d\circ \Psi $ (see Proposition 3.18 in
[CPT1]). Then we have the equality  for $u\in R[ G] ^{\times }$
\begin{equation*}
 ( 1-\Psi  ) \xi_{G}(u)  =d \left( p^{-1}
\mathcal{L}_{G}(u)\right) =d( \upsilon _{G}^{\prime
} ( \mathrm{Det} ( u )  )  ) =\omega _{G}\circ
\upsilon _{G}^{\prime }(\mathrm{Det}(u)) .
\end{equation*}

\end{lemma}

\begin{proof} Let $u=\sum\nolimits_{g}s_{g}g$, $u^{-1}=\sum\nolimits_{h}t_{h}h$.
Since $( g-1) h\equiv ( g-1) \mod I_{G}^{2}$,
we obtain the congruence $\mod I_{G}^{2}$:
\begin{equation*}
 {du}\cdot {u}^{-1}=\sum\nolimits_{g,h}( g-1) h\otimes s_{g}t_{h}\equiv
\sum\nolimits_{g,h}( g-1) \otimes s_{g}t_{h}\equiv \xi_{G} (
u)
\end{equation*}
and also
\begin{eqnarray*}
 ( 1-\Psi  ) \xi _{G}(u)=( 1-\Psi ) d\log
\left( u\right) =\ \ \ \ \ \ \ \ \ \\
\ \ \ \ \ \ \ \ \ \ \ =d( ( 1-\Psi /p) \log  ( u )
) =d( p^{-1}\mathcal{L}_{G} ( u )  ) =d(
\upsilon _{G}^{\prime }( \mathrm{Det}( u))) .
\end{eqnarray*}
\end{proof}

For this short paragraph $G$ is still a $p$-group so that $G_r=\{1\}$. As in the final part of \ref{ss6b}
we can decompose the map  $\upsilon _{G}^{\prime }=\upsilon _{\varepsilon
}\times \upsilon _{I}\ $via the decomposition $\mathrm{Det} ( R [ G%
 ] ^{\times } ) =R^{\times }\times \mathrm{Det} (
1+I_{G} ) $ and note that $\upsilon _{\varepsilon }$ coincides with the
map $p^{-1}\mathcal{L}_{R}$ of \ref{ss6b}.  Using the above and noting that $H_{1}(
G,R) =R\otimes G^{\rm ab}$ we see that the exact sequence
\begin{equation*}
\mathrm{Det} ( 1+I_{G} ) \xrightarrow{\upsilon _{G}^{\prime }} \phi  ( I_{G}) \xrightarrow{\omega _{G}}
\frac{R}{( 1-F) R}\otimes G^{\rm ab}\rightarrow 0
\end{equation*}%
of (\ref{eq3.11}) affords the further exact sequence
\begin{equation*}
\mathrm{Det} ( 1+I_{G} ) \xrightarrow{\upsilon _{G}^{\prime }\times \xi
_{G}} \phi ( I_{G}) \oplus H_{1}( G, R)
\xrightarrow{\omega _{G}+ (\Psi -1)} H_{1} ( G,R )
\rightarrow 0.
\end{equation*}%
By adding this to the exact sequence
\begin{equation*}
R^{\times }\xrightarrow{p^{-1}\mathcal{L}_{R}\oplus \theta }
R\oplus  {\Lambda } ( R ) \xrightarrow{T}
{R}\rightarrow 0
\end{equation*}%
of (\ref{newExseq6.3}) we get the further exact sequence
\begin{equation}\label{eqExact1}
\mathrm{Det} ( R [ G ] ^{\times } ) \xrightarrow{s_{G}} R[ C_{G}] \oplus H_{1}( G,R) \oplus
 {\Lambda }( R) \xrightarrow{T_{G}}H_{1}(
G,R) \oplus  {R}\rightarrow 0
\end{equation}
where $s_{G}( \mathrm{Det}( z))=( \upsilon
_{G}^{\prime }\times \xi _{G}\times \theta _{G} ) $ and for $ (
x,y,z ) = (  x_{\epsilon }\oplus x_{I}, y_{I}, z_{\varepsilon
} ) $ we have $T_{G} ( x,y,z ) = ( \omega _{G} (
x_{I}) + ( \Psi -1 ) y  ,$ $x_{\varepsilon }-\upsilon'
_{\varepsilon }  ( z ) )$.

\begin{lemma}\label{le71}
 For an arbitrary finite group $G$ let $J=J_{R[G]}$ denote the Jacobson radical of $R[G]$. Then $\SK_{1}( R[G] )$
is contained in $\Kr^{\prime}_{1}( R[G],J  )$. If $G$ is abelian, then $\SK_{1}( R[G] ) =\{1\}$.
\end{lemma}

\begin{proof}
Consider the Wedderburn decomposition
\begin{equation*}
\Z_{p}[G]/J=\prod\nolimits_{i}{\rm M}_{n_{i}}(
k_{i})
\end{equation*}%
where each $k_{i}\ $is a finite field of characteristic $p$ and let $W_{i}$
denote the ring of Witt vectors of $k_{i}$. By the Standing Hypotheses
we know that $\SK_{1}( R\otimes_{\Z_p} W_{i}) =\{1\}$. We write $%
\overline{R}=R/pR$ which is a domain. Observe that we have the commutative diagram
\begin{equation*}
\begin{array}{ccccccccc}
&  &  &  & \Kr_{1}( R\otimes_{\Z_p} W_{i},( p) ) & \rightarrow
& 1+pR\otimes_{\Z_p} W_{i} &  &  \\
&  &  &  & \downarrow &  & \downarrow &  &  \\
1 & \rightarrow & \SK_{1}( R\otimes_{\Z_p} W_{i}) & \rightarrow &
\Kr_{1}( R\otimes_{\Z_p}  W_{i}) & \overset{\det }{\rightarrow } & (
R\otimes_{\Z_p}  W_{i})^{\times } & \rightarrow & 1 \\
&  & \downarrow &  & \downarrow &  & \downarrow &  &  \\
1 & \rightarrow & \SK_{1}( \overline{R}\otimes W_{i}) &
\rightarrow & \Kr_{1}( \overline{R}\otimes W_{i}) & \overset{\det }%
{\rightarrow } & ( \overline{R}\otimes W_{i}) ^{\times } &
\rightarrow & 1 \\
&  &  &  & \downarrow &  & \downarrow &  &  \\
&  &  &  & 1 &  & 1 &  &
\end{array}
\end{equation*}
Since $1+pR\otimes_{\Z_p}  W_{i}$ injects into $( R\otimes_{\Z_p}  W_{i}) ^{\times
}$ and since $\Kr_{1}( R\otimes_{\Z_p} W_{i}) $ surjects onto $\Kr_{1}(
\overline{R}\otimes_{\Z_p}  W_{i})$, because $p$ is in the Jacobson radical
of $R\otimes_{\Z_p}  W_{i}$, we see that $\SK_{1}( R\otimes W_{i})
\rightarrow \SK_{1}( \overline{R}\otimes_{\Z_p}  W_{i}) $ maps onto and so
$\SK_{1}( \overline{R}\otimes_{\Z_p}    W_{i}) =\{1\}$. This then shows $%
\SK_{1}( R[G] ) $ is contained in ${\rm Im}(
\GL( R[G] ,J) )$, as required.

 Now suppose that $G$ is abelian. By Lemma \ref{le14} (b) we know that if we have $x\in \GL( R[G] ,J) $
with $\mathrm{Det}( x) =1$, then $x$ is elementary. The result then follows since by the above
$\SK_{1}( R[G] ) \subset {\rm Im}(
\GL( R[G] ,J) )$.
\end{proof}

\begin{lemma}\label{le72}
For an arbitrary finite group $G$, $\xi_G $ is trivial on $\SK_{1}( R%
[G] )$. We shall therefore henceforth feel free to view $\xi_G $ as being defined on $\mathrm{Det}( R[G] ^{\times })$.
\end{lemma}

\begin{proof} We know from Theorem \ref{thm52} that $\SK_{1}( R[G]
) $ is computable by induction from $\Q_{p}$-$p$-elementary
groups. Thus, by Theorem \ref{thm58}, it will suffice to prove the result
for $p$-groups. However, from Lemma \ref{le70} we know that $\xi ( \SK_{1}(
R[G] ) ) =\xi ( \SK_{1}( R[G^{\rm ab}%
] ) ) $ and Lemma \ref{le71} above shows that $\SK_{1}( R[
G^{\rm ab}] ) =1$.
\end{proof}
\medskip

\begin{lemma}\label{leNew} Suppose we have a surjection of finite groups $\alpha: \widetilde{G}\xrightarrow{ }G$.
Then $\SK_{1}(R[G])
\subset \alpha _{\ast }(\Kr_{1}(R[ \widetilde{G}]))$.
\end{lemma}

\begin{proof} By Lemma \ref{le71} we have $\SK_{1}(R[G]) \subset
\Kr_{1}^{\prime } ( R[G], J_{R[G]}) $ where
$J_{R[ G] }$ denotes the Jacobson radical of $R[G]$.
We claim next that $\alpha ( J_{R[ \widetilde{G}] })
=J_{R[G] }$. We have $J_{R[G]}=(J_R, J_{\Z_p[G]})$
and similarly for $J_{R[\widetilde G]}$; thus it will
suffice to show that $\alpha(J_{\Z_p[\widetilde G]})= J_{\Z_p[G]}$
and therefore it will
suffice to show $\alpha(J_{\bF_p[\widetilde G]})= J_{\bF_p[G]}$.
We know that the quotient
rings $\bF_p[\widetilde G]/J_{\bF_p[\widetilde G]}$ and
$\bF_p[G]/J_{\bF_p[G]}$
  are semisimple with the former mapping
onto the latter with kernel a two-sided ideal. Thus we have a decomposition
\begin{equation*}
\bF_p[\widetilde G]/J_{\bF_p[\widetilde G]}=X\oplus (\bF_p[G]/J_{\bF_p[G]})
\end{equation*}%
which by lifting of idempotents lifts to a splitting
\[
\bF_{p} [ \widetilde{G} ] =X^{\prime }\oplus Y
\]%
where $Y/J_{Y}\cong \bF_{p}[G]/J_{\bF_{p} [ G] }$; and so we see that indeed $J_{\bF_p[\widetilde G]}$
maps onto $J_{\bF_p[ G]}$.  Finally we conclude since
\begin{equation*}
{\rm Mat}(R[\widetilde G] , J_{R[\widetilde G]})=\GL
 ( R [ \widetilde G ]
, J_{R [\widetilde  G ] } ),\quad  {\rm Mat}(R[ G] , J_{R[ G]})=\GL
 ( R [ G ]
, J_{R [ G ] } )
\end{equation*}
and by the above ${\rm Mat}(R[\widetilde G] , J_{R[\widetilde G]})$ maps
onto $ {\rm Mat}(R[ G] , J_{R[ G]})$.
\end{proof}

\subsubsection{The map $\theta _{G}$}\label{sss6e3}

First we need:

\begin{lemma}\label{le73}
Let $\mathcal{E}_{p}( G) $ denote the set of $\Q_{p}$-$p$-elementary subgroups of $G$.
 Then

(a)  $\varinjlim_{ H\in \mathcal{E}_{p}( G) }H_{0}( H,R%
[H] ) =H_{0}( G,R[G] )$;

(b)  $\varinjlim_{ H\in \mathcal{E}_{p}( G) }H_{1}(
H,R[H_{r}] ) =H_{1}( G,R[G_{r}] )$.

(c) Suppose that for each subgroup $H$ of $G$, we are given a subgroup  $M( H) $
 of $\mathrm{Det}( R[H] ^{\times }) $ such that
 the $M( H)$ are natural with respect to inclusion, so that we have
${\rm Ind}_{H}^{G}( M( H) ) \subset M(
G)$. Suppose further that for each $H$ the quotient group $%
\mathrm{Det}( R[H] ^{\times }) /M( H) $ is
a $\Z_{p}$-module
and   that
$\mathrm{Det}( R[G] ^{\times }) /M(G) $ is torsion free; then the natural map
\begin{equation*}
\varinjlim_{ H\in \mathcal{E}_{p}( G) }\mathrm{Det}( R%
[H] ^{\times }) /M( H) \rightarrow \mathrm{Det}%
( R[G] ^{\times }) /M( G)
\end{equation*}%
is surjective.
\end{lemma}

\begin{proof} Let $d\in \mathrm{Det}( R[G] ^{\times })$.
To see that (c) holds observe that by \ref{ss4b}
for some $m$ coprime to $p$ we
can write
\begin{equation*}
m\cdot d=\sum_{H \in  \mathcal{E}_{p}( G) }n_{H}\mathrm{Ind}%
_{H}^{G}( \theta _{H})  d=\sum_{H \in \mathcal{E}%
_{p}( G) }n_{H}\mathrm{Ind}_{H}^{G}( \theta _{H}\text{%
\textrm{Res}}_{H}^{G}( d) )
\end{equation*}
which belongs to $ \sum_{H  \in
\mathcal{E}_{p}( G) }\mathrm{Ind}_{H}^{G}( \mathrm{Det}%
( R[H] ^{\times }) )$.
The result then follows since all the quotient groups $\mathrm{Det}( R[H]
^{\times }) /M( H) $ are $\Z_{p}$-modules where $m$
acts as an automorphism.

For (a) and (b) by Theorem \ref{thm50} it suffices to show that the right-hand terms
are generated over $\mathcal{E}_{p}( G)$. This holds  for (a)
because $\mathcal{E}_{p}( G) $ contains all cyclic
subgroups of $G$. To see that (b) holds, note that for $g\in G_{r}$ we have
\begin{equation*}
H_{1}( G,\sum\nolimits_{h\in G}Rg^{h}) =H_{1}( Z_{G}(
g) ,Rg) =H_{1}( Z_{G}( g) _{p},Rg)
\end{equation*}%
where $Z_{G}( g) _{p}$ denotes a $p$-Sylow subgroup of the centralizer $Z_{G}( g)$. The result then follows since $Z_{G}( g)
_{p} \cdot \langle g \rangle \in $ $\mathcal{E}_{p}( G)$. \end{proof}

 \begin{lemma}\label{lem74}
For any finite group $G$ we have that $\mathrm{Det}( R[G]
^{\times }) /\ker ( \upsilon _{G}^{\prime }\times \xi _{G})
$  is a $\Z_{p}$-module.
\end{lemma}

\begin{proof} Since ${\rm Im}( \xi_G ) $ is a finite abelian $p$-group,
it will suffice to show that the quotient $\mathrm{Det}( R[G] ^{\times
}) /\ker ( \upsilon _{G}^{\prime}) $ is a $\Z_{p}$%
-module. This follows for $p$-groups by (\ref{newExseq6.3}) and (\ref{eq3.11}). The result for $%
\Q_{p}$-$p$-elementary groups then follows from Theorem \ref{thm54}. The
result for general finite groups follows from (c) above using Theorem \ref{thm65} to see that
$\mathrm{Det}( R[G] ^{\times
}) /\ker ( \upsilon _{G}^{\prime}) $ is torsion free.
\end{proof}
\smallskip

We now define a group ${\Lambda } ( R [ G ]  ) $
and  maps $\theta _{G}$ and $\widetilde{\mathcal{L}}_{R [ G] }$;
these appear in the statement of Proposition \ref{pro75}.

Let $d:\G_{0}(\Q_p[G])\rightarrow \G_{0}(\mathbb{F}_p[G])$ denote the decomposition
map, which is a surjective $\lambda $-ring homomorphism which admits a
natural splitting. Thus the isomorphism classes of $\mathbb{F}_{p} [ G%
 ] $-modules are identified with Brauer modular characters in
characteristic zero; so, in the sequel, for $\xi \in G_{0} ( \mathbb{Q}%
_{p}[G])) $ we write $d ( \xi  ) $ both for the
image of $\xi $ in $\G_{0} ( \mathbb{F}_{p}[G]) $ and
for its Brauer lift in $\xi \in \G_{0}(\Q_p[G]) $. Note also that for $\theta \in \G_{0}(\Q_p[G]) $ we have the equality $d ( \xi  ) \cdot\theta =d (
\xi  ) \cdot( d ( \theta  ) + ( \theta -d ( \theta
)))=d ( \xi \cdot\theta ) .$  We define $\mathcal{%
M} ( R [ G ] ) $ as
\begin{equation*}
\{\mathrm{Det} ( x ) \in \mathrm{Det}(R[G]
^{\times } ) \mid \ \Psi  ( \mathrm{Det}(x))
\left( d\left( \xi \right) \right) =\mathrm{Det}( x )  (
d ( \xi  ) ) ^{p}\ \text{for all }\xi \in \G_{0}(\Q_p[G])\}.
\end{equation*}%
We assert that the $\G_{0}(\Q_p[G])$-Green structure of $\mathrm{Det} ( R\left[ G\right] ^{\times } ) $
makes $\mathcal{M} ( R[ G ] ) $ a $\G_{0}(\Q_p[G])$-Green submodule: indeed, for $\theta \in
\G_{0}(\Q_p[G])$ recall we have $\theta \cdot
\mathrm{Det}( x) ( \xi ) =\mathrm{Det}( x)( \xi \overline{\theta })$, and so for $\mathrm{Det}(
x) \in \mathcal{M}(R[G])$ we have
\begin{eqnarray*}
\Psi ( \theta \cdot\mathrm{Det}( x) ) ( d( \xi
) )  &=&F( ( \theta \cdot\mathrm{Det}( x)
) ) ( \psi ^{p}d( \xi ) ) =F( (
\theta \cdot\mathrm{Det}( x) ) ) ( d( \psi
^{p}\xi ) )  \\
&=&F( \mathrm{Det}( x) ) ( d( \psi ^{p}\xi) \cdot F( \overline{\theta })  ) =F( \mathrm{Det}%
( x) ) ( d( \psi ^{p}\xi ) \cdot dF(
\overline{\theta }) )  \\
&=&F( \mathrm{Det}( x) ) ( d( \psi ^{p}\xi)\cdot \psi ^{p}d ( \overline{\theta }) ) =F( \mathrm{%
Det}( x) ) ( \psi ^{p}d( \xi \cdot\overline{\theta }) )  \\
&=&\Psi ( \mathrm{Det}( x) ) ( d( \xi
\overline{\theta } )  ) =\mathrm{Det} ( x )  (
d ( \xi \overline{\theta } )  ) ^{p}= ( \theta \cdot\mathrm{Det%
}( x) ( d( \xi ) )) ^{p}.
\end{eqnarray*}%
For future reference we note that $\mathcal{M} ( R [ G ]
 ) \supset \ker d\cdot\mathrm{Det} ( R [ G ] ^{\times } )$.

We define:
\begin{eqnarray*}
\mathrm{Det}^{\prime } ( R [ G ] ^{\times } ) &=&\mathrm{%
Det}^{\prime } ( R [ G ] ^{\times } ) \otimes
_{\G_{0} ( \mathbb{Q}_{p}[ G] ) }\G_{0}( \mathbb{F}%
_{p}[ G] ) =\frac{\mathrm{Det}( R[ G]
^{\times }) }{\ker d\cdot \mathrm{Det} ( R [ G ] ^{\times
} ) } \\
\ \mathcal{M}^{\prime } ( R [ G ]  ) &=&\frac{\ \mathcal{M}%
 ( R [ G ]  ) }{\ker d\cdot \mathrm{Det} ( R [ G]
^{\times } ) }
\end{eqnarray*}%
and we set:
\begin{equation*}
\Lambda ( R [ G ]  ) =\mathrm{Det}^{\prime } ( R[
G ] ^{\times } ) /\mathcal{M}^{\prime }( R[ G]
 ) =\mathrm{Det}( R[ G] ^{\times }) /\mathcal{M}( R[ G] ) .
\end{equation*}%
The map $\theta _{G}$ is defined to be the composite
\begin{equation*}
\theta _{G}:\mathrm{Det}(R[G]^{\times } )
\rightarrow \mathrm{Det}^{\prime }(R[G]^\times)
\rightarrow \Lambda(R[G]).
\end{equation*}%
We now define
\begin{equation*}
\widetilde{\mathcal{L}}_{R [ G ] }:\Lambda ( R [ G ]
 ) \rightarrow \prod\nolimits_{i}R\otimes_{\Z_p} W_{i}=H_{0} ( G, R [
G_{r} ]  )
\end{equation*}%
where the product extends over the irreducible modular characters of $%
\G_{0} ( \mathbb{F}_{p} [ G ]  )$, and where $\widetilde{\
\mathcal{L}}_{R [ G ] }$ is the product of the logarithmic maps
\begin{equation*}
\frac{1}{p}\mathcal{L}_{R\otimes_{\Z_p} W_{i}}: ( R\otimes_{\Z_p} W_{i} )
^{\times }\rightarrow R\otimes_{\Z_p} W_{i}
\end{equation*}%
followed by identification of the free $R$-module on the $p$-regular
conjugacy classes of $G$, namely $H_{0} ( G, R [ G_{r} ])$, with $\prod\nolimits_{i}R\otimes_{\Z_p} W_{i}$, by evaluation on the
irreducible modular characters $\chi _{i}$ of $G$ (see Brauer's Theorem in
18.2 of [S] and in particular Exercise 4 in 18.2). To be more precise,
writing $\chi _{i}$ for the irreducible character associated to the $i$-th
component of the above description, explicitly for each $x\in R[ G] ^{\times }$ we may view $\widetilde{\mathcal{L}}_{R[ G]
}\circ \theta _{G}( \mathrm{Det}( x)) $ as being
given by the value%
\begin{equation*}
\prod\nolimits_{i}\frac{1}{p}\mathcal{L}_{R\otimes_{\Z_p} W_{i}}( \mathrm{Det}%
_{\chi _{i}}( x)) .
\end{equation*}%
We view $\mathrm{Det}^{\prime }( R[ G] ^{\times }) $
as a Green module for the ring $\G_{0}( \mathbb{F}_{p}[ G])$, and by the above $\mathcal{M}^{\prime }( R[ G]
) $ is a sub-Green module of $\mathrm{Det}^{\prime } ( R [ G
 ] ^{\times } ) ,$ and so $\Lambda  ( R [ G ]  )
$ is a quotient Green module for $\G_{0} ( \mathbb{F}_{p} [ G ]
 ) .$

For each $H\subset G$, in the usual way we have the induction map $\mathrm{%
Ind}_{H}^{G}:\mathrm{Det}^{\prime } ( R [ H ] ^{\times } )
\rightarrow \mathrm{Det}^{\prime } ( R [ G ] ^{\times } )$ and hence we have a natural map $\Lambda  ( R [ H ] )
\rightarrow \Lambda  ( R [ G ]  )$.

\begin{proposition}\label{PropNss6e3}
For each finite group $G$, the group $ {\Lambda } ( R [ G%
 ]  ) $ is a pro-$p$-group and
\[
\varinjlim_{ H\in \mathcal{E}_{p} ( G ) } {\Lambda }%
 ( R [ H ] ) = {\Lambda } ( R [ G ]) .
\]
\end{proposition}

\begin{proof} The proof proceeds in three steps.

First Step: Note that if $p\nmid \left\vert G\right\vert$, then we have a
decomposition
\[
R [ G ] \cong \prod\nolimits_{i}{\rm M}_{n_{i}} ( R\otimes _{\mathbb{%
Z}_{p}}W_{i} )
\]%
for rings of integers $W_{i}$ of non-ramified extensions of $\mathbb{Q}_{p}$.
The result in this case therefore follows from Hypothesis (2) in \ref{ss6c}.

Second step: Suppose now that $G$ is $\mathbb{Q}_{p}$-$l$-elementary for
some prime number $l\neq p$. We may write $G= ( C^{\prime }\times
C_{p^{m}} ) \rtimes L$ where $L$ is an $l$-group, $C_{p^{m}}$ is a
cyclic group of order $p^{m}$ for some $m\geq 0$ and $C^{\prime }$ is a
cyclic group with order prime to $lp$. If $m=0$, then we are in the situation
dealt with in the first step and so we now suppose that $m>0$. We set $G^{\prime }=C^{\prime }\rtimes L$,
and note that the natural quotient $
G\rightarrow G^{\prime }$ is split.

Every irreducible character of $G$ can be written in the form $\mathrm{Ind}_{H}^{G}(\chi)$ for $\chi $ an abelian character of a subgroup $H$ of $G$
which contains $C^{\prime }\times C_{p^{m}}$; moreover, the irreducible
characters of $G$ which are inflated from $G^{\prime }$ arise precisely from
those $\chi $ which are trivial on $C_{p^{m}}$.

Recall that $\left\langle c\right\rangle =C_{p^{m}}$. We have the split
exact sequence:
\begin{equation*}
1\rightarrow \mathrm{Det}( 1+( 1-c) R[ G] )
\rightarrow \mathrm{Det}( R[ G] ^{\times })
\leftrightarrows \mathrm{Det}( R[ G^{\prime }] ^{\times
}) \rightarrow 1
\end{equation*}%
and the surjection $G\rightarrow G^{\prime }$ induces an isomorphism $%
G_{0}( \mathbb{F}_{p}[ G] ) \cong G_{0}( \mathbb{F%
}_{p}[ G^{\prime }] )$. Thus, from the definition of $%
\mathcal{M}( R[ G] ) ,$ we have the direct
decomposition
\begin{equation*}
\mathcal{M}( R[ G] ) =\mathrm{Det}( 1+(
1-c) R[ G] ) \times \mathcal{M}( R[
G^{\prime }] )
\end{equation*}%
and hence the natural map $\Lambda ( R[ G] )
\rightarrow \Lambda ( R[ G^{\prime }] ) $ is an
isomorphism, and we are now done by Step 1.

Third Step: To prove the final part of the Proposition, we use Brauer's
induction theorem for modular characters  (see Theorem
39 in [S]). By this result we know that for each prime $l\neq p$, we can find
a positive integer $d_{l}$, prime to $l $ such that for each $\lambda \in
 {\Lambda } ( R [ G ] ) $ we know that $%
d_{l}\lambda $ is in the image of induction from $\mathbb{Q}_{p}$-$l$-elementary subgroups of $G$.
 So varying over the primes different $l$ from $p,$ by the above we know that
we can find a non-negative integer $m$ such that $p^{m}\Lambda ( R[
G] ) $ is a finite sum of pro-$p$-groups.

To conclude we use Brauer's induction theorem for $\mathbb{Q}_{p}$-$p$%
-elementary groups expressed in terms of modular characters (by applying the
decomposition map to the standard form of Brauer's theorem) together with
the Green module structure of $\Lambda ( R[ G] ) $
over $G_{0}( \mathbb{F}_{p}[ G] ) $; this then shows
that $\Lambda ( R[ G] ) $ is generated by induction
from $\mathbb{Q}_{p}$-$p$-elementary groups, and so by Theorem \ref{thm50} we
deduce
\[
\varinjlim_{ H\in \mathcal{E}_{p} ( G ) } {\Lambda }%
 ( R [ H] ) = {\Lambda } ( R [ G ]
 ) .
\]\end{proof}

 As previously we set $M(H)={\rm ker}(\upsilon_H'\times \xi_H)$. We conclude this subsection by noting that the maps
 $$
  \upsilon _{H}^{\prime }:\mathrm{Det}( R[ H ]
^{\times }) /M( H) \rightarrow H_{0}( H,R[ H] ) , $$
$$\xi _{H}:\mathrm{Det}( R[ H] ^{\times
}) /M( H) \rightarrow H_{1}( H,R[ H]),
$$
 afford commutative diagrams
\begin{equation*}
\begin{array}{ccc}
\varinjlim_{H\in \mathcal{E}_{p}(G)}\mathrm{Det}( R[ H] ^{\times }) /M( H) & \rightarrow &
\varinjlim_{H\in \mathcal{E}_{p}(G)}H_{0}( H,R[
H] ) \\
\downarrow & \searrow \widehat \upsilon_{G}^{\prime } & \downarrow \\
\mathrm{Det}( R[ G]^{\times }) /M( G) &
\rightarrow & H_{0}( G, R[G] )
\end{array}
\end{equation*}
and
\begin{equation*}
\begin{array}{ccc}
\varinjlim_{H\in \mathcal{E}_{p}(G)}\mathrm{Det} ( R%
 [ H ] ^{\times } ) /M ( H ) & \rightarrow &
\varinjlim_{H\in \mathcal{E}_{p}(G)}H_{1} ( H, R [
H ]  ) \\
\downarrow & \searrow \widehat\xi _{G} & \downarrow \\
\mathrm{Det} ( R [ G ] ^{\times } ) /M ( G ) &
\rightarrow & H_{1} ( G,R [ G ] )
\end{array}%
\end{equation*}%
where we denote the diagonal maps by $\widehat{\upsilon }_{G}^{\prime }$ and $\widehat{\xi }_{G}$ respectively.

\subsubsection{An exact sequence}\label{sss6e4}

Recall that we identify $H_{0}( G,R[G] ) $ and $R%
[C_{G}] $. We define $s_G=\upsilon^{\prime}_G \times \xi_G \times \theta_G$
and we let $\rho_G: H_0(G, R[G])\to H_0(G, R[G_r])$
be the $R$-linear map induced by mapping each group element $g\in G$
to its $p$-regular component $g_r$.

\begin{proposition}\label{pro75}
The following sequence is exact:
\begin{eqnarray}\label{eq6.60}
\ \ \ \ \mathrm{Det}%
( R[G] ^{\times })   \xrightarrow{\ s_G\ } H_{0}( G,R[G] ) \oplus
H_{1}( G,R[G_{r}] ) \oplus  {\Lambda} ( R [
G] )\rightarrow  \\ \xrightarrow{\ T_G\ } H_{1}( G,R[G_{r}%
] ) \oplus H_{0}( G, {R}[G_{r}] )
\rightarrow 0 \ \ \ \ \ \notag
\end{eqnarray}%
where $T_{G}( x\oplus y\oplus z) = (\omega _{G}( x) -(
1-\Psi ) y)\oplus (\rho_G ( x) -  \widetilde{\mathcal{L}}%
_{R[G]}( z) )$.
\end{proposition}

\begin{proof} Since formation of direct limits over $\mathcal{E}_{p}(
G) $ (which need not necessarily form a directed system) is a right exact
functor, by Lemma \ref{le73} we are reduced to showing that (\ref{eq6.60}) is exact when
$G$
is $\Q_{p}$-$p$-elementary.

We again use the notation of Section \ref{s5} and write $G=C\rtimes P$. First note
that, as $H_{m}$ is a $p$-group, clearly $H_{m,r}=\{1\}$. For  ease of notation, we will write $H=H_m$, $A=A_m$ but still retain $R[m]$ so that we are reminded of the
corresponding extension of scalars. Next observe that as in (\ref{eqExact1})
 we have the exact sequence:
\begin{eqnarray*}
\ \ \ \ \mathrm{Det}( R[m] [H] ^{\times })
  \xrightarrow{\ s_{H }\ } H_{0}( H ,R[m]
H ] ) \oplus H_{1}( H ,R[m] )
\oplus  {\Lambda} ( R[m] ) \\
\xrightarrow{\ T_{H}\ }H_{1}( H ,R[m] ) \oplus H_{0}(
H , {R}[m] ) \rightarrow 0 \ \ \ \
\end{eqnarray*}%
which we rewrite as
\begin{eqnarray}\label{eq6.61}
\ \ \ \ \mathrm{Det}( R[m] [H ] ^{\times })
  \xrightarrow{\ s_{H }\ }H_{0}( H ,R[m]  [H ] ) \oplus ( H^{\rm ab}\otimes R[m]
) \oplus  {\Lambda} ( R[m] ) \\ \xrightarrow{\ T_{H }\ }( H^{\rm ab}\otimes R[m] ) \oplus
 {R}[m] \rightarrow 0. \ \ \ \ \ \notag
\end{eqnarray}

We then take $A $-fixed points to obtain the sequence:
\begin{eqnarray}\label{eq6.62}
\ \ \ \ \ \ \ \ \ \ \mathrm{Det}( R[m][H] ^{\times }) ^{A}
\xrightarrow{\ s_{H }^{A }\ } H_{0}( H ,R_{m}[H
] ) ^{A }\oplus ( H^{\rm ab}\otimes R[m]
) ^{A }\oplus  {\Lambda} ( R[m] ) ^{A } \\
\xrightarrow{T_{H }^{A }} ( H^{\rm ab}\otimes R[m%
] ) ^{A }\oplus  {R}[m]
^{A }\rightarrow 0 \ \ \ \ \ \ \ \ \ \notag
\end{eqnarray}
and we now wish to show that this sequence is exact. For this we first
appeal to the fact that the two following sequences are left exact:
\begin{eqnarray*}
\ \ \ \ \ 1  \rightarrow  \ker ( s_{H }) ^{A }  \rightarrow \mathrm{Det}%
( R[m][H ] ^{\times }) ^{A }\rightarrow {\rm Im}( s_{H }) ^{A } \ ,
\ \ \ \ \ \ \ \ \ \ \ \ \ \ \ \ \ \ \ \ \ \ \ \ \\
\ \ \ \ \ 1  \rightarrow  \mbox{Im}( s_{H }) ^{A }  \rightarrow
H_{0}( H ,R[m][H ] ) ^{A }\oplus (
H ^{\rm ab}\otimes R[m] ) ^{A_{m}}\oplus  {\Lambda} ( R[m] ) ^{A }\\
 \xrightarrow{T ^{A }} (
H ^{\rm ab}\otimes R[m] ) ^{A }\oplus  {R} [m] ^{A }.\ \ \ \ \ \
\end{eqnarray*}%
Considering the terms in the lower sequence (before we take $A $-fixed
points) we see that, because $R[m] $ is $A=A_{m}$-free, $%
H_{0}( H ,R[m] [H ] ) $,
$H_{1}( H_{m},R[m] ) $ and (by hypothesis) the $%
\Lambda ( R[m] )$ are all $A$-cohomologically
trivial modules, and so ${\rm Im}( s_{H })$ is also $A$-cohomologically trivial. Hence the map from $H^{1}(A,\ker ( s_{H }))$ to $H^{1}(A,\mathrm{Det}%
( R[m][H ] ^{\times }))$ is an isomorphism;
therefore  (\ref{eq6.62}) is indeed seen to be exact.

We claim that $H_{1}( H ,R[m] )
^{A }=H_{1}( G,R[m] )$. By restriction and
Shapiro's lemma (using the fact that $R[m] $ is $A $-free)
\begin{equation*}
H_{1}( G,R[m] ) \cong H_{1}( P,R[m]
) \cong H_{1}( H ,R[m] ^{A })
\end{equation*}
and
\begin{equation*}
H_{1}( H ,R[m] ) ^{A }=( H ^{\rm ab}\otimes
R[m] ) ^{A }=( H ^{\rm ab}\otimes R[m]
^{A }) =H_{1}( H ,R[m] ^{A }) .
\end{equation*}%
 We then sum over $m$ to obtain
\begin{eqnarray*}
\oplus_{m}H_{1}( H_{m},R[m] ) ^{A_{m}}\cong \oplus
_{m}H_{1}( G,R[m] ) \cong   \ \ \ \ \ \ \ \ \ \ \ \ \ \ \ \ \ \ \\
\ \ \ \ \ \ \ \ \ \ \cong H_{1}( G,\oplus _{m}R[m] )  \cong H_{1}( G,R[C] ) \cong H_{1}( G,R%
[G_{r}] ) .
\end{eqnarray*}
By Theorem \ref{thm54} we know that
\begin{equation*}
\oplus _{m}\mathrm{Det}( R[m] [H_{m}] ^{\times
}) ^{A_{m}}=\oplus _{m}\mathrm{Det}( (R[m] \circ
P)^{\times }) =\mathrm{Det}( R[G] ^{\times })
\end{equation*}%
and by Lemma \ref{le67} we know that
\begin{equation*}
\oplus _{m}H_{0}( H_{m},R[m] [H_{m}] )
^{A_{m}}=\oplus _{m}R[m] [C_{H_{m}}] ^{A_{m}}=R[C_{G}]
\end{equation*}
and finally we must show the equality
\begin{equation*}
\oplus _{m}\Lambda ( R[ H_{m}] ) ^{A_{m}}=\Lambda
( R[ G] ).
\end{equation*}
Let $N$ again denote the field of fractions of $R$ and we put
\begin{equation*}
\ker (d_{m})=\ker (d_{G})\cap \G_{0} ( N [ m ] \circ P ) \ \ \
\text{and\ \ \ }{\rm Im}(d_{m})={\rm Im}(d_{G})\cap G_{0} ( N [ m] \circ P )
\end{equation*}
so that
\begin{equation*}
\ker d_{G}=\oplus _{m}\ker d_{m}\ \ \ \text{and\ \ \ }d_{G}=\oplus _{m}{\rm Im} d_{m},
\end{equation*}%
and, as usual, $\varepsilon _{m}$ denotes the augmentation map of $N [ m%
 ]  [ H_{m} ]$. Then, for each divisor $m$ of $\left\vert
C\right\vert , $ we fix an abelian character $\chi $ of order $m$ and write $%
G_{m}=C\rtimes H_{m}$. Then we see that $\ker d_{m} $ is the group
generated by virtual characters of degree zero of the form $\mathrm{Ind}
_{G_{m}}^{G}\chi \cdot  ( \rho -\rho  ( 1 )  ) $ for characters
$\rho $ of $H_{m},$ and ${\rm Im}d_{m}=\Z\cdot \mathrm{Ind}%
_{G_{m}}^{G}\chi \cdot \epsilon _{m}$ where $\epsilon _{m}$ is the trivial
character of $H_{m}$.

We let $I_{P,m}$ denote the 2-sided $R [ m ] \circ P$-ideal
generated by the augmentation ideal $I ( R [ m ]  [ H_{m}]  )$.
 From the proof of Corollary \ref{cor57} we observe that we have
the exact sequence
\begin{equation*}
1\rightarrow \mathrm{Det} ( 1+I_{P,m} ) \rightarrow \mathrm{Det}%
 ( R [ m ] \circ P ) \rightarrow R [ m ]
^{A_{m}\times }\rightarrow 1
\end{equation*}%
where the right-hand map is evaluation on $\mathrm{Ind}_{G_{m}}^{G}\chi
\cdot \epsilon _{m};$ we therefore see that
\begin{eqnarray}\label{neweq6.9}
\ker d_{G}\cdot\mathrm{Det} ( R [ G ] ^{\times } ) \cong
\oplus _{m}\mathrm{Det} ( R [ m ] \circ P^{\times } )|
_{\ker d_{m}}=\ \ \ \ \ \\
\ \ \ \ \ \ =\oplus _{m}\mathrm{Det} ( 1+I_{P,m} )|_{\ker
d_{m}}\cong \oplus _{m}\mathrm{Det} ( 1+I_{P,m} )\notag.
\end{eqnarray}
Hence we have the commutative diagram with exact rows
\begin{equation*}
\begin{array}{ccccc}
   \ker d_{G}\cdot \mathrm{Det} ( R [ G ] ^{\times
} )  & \hookrightarrow    &\mathrm{Det} ( R [ G ] ^{\times
} ) &  \twoheadrightarrow  & \mathrm{Det}^{\prime } ( R [ G ]
^{\times } )  \\
       \downarrow& &\downarrow r  & &
 \downarrow        \\
    \oplus _{m}\mathrm{Det} ( 1+I ( R [ m ] %
 [ H_{m} ]  )  ) ^{A_{m}}   & \hookrightarrow  & \oplus _{m}%
\mathrm{Det} ( R [ m ]  [ H_{m} ] ^{\times } ) ^{A_{m}}
& \twoheadrightarrow  & \oplus _{m}\mathrm{Det}^{\prime } ( R [ m%
 ]  [ H_{m} ] ^{\times } )^{A_{m}}.
\end{array}
\end{equation*}%
Indeed, the top sequence is exact by definition; the lower sequence is exact
by applying the $A_{m}$-fixed point functor to the exact sequence for the
definition of $\mathrm{Det}^{\prime } ( R [ m ]  [ H_{m}%
 ] ^{\times } ) $ and then appealing to Proposition 5.3 which
shows that $\mathrm{Det} ( 1+I ( R [ m ]  [ H_{m} ]
 )  ) $ is $A_{m}$-cohomologically trivial. The central vertical
map is the isomorphism $r=\oplus r_{m}$ of Theorem 5.1; and the left-hand
vertical arrow is an isomorphism by the above discussion and Theorem 5.1
which shows that $r_{m} ( \mathrm{Det} ( 1+I_{P,m} )  ) =%
\mathrm{Det} ( 1+I ( R [ m ]  [ H_{m} ]  )
 ) ^{A_{m}}$; therefore the right-hand vertical is also an isomorphism.

The second pair of exact sequences that we need are
\begin{equation*}
\begin{array}{ccccc}
 \ker d_{G}\cdot \mathrm{Det} ( R [ G ] ^{\times
} )  & \hookrightarrow  & \mathcal{M} ( R [ G ]  )  &
\twoheadrightarrow  & \mathcal{M}^{\prime } ( R [ G ]  )  \\
 \downarrow  &  & \downarrow  &  & \downarrow    \\
 \oplus _{m}\mathrm{Det} ( 1+I( R[ m] [ H_{m}]) ) ^{A_{m}} & \hookrightarrow  & \oplus _{m}%
\mathcal{M}( R[ m]  [ H_{m} ]  ) ^{A_{m}} &
\twoheadrightarrow  & \oplus _{m}\mathcal{M}^{\prime }( R [ m ]  [
H_{m} ] ) ^{A_{m}}.
\end{array}%
\end{equation*}%
Here again the top row is exact by definition; for the lower exact sequence
we note that $\ker d_{H_{m}}\mathrm{Det} ( R [ m ]  [ H_{m}] ^{\times }) =\mathrm{Det}( 1+I( R[ m ] [ H_{m}])) $, as in (\ref{neweq6.9}) above, and so we have the
exact sequence
\begin{equation*}
1\rightarrow \mathrm{Det}( 1+I( R[ m ] [ H_{m}%
 ] )  ) \rightarrow \mathcal{M} ( R [ m ]  [
H_{m} ] ) \rightarrow \mathcal{M}^{\prime } ( R [ m ]
 [ H_{m} ] ) .
\end{equation*}
We again we take $A_{m}$-fixed points. The central vertical map is then an
isomorphism by the definition of $\mathcal{M}$ and the fact that $r\ $is an
isomorphism on $\mathrm{Det} ( R [ G ] ^{\times } ) $.

Comparing the two pairs of exact sequences and using the hypothesis (2) in
\ref{ss6c} we conclude that
\begin{equation*}
\Lambda  ( R [ G ]  ) =\frac{\mathrm{Det}^{\prime } ( R
 [ G ]  ) }{\mathcal{M}^{\prime } (R[G])}=
\oplus_m\frac{\mathrm{Det}^{\prime }(R[m][H_m]^\times)^{A_m}}
{\mathcal{M}^{\prime } ( R [m] [ H_{m} ] ^{\times } ) ^{A_{m}}} =\oplus _{m}\Lambda
 ( R [ m ]  [ H_{m} ]  ) ^{A_{m}}
\end{equation*}%
as required.
\end{proof}

\begin{comment}
follows in the same way from the decompositions
\begin{equation*}
\oplus _{m}\mathrm{Det}( R[ H_{m}] ^{\times })
^{A_{m}}=\mathrm{Det}( R[ G] ^{\times }) ,\ \ \ \oplus
_{m}{\rm G}_{0}( \mathbb{F}_{p}[ H_{m}] )
^{A_{m}}={\rm G}_{0}( \mathbb{F}_{p}[ G] ) .
\end{equation*}
\end{comment}

\subsubsection{The maps $\delta_\alpha$ and $\iota$}

Compare to page 288 of [O5]. Given an extension of finite groups
\begin{equation}\label{eq6.63}
1\rightarrow K\rightarrow \widetilde{G} \xrightarrow{\alpha}
G\rightarrow 1
\end{equation}
we have the standard homology sequence (see for instance Theorem 8.1 and
(8.2) on page 202 of [HS]) :%
\begin{eqnarray}\label{eq6.64}
H_{2}( \widetilde{G},R[G_{r}] ) \xrightarrow{H_{2}(
\alpha ) } H_{2}( G,R[G_{r}] )
\xrightarrow{\delta _{\alpha }} K^{\rm ab}\otimes _{\Z[G%
] }R[G_{r}] \rightarrow \\
\xrightarrow{\iota } H_{1}(
\widetilde{G},R[G_{r}] )   \xrightarrow{H_{1}( \alpha
) }H_{1}( G,R[G_{r}] ) \rightarrow 0.\notag
\end{eqnarray}
Recall that here $\widetilde{G}$, $G$ and $K$ all act on the module $R[G_{r}] $ via conjugation on $G_{r}$. We let $\overline{\xi }_{\widetilde G}$
(which we recall is denoted $\overline{\nu }_{\widetilde G}$ in [O5]) denote the
composite
\begin{equation*}
\Kr_1(R[\widetilde G])\xrightarrow{\xi_{\widetilde G}}
H_{1}( \widetilde{G},R [ \widetilde{G}_{r} ]  ) \xrightarrow{
\alpha_{\ast }} H_{1} ( \widetilde{G},R [ G_{r} ]
).  \end{equation*}
For $u\in \SK_{1}( R[ G] ) $ we choose a lift $
\widetilde{u}\in \Kr_{1}( R[ \widetilde{G}] ) $ (see
Lemma \ref{leNew}) and by Lemma \ref{le72} we know that\ $\xi _{G}( u) =1;$
hence we have
\begin{equation}\label{eq6.65}
\overline{\xi }_{\widetilde{G}}( \widetilde{u}) =\alpha _{\ast
}\circ \xi _{\widetilde{G}}( \widetilde{u}) =\xi _{G}(
\alpha ( \widetilde{u}) ) =\xi _{G}( u) =1.
\end{equation}

\subsubsection{The map $\overline{ \tau }_\alpha$}\label{sss6e6}

We continue with the above notations.

\begin{definition}\label{def76}
 Let $A_{\alpha }=\ker ( R[\widetilde{G}]  \xrightarrow{\alpha} R[G] ) $. Then $A_{\alpha }$ is the
$R[\widetilde{G}] $-ideal generated by $( 1-z) $ for all $z\in
K$. We now form the $\widetilde{G}$-homology with respect to the exact sequence
\begin{equation*}
0\rightarrow A_{\alpha }\rightarrow R[\widetilde{G}] \rightarrow
R[G] \rightarrow 0
\end{equation*}
to get
\begin{equation}\label{eq6.66}
H_{1}( \widetilde{G},R[G] ) \xrightarrow{ \partial
^{\alpha  }} H_{0}( \widetilde{G}, A_{\alpha })
\rightarrow H_{0}( \widetilde{G}, R[\widetilde{G}] )
\rightarrow H_{0}( \widetilde{G}, R[G] ) \rightarrow 0
\end{equation}
and we define
\begin{eqnarray*}
\overline{H}_{0}(
\widetilde{G},A_{\alpha }) :&=&\frac{H_{0}( \widetilde{G},A_{\alpha }) }{\partial ^{\alpha }(
H_{1}( \widetilde{G},R[G] ) ) }=\\
&=&\ker ( H_{0}( \alpha ) :H_{0}(
\widetilde{G},R[\widetilde G] ) \rightarrow H_{0}(
\widetilde{G},R[G] )=H_{0}( G,R[G] ) ) .
\end{eqnarray*}
\end{definition}

We now define the map $\tau _{\alpha }:A_{\alpha }\rightarrow K^{\rm ab}\otimes _{
\Z[G] }R[G_{r}] $   by the rule that
$$
s( 1-z) g\mapsto sz\otimes  \alpha ( g) _{r}
$$ for $s\in R$,
$g\in G$, $z\in K$. We see that $\tau _{\alpha }$ induces a further
map, also denoted $\tau _{\alpha }$, on the covariants
\begin{equation*}
\tau _{\alpha }:H_{0}( \widetilde{G},A_{\alpha }) \rightarrow
H_{0}( \widetilde{G},H_{1}( K,R[G_{r}] ) )
=K^{\rm ab}\otimes _{\Z[G] }R[G_{r}] ;
\end{equation*}%
here we use the fact that $H_{1}( K,R[G_{r}] )
=K^{\rm ab}\otimes _{ \Z }R[G_{r}] $ and  for $R[G]$-modules $M$ and $N$ we know that
$H_{0}( \widetilde{G},M\otimes N)=H_{0}( G,M\otimes N) =M\otimes _{\Z[G] }N$.

  Using the above  we obtain the exact top row of
the following diagram:
\begin{equation}\label{eq6.67}
{\small
\begin{array}{ccccccc}
H_{1}( \widetilde{G},R[G] ) & \xrightarrow{
\partial ^{\alpha }} & H_{0}( \widetilde{G},A_{\alpha }) &
\rightarrow & H_{0}( \widetilde{G},R[\widetilde{G}] )
& \rightarrow & H_{0}( G,R[G] ) \\
\downarrow \lambda &  & \downarrow \tau _{\alpha } &  & \downarrow \tau _{%
\widetilde{G}} &  &  \\
H_{2}( G,R[G_{r}] ) &  \xrightarrow{ \delta _{\alpha }}
 & K^{\rm ab}\otimes _{\Z[G] }R[G_{r}%
] &  \xrightarrow{ \iota}  & H_{1}( \widetilde{G},R[
G_{r}] ) &  &
\end{array}}
\end{equation}%
where $\lambda  $ is defined as follows: note that $H_{1}( \widetilde{G}%
,R[G] )  $ is generated by elements $g\otimes sh$ for $s\in R$,
 and $g\in \widetilde{G}$, $h\in G$ with the property that $\alpha
( g) $ and $h$ commute; we define $\lambda ( g\otimes
sh) =\alpha ( g) \wedge h\otimes sh_{r}\in H_{2}^{\rm ab}(
G,R[G_{r}] )$. The lower row in the diagram is exact by
(\ref{eq6.64}). Recall that $\overline{H}_{0}(  \widetilde{G},A_{\alpha }) $ was defined
in Definition \ref{def76} with $\overline{H}_{0}(  \widetilde{G},A_{\alpha }) =
{\rm coker} (\partial _{\alpha })$ and so $\tau _{\alpha }$
induces a map
\begin{equation*}
\overline{\tau }_{\alpha }:\overline{H}_{0}( \widetilde{G},A_{\alpha }) =\ker
( H_{0}( \alpha ) :H_{0}( \widetilde{G},R[\widetilde{G}]
) \rightarrow H_{0}( \widetilde{G},R[G]) ) \xrightarrow{ \ \ }
\frac{K^{\rm ab}\otimes _{\Z[G] }R[G_{r}] }{%
\delta _{\alpha }( H^{\rm ab}_{2}( G,R[G_{r}] ) ) }.
\end{equation*}

\subsubsection{The diagram.}\label{sss6e7}

Noting that by the exact sequence (\ref{eq6.64})
\begin{equation*}
\ker ( H_{1}( \widetilde{G},R[G_{r}] ) \overset{%
H_{1}( \alpha ) }{\rightarrow }H_{1}( G,R[G_{r}]
) ) ={\rm Im}(\iota) ,
\end{equation*}%
we see from the above work that we may now assemble  the following diagram:
\begin{equation}\label{eq6.68}
\begin{array}{ccc}
 \ker ( H_{0}( \widetilde{G},R[\widetilde{G}] )
\xrightarrow{H_{0}( \alpha ) } H_{0}( G,R[G%
] ) )  &  &  \\
 \downarrow  \overline{\tau }_{\alpha } &  &  \\
 { \displaystyle\frac{K^{\rm ab}\otimes_{ \Z[G] } R[G_{r}] }{\delta _{\alpha }(
H_{2}^{\rm ab}( G,R[G_{r}] ) +( 1-\Psi )
H_{2}( G,R[G_{r}] ) ) } }& \xleftarrow{  \delta_{\alpha }  }  &  \overline{H}_{2}( G,R[G_{r}] )_{\Psi} \\
\ \ \ \ \ \ \ \ \ \ \ \ \uparrow ( \Psi -1) \circ \iota ^{-1} &  &  \\
\ker ( H_{1}( \widetilde{G},R[G_{r}] ) \overset{%
H_{1}( \alpha ) }{\rightarrow }H_{1}( G,R[G_{r}]
) ) . &  &
\end{array}
\end{equation}
where by definition (as in the Introduction)
\begin{equation*}
 \overline{H}_{2}( G,R[G_{r}] )\overset{\mathrm{%
def}}{=} {\displaystyle\frac{H_{2}( G,R[G_{r}]
) }{H_{2}^{\rm ab}( G,R[G_{r}] ) }}.
\end{equation*}

\subsubsection{The map $\Theta _{R[G]}$}\label{sss6e8}
We continue to assume that $R$ is as in the beginning of \S \ref{ss6c}.

\begin{theorem}
Choose a group extension as in (\ref{eq6.63}) with the property that the
image of the
map $\delta _{\alpha }:H_{2} ( \widetilde{G},R [ G_{r} ]
 ) \rightarrow H_{2} ( G,R [ G_{r}]) $ is contained in
 $H_{2}^{\mathrm{ab}} ( G,R[ G_{r} ]  ) $. Note that
  such extensions exist by Lemma 8.3.iii in [O5].
 Let $u\in \SK_{1}(R[G])$;  by Lemma \ref{leNew} we may choose
a lift of $u$ denoted $\widetilde{u}\in \Kr_{1}(R[ \widetilde{G}])$. Mapping $u$ to the value
\begin{equation}\label{eq6main}
\Theta _{R[G]} ( u ) =\delta _{\alpha }^{-1}\left( \overline{\tau }%
_{\alpha }\circ \upsilon' _{\widetilde{G}}\left( \widetilde{u}\right) + (
\Psi -1 ) \circ \iota ^{-1} ( \overline{\xi }_{\widetilde{G}} (
\widetilde{u} )  ) \right) \in \overline{H}_{2} ( G,R [
G_{r} ]  )
\end{equation}
yields an isomorphism $\Theta _{R[G]}:\SK_{1}( R[ G])
\rightarrow \overline{H}_{2} ( G, R [ G_{r}]))_\Psi$
which is independent of
the choice of the group extension (\ref{eq6.63}).
\end{theorem}

\begin{remark}
{\rm If $G$ is  a $p$-group, then $G_r=\{1\}$ and it is easily seen
that the above map $\Theta_{R[G]}$ coincides with the map defined in 3.c.}
\end{remark}

\begin{proof}
  First we observe that by the condition on the group extension the map
$\delta _{a}$ (in the diagram (\ref{eq6.68})) is a monomorphism: indeed using the
exact sequence
\begin{equation*}
H_{2}( \widetilde{G},R[ G_{r}] ) \xrightarrow{H_{2} (
\alpha ) } H_{2} ( G, R [ G_{r} ]  )
\xrightarrow{\delta _{\alpha }}  K^{\mathrm{ab}}\otimes_{\Z[G]} R [ G_{r} ]
\end{equation*}%
we see that the induced map (also denoted $\delta _{\alpha })$
\begin{equation*}
\delta _{\alpha }:\overline{H}_{2} ( G,R[ G_{r}])
\rightarrow \frac{K^{\mathrm{ab}}\otimes_{\Z[G]} R [ G_{r} ] }{\delta
_{a} ( H_{2}^{\mathrm{ab}} ( G,R [ G_{r} ]  ) ) }
\end{equation*}%
is injective and hence
\begin{equation*}
\delta _{\alpha }:\overline{H}_{2} ( G,R [ G_{r} ]  )
_{\Psi }\rightarrow \frac{K^{\mathrm{ab}}\otimes_{\Z[G]}  R [ G_{r} ] }{%
\delta _{a} ( H_{2}^{\mathrm{ab}} ( G,R [ G_{r} ] )
+ ( 1-\Psi  ) H_{2} ( G,R [ G_{r} ]  ) ) }
\end{equation*}%
is also injective.

Next we want to show that the terms inside the bracket on the right of
(\ref{eq6main}) are all defined and that
\begin{equation}
\overline{\tau }_{\alpha }\circ \upsilon' _{\widetilde{G}} ( \widetilde{u}%
 ) + ( \Psi -1 ) \circ \iota ^{-1} ( \overline{\xi }_{%
\widetilde{G}} ( \widetilde{u} )  ) \in \ker (i)=\mbox{\rm Im}(\delta
_{a}).
\end{equation}%

First note that $\tau _{G}$ and $\overline{\xi }_{\widetilde G}$ admit factorizations
\begin{eqnarray*}
\tau _{\widetilde{G}} &:&H_{0}  (\widetilde{G}, R [ \widetilde{G}_{r}])\xrightarrow {\omega _{\widetilde{G}}}H_{1}(
\widetilde{G}, R[ \widetilde{G}_{r}]) \xrightarrow{\alpha
_{\ast }} H_{1}( \widetilde{G},R[ G_{r}]) =H_{1}( G,R[ G_{r}]) \\
\overline{\xi }_{\widetilde{G}} &:&\mathrm{Det} ( R [ \widetilde{G}] ^{\times } ) \xrightarrow{\xi _{\widetilde{G}}}
H_{1} ( \widetilde{G},R [ \widetilde{G}_{r}]) \xrightarrow{
\alpha _{\ast }} H_{1} ( \widetilde{G},R[G_r]
 ) =H_{1} ( G,R[G_r]) .
\end{eqnarray*}
Next observe that as in (\ref{eq6.65})
\begin{equation*}
\overline{\xi }_{\widetilde{G}} ( \widetilde{u} ) \in \ker (
H_{1} ( \alpha ) :H_{1} ( \widetilde{G},R [ G_{r} ]
 ) \rightarrow H_{1} ( G,R[G_r]) ) =%
\mbox{\rm Im}(\iota) .
\end{equation*}%
and also note that, as $u\in \SK_{1}(R[G])$, we
have $\mathrm{Det}(u)=1$, and so $\upsilon _{G}^{\prime
} ( u ) =0$ and hence $u\in H_{0} ( G,A_{\alpha } )$. By
the diagram (\ref{eq6.67}) we know that $i\circ \tau _{\alpha }=\tau _{G}$ and so
applying $i$ to the term in the statement of the theorem we get
\begin{equation*}
i( \overline{\tau }_{\alpha }\circ \upsilon' _{\widetilde{G}%
} ( \widetilde{u} ) + ( \Psi -1 ) \circ \iota ^{-1} (
\overline{\xi }_{\widetilde{G}} ( \widetilde{u} ) )
 ) =\overline{\tau }_{\widetilde{G}}\circ \upsilon _{\widetilde{G}%
} ( \widetilde{u} ) + ( \Psi -1 ) ( \overline{\xi }_{%
\widetilde{G}}  ( \widetilde{u} )  ).
\end{equation*}
We claim that the latter term vanishes: by Proposition \ref{pro75} we know $\omega _{\widetilde{G}%
}\circ \upsilon _{\widetilde{G}}^{\prime } ( \widetilde{u} )
=\left( 1-\Psi \right) \overline{\xi }_{\widetilde{G}} ( \widetilde{u}%
 )$; then using the factorizations for $\tau _{\widetilde{G}}$
and $\overline{\xi }_{\widetilde{G}}$ above we see
\begin{equation*}
\overline{\tau }_{\widetilde{G}}\circ \upsilon' _{\widetilde{G}}(
\widetilde{u} ) + ( \Psi -1 ) ( \overline{\xi }_{%
\widetilde{G}} ( \widetilde{u}))=0
\end{equation*}%
as required.
 The fact that this value is independent of choices follows exactly as in the argument provided in the proof of Theorem 12.9 of [O5].

To conclude, we know from Theorem \ref{thm6} that $\Theta _{R[G] }$ is an
isomorphism if $G$ is a $p$-group. We now use Theorem \ref{thm58} and Corollary \ref{cor57}
to show that
$\Theta _{R[G] }$ is an isomorphism if $G$ is $\Q_{p}$-$p$-elementary. The functoriality of   $\Theta _{R[G] }$ together
with the induction Theorem \ref{thm53} will then show that $\Theta _{R[G]
}$ is an isomorphism, in all cases, which agrees with our construction in \S \ref{ss6a}.

Suppose now that $G$ is $\Q_{p}$-$p$-elementary  and we yet again
adopt the notation of Sect. \ref{s5}. By Theorem \ref{thm58} and Corollary \ref{cor57} we then have isomorphisms%
\begin{equation*}
\SK_{1}( R[G] ) \cong \oplus _{m}\SK_{1}( R[m
] \circ P, I_P) \cong \oplus _{m}H_{0}( A_{m},\SK_{1}( R
[m] [H_{m}]) )
\end{equation*}%
and
\begin{equation*}
\overline{H}_{2}( G,R[G_{r}] ) _{\Psi }=\overline{H}%
_{2}( G,R[C] ) _{\Psi }\cong \oplus _{m}\overline{H}%
_{2}( G,R[m] ) _{\Psi }\cong \oplus _{m}H_{0}(
A_{m},\overline{H}_{2}( H_{m},R[m] ) ) _{\Psi }.
\end{equation*}%
The latter isomorphism comes from the composition of functors
spectral sequence for covariants of $G$ as covariants of $H_{m}$
 followed by covariants of  $A_{m}$  and the fact that $R[m] $ is  $A_{m}$-free
 with trivial $H_m$-action.
 The result then
follows from the fact that by Theorem \ref{thm43}, $\Theta_{R[m][H_m]}$ yields a functorial
isomorphism from $\SK_{1}( R[m][H_{m}]) $ to $\overline{H}_{2}( H_{m},R[m] )_{\Psi }$ for each $m$.
\end{proof}

 \bigskip
\bigskip

% ----------------------------------------------------------------------
% ----------------------------------------------------------------------
\section{Appendix A: Adams operations}\label{s7}
\setcounter{equation}{0}

Throughout this Appendix we assume that $R$ satisfies the conditions imposed in \S \ref{ss6c}. The proof of Theorem \ref{thm62} follows the proof of Theorem 1 in [CNT]
and the proof of Theorem 1.2 in [T]. There is one crucial difference between
these proofs and the proof that we now give for the much more general rings $R$;
this occurs in the special case when $G$ is a $p$-group. As in [T]
our proof proceeds in five steps. Four steps proceed essentially in the same
manner as in [CNT] and [T], and so in these cases we often refer the reader
to [T] for details; the fourth step is the case where $G$ is a $p$-group: this is considerably more involved and it  is dealt   with in full detail.

For a finite non-ramified extension $L$ of $\Q_{p}$ we again set $%
R_{L}=R\otimes _{\Z_{p}}\O_{L}$. Note that if $R$ satisfies the standing hypotheses
and the additional above hypothesis, then so does $R_{L}$.

For an integer $h$ we define
\begin{equation}\label{eq7.73}
M_{h}( R[G] ) =\frac{\psi ^{h}( \mathrm{Det}%
( R[G] ^{\times }) ) \mathrm{Det}( R[
G] ^{\times }) }{\mathrm{Det}( R[G] ^{\times
}) }.
\end{equation}%
In order to prove Theorem \ref{thm62} it will suffice to show that $M_{h}( R%
[G] ) =\{1\}$. Our proof proceeds in five steps:\medskip

\noindent\textbf{Step 1.}

\begin{lemma}\label{le79}
If $G$ has order prime to $p,$ then $M_{h}( R[G] )
=\{1\}$ for all $h$.
\end{lemma}

\begin{proof} Since $G$ has order prime to $p$  we have isomorphisms%
\begin{eqnarray*}
\Z_{p}[G] &\cong &\prod\nolimits_{i}M_{n_{i}}(
\O_{i}) \\
R[G] &\cong &\prod\nolimits_{i}M_{n_{i}}( R\otimes_{\Z_p}
\O_{i})
\end{eqnarray*}
for some non-ramified rings of  $p$-adic integers $\O_{i}$.  Let $\O^{c}$
denote the valuation ring of the chosen algebraic closure $\Q_{p}^{c} $ of $\Q_{p}$ and set $\Omega _{p}=\mathrm{Gal}(
\Q_{p}^{c}/\Q_{p})$. Then, as in Proposition 22 on
page 23 of [F], we have the isomorphism
\begin{equation*}
\mathrm{Det}( R[G] ^{\times }) =\mathrm{Hom}_{\Omega
_{p}}( \Kr_{0}( \Q_{p}^{c}[G] ) ,(
R\otimes_{\Z_p} \O^{c}) ^{\times })
\end{equation*}%
and the right-hand side is clearly stable under $\psi ^{h}$ for any integer $%
h,$ since the actions of $\psi ^{h}$ and $\Omega _{p}$ commute.
\end{proof}
\medskip

\noindent\textbf{Step 2.}

\begin{proposition}\label{pro80}
If $G$ is $\Q_{p}$-$l$-elementary with $l\neq p$, then $M_{h}( R[G] ) $ is killed by a power of $p$.
\end{proposition}

\begin{proof} See the proof of Proposition 2.5 on page 105 of [T].
\end{proof}
 \medskip

\noindent\textbf{Step 3.} Here we suppose that $G$ is a $p$-group. Recall that we have the natural
decomposition
\begin{equation*}
\mathrm{Det}( R[G] ^{\times }) =\mathrm{Det}(
1+I( R[G] ) ) \times \mathrm{Det}(
R^{\times }) .
\end{equation*}
The factor $\mathrm{Det}( R^{\times }) $ is clearly stable under
Adams operations, and so it will suffice to show that $\mathrm{Det}(
1+I( R[G] ) ) $ is stable under Adams
operations. For future use we note the isomorphism
\begin{eqnarray}\label{eq7.74}
M_{h}( R[G] ) &=&\frac{\psi ^{h}( \mathrm{Det}
( R[G] ^{\times }) ) \mathrm{Det}( R [
G] ^{\times }) }{\mathrm{Det}( R[G] ^{\times
}) } \\
&\cong &\frac{\psi ^{h}( \mathrm{Det}( 1+I( R[G]
) ) ) \mathrm{Det}( 1+I( R[G]
) ) }{\mathrm{Det}( 1+I( R[G] )
) }.\notag
\end{eqnarray}%
We shall also use the exact sequence%
\begin{equation*}
0\rightarrow p\phi ( \mathcal{A}( R[G] ) )
\rightarrow \mathrm{Det}( 1+I( R[G] ) )
\rightarrow 1+I( R[G^{\rm ab}] ) \rightarrow 1
\end{equation*}
(see (3.8) in [CNT1]). Given $x\in 1+I( R[G] )$,
because $g\mapsto g^{h}$ induces an endomorphism of $R[G^{\rm ab}]$,
 because $\mathcal{A}( R[G] ) \subset I( R%
[G] )$, and because $I( R[G] ) $
is contained in the Jacobson radical of $R[G] ,$ we see that if
we set $x=\sum\nolimits_{g}x_{g}g$ and $y=\sum\nolimits_{g}x_{g}g^{h}\in
1+I( R[G] )$ and if we put
\begin{equation*}
\gamma \overset{\rm defn}{=}\psi ^{h}(\mathrm{Det}( x))\cdot
\mathrm{Det}( y)^{-1},
\end{equation*}
then $\gamma$ is trivial on all abelian characters of $G$.

\begin{lemma}\label{le81}
We have
$\nu _{G}( \gamma ) \in p\phi ( \mathcal{A}( R[G%
] ) ) =\nu _{G}( \mathrm{Det}( 1+\mathcal{A}%
( R[G] ) )   )$.
\end{lemma}

\begin{proof} Since for an abelian character $\chi $ of $G$ we have%
\begin{equation*}
\chi ( \nu _{G}( \gamma ) ) =\log ( \gamma (
p\chi -\psi ^{p}\chi ) ) =0,
\end{equation*}%
it follows that $\nu _{G}( \gamma ) \in p\phi ( \mathcal{A}%
( R[G] ) ) \otimes \Q_{p}$. Let $\nu
_{G}( \mathrm{Det}( x) ) =\sum\nolimits_{c\in
C_{G}}\lambda _{c}c$ with $\lambda _{c}\in pR$ by Theorem \ref{thm22}.
We now show that \begin{equation*}
\nu _{G}( \psi ^{h}\mathrm{Det}( x) )
=\sum\nolimits_{c\in C_{G}}\lambda _{c}c^{h}.
\end{equation*}
 This follows from the fact that, for each character $\chi $ of $G$, we
have the two equalities
\begin{eqnarray*}
\chi ( \nu _{G}( \psi ^{h}\mathrm{Det}( x) )
) &=&\log ( ( \psi ^{h}\mathrm{Det}( x) )
( p\chi -\psi ^{p}\chi ) ) \\
&=&\log ( \mathrm{Det}( x) ( p\psi ^{h}\chi -\psi
^{p}\psi ^{h}\chi ) ) =\log ( \mathrm{Det}( x)
( p\psi ^{h}\chi -\psi ^{ph}\chi ) )
\end{eqnarray*}%
and%
\begin{eqnarray*}
\chi ( \sum\nolimits_{c\in C_{G}}\lambda _{c}c^{h}) &=&\psi
^{h}\chi ( \sum\nolimits_{c\in C_{G}}\lambda _{c}c)
\\
&=&\psi
^{h}\chi ( \nu _{G}( \mathrm{Det}( x) ) )
=\log ( \mathrm{Det}( x) ( p\psi ^{h}\chi -\psi
^{ph}\chi ) ) .\ \ \ \
\end{eqnarray*}
This shows that $\nu _{G}( \gamma ) \in pR[C_{G}] $
and we have seen that $p\phi ( \mathcal{A}( R[G]
) ) \otimes \Q_{p}$ and so $\nu _{G}( \gamma
) \in p\phi ( \mathcal{A}( R[G] ) )$. To conclude we note that by Theorem \ref{thm23} we know that $p\phi (
\mathcal{A}( R[G] ) ) =\nu _{G}( \mathrm{Det}( 1+\mathcal{A}( R[G] ) ) )$.
\end{proof}

\begin{lemma}\label{le82}
The group
\begin{equation*}
\ker ( \nu_G :\psi ^{h}\mathrm{Det}( 1+I( R[G]
) ) \cdot\mathrm{Det}( 1+I( R[G] )
) \rightarrow N[C_{G}] )
\end{equation*}
is $p$-power torsion.
\end{lemma}

\begin{proof} In (\ref{eq3.11})   we have seen that
\begin{equation*}
\ker ( \nu_G:\mathrm{Det}( 1+I( R[G] )
) \rightarrow pR[C_{G}] ) =\mathrm{Det}(
G)
\end{equation*}%
which is $p$-power torsion. The result then follows from part (a) of
Proposition \ref{pro78} in Step 5 (which only uses Step 2)
that $M_{h}( R[G] )$ is $p$-power torsion.
\end{proof}
\smallskip

With the above notation we again consider $\gamma =\psi ^{h}(\mathrm{Det}%
( x)) \cdot \mathrm{Det}( y) ^{-1}$. By Lemma \ref{le81} we know
that $\nu_G ( \gamma ) \in p\phi ( \mathcal{A}) =\mathrm{
Det}( 1+\mathcal{A}) $ and so by Lemma \ref{le82} we know that we can
write
\begin{equation*}
\gamma =\mathrm{Det}( z)\cdot t
\end{equation*}
with $z\in 1+{\mathcal A}$, and
\begin{equation*}
t\in \ker ( \nu_G :\psi ^{h}(\mathrm{Det}( 1+I( R[G]
) )\cdot\mathrm{Det}( 1+I( R[G] )
) \rightarrow N[C_{G}] )
\end{equation*}
so that by Lemma \ref{le82} $t$ is $p$-power torsion and is trivial on abelian
characters of $G$. Therefore we may write
\begin{equation*}
\psi ^{h}(\mathrm{Det}( x)) =\mathrm{Det}( zy)\cdot t.
\end{equation*}%
Thus, to prove that $M_{h}( R[G] ) =\{1\}$, it will
suffice to show that $t=1$.

First note that if $\chi $ is a character of $G$, whose degree is denoted $%
\chi ( 1)$, then, as $G$ is a $p$-group, $\chi -\chi (
1) \in \ker d_{p}$ and so as in the methods used for Step 2 (see page
106 in [T]), we know that
\begin{equation*}
\psi ^{h}\mathrm{Det}( x) ( \chi -\chi ( 1)
) =\mathrm{Det}( x) ( \psi ^{h}\chi -\psi ^{h}\chi
( 1) ) \equiv 1\mod P^{c}
\end{equation*}
where $P^{c}$ denotes the $R\otimes _{\Z_{p}}\O^{c}$-ideal generated
by the maximal ideal of $\O^{c}$; therefore, since $\mathrm{Det}(
zy) ( \chi -\chi ( 1) ) \equiv 1\mod P^{c}$,
we deduce that
\begin{equation*}
t( \chi ) =t( \chi -\chi ( 1) ) \equiv 1 \mod P^{c}.
\end{equation*}

Let $T_{R} $ denote the subgroup of $p$-torsion elements $\mathrm{Hom}_{\Omega_p}( \Kr_{0}( \Q_{p}^{c}[G] ) ,(
R\otimes \O^{c}) ^{\times }) $ which are trivial on the abelian
characters of $G$. We can now use the above work to define a homomorphism $
\xi :M_{h}( R[G] ) \rightarrow T_{R}$ by the rule that
\begin{equation*}
\xi ( \psi ^{h}\mathrm{Det}( x) {\rm Det}(R[G])^{\times }) =t.
\end{equation*}%
Note that by the exact sequence (\ref{eq3.11}) we know that
\begin{equation*}
\mathrm{Det}( 1+I( R[G] ) ) \cap T_{R}=%
\mathrm{Det}( G) \cap T_{R}=\{1\}.
\end{equation*}
This map is well-defined and injective. For instance to see that $\xi$ is
injective, with the obvious notation suppose that
\begin{equation*}
\xi ( \psi ^{h}\mathrm{Det}( x) ) =t=\xi ( \psi
^{h}\mathrm{Det}( x^{\prime }) )
\end{equation*}%
then
\begin{equation*}
\psi ^{h}\mathrm{Det}( x) \mathrm{Det}( y)
^{-1}=t=\psi ^{h}\mathrm{Det}( x^{\prime }) \mathrm{Det}(
y^{\prime }) ^{-1}
\end{equation*}%
and so $\psi ^{h}\mathrm{Det}( xx^{\prime -1}) =\mathrm{Det}%
( yy^{\prime -1}) \in \mathrm{Det}( R[G]
^{\times })$.

\begin{lemma}\label{le83}
Let $L$ denote a finite non-ramified extension of $\Q_{p}$ in
$\Q_{p}^{c}$ and put $R_{L}=R\otimes_{\Z_{p}}\O_{L}$. From
Lemma 6.1 in [CNT1] we know that there is a decomposition of $R$-algebras
\begin{equation*}
R_{L}=\prod\nolimits_{i=1}^{n( L) }R_{L,i}
\end{equation*}%
where the $R_{L,i}$ are integral domains which satisfy the Standing
Hypotheses. Then the numbers $n( L) $ are bounded as $L$ ranges
over all finite non-ramified extensions of $\Q_{p}$ in $\Q_{p}^{c}$.
\end{lemma}

\begin{proof} Let $\O_{p}^{\rm nr}$ denote the valuation ring of the maximal
non-ramified extension of $\Q_{p}$ in $\Q_{p}^{c}$. By
lifting idempotents (see for instance Theorem 6.7 on page 123 of [CR1]) we
know that there is  natural bijection between the idempotents of $\overline{R}\otimes \mathbb{F}_{p}^{c} $
and $R\otimes_{\Z_{p}}\O_{p}^{\rm nr}$;
therefore, by the additional hypotheses introduced in 6.c, $R\otimes _{\Z_{p}}\O_{p}^{\rm nr} $
contains only a finite number of orthogonal idempotents.
The result then follows.
\end{proof}

\begin{lemma}\label{le84}
Let $\zeta _{p^{n}}$ denote a primitive $p^{n}$-th root of unity in $\Q_{p}^{c}$ for $n>0$.
Suppose $E\subset \Q_{p}( \zeta
_{p^{n}})$. Then

(a) $R\otimes_{\Z_{p}}\O_{E}$ is an integral domain;

(b) $\mu _{p^{\infty }}( R\otimes_{\Z_{p}}\O_{E})
=\mu_{p^{\infty }}( \O_{E})$.
\end{lemma}

\begin{proof} (a) Recall that we write $N$ for the field of fractions of the normal domain $R$. Since
$R$ and $\O_{E}$ are flat over $\Z_{p}$ we know that $R\otimes_{\Z_{p}}\O_{E}\subset N\otimes_{\Q_{p}}E$ and so it will
suffice  to show that $N$ and $E$ are linearly disjoint over $\Q_{p}$. Let $\pi _{E}$ denote a uniformising parameter for $E$ and suppose
for contradiction that
\begin{equation*}
[NE:N] =m<[E:\Q_{p}] .
\end{equation*}%
Then $\mathrm{Norm}_{NE/N}( \pi _{E}) $ belongs to $R$ and has $p$-valuation $[
NE:N] [E:\Q_{p}] ^{-1}<1$ which contradicts the
fact that $p$ is prime in $R$.

To prove (b) suppose for contradiction that $\zeta _{p^{r}}\in \mu
_{p^{\infty }}( R\otimes _{\Z_{p}}\O_{E}) $ but $\zeta
_{p^{r}}\notin E.$ We put $M=E( \zeta _{p^{r}}) $ which strictly
contains $E$ and is totally ramified over $\Q_{p}$ and we let $\pi
_{M}$ denote a uniformising parameter for $M$. Then using (a) we see that $%
\mathrm{Norm}_{NE/N}( \pi _{M}) $ has $p$-valuation
\begin{equation*}
[NE:N] [M:\Q_{p}] ^{-1}=[E:\Q_{p}] [M: \Q_{p}] ^{-1}<1
\end{equation*}
which again contradicts the fact that $p$ is prime in $R$.\end{proof}
\smallskip

We now conclude the proof of Step 3. Suppose that the torsion group $T_{R}$
has exponent $p^{e},$ let $L^{\prime }/\Q_{p}$ denote a non-ramified
extension with $n( L^{\prime }) $ maximal as in Lemma \ref{le83}. We
write $\mathcal{N}_{L^{\prime }/\Q_{p}}$ for the co-restriction (or
norm map) of $L^{\prime }/\Q_{p}$. Then we can find a non-ramified
extension $L$ of $L^{\prime }$ of degree $p^{e}$. By Lemma \ref{le84} we know that
$T_{R_{L}}=T_{R}$ and so
\begin{equation*}
\mathcal{N}_{L/\Q_{p}}( T_{R_{L}}) =\mathcal{N}
_{L^{\prime }/\Q_{p}}\circ \mathcal{N}_{L/L^{\prime }}(
T_{R_{L}}) =\mathcal{N}_{L^{\prime }/\Q_{p}}(
T_{R_{L}}^{p^{e}}) =1.
\end{equation*}

We then have a diagram%
\begin{equation*}
\begin{array}{ccccc}
1 & \rightarrow & M_{h}( R_{L}[G] ) & \overset{\xi }{%
\rightarrow } & T_{R_{L}} \\
&  & \ \ \ \ \downarrow \mathcal{N}_{L/\Q_{p}} &  & \downarrow 0 \\
1 & \rightarrow & M_{h}( R[G] ) & \overset{\xi }{%
\rightarrow } & T_{R}.%
\end{array}%
\end{equation*}%
But by Theorem \ref{thm63} we know that $\mathcal{N}_{L/\mathbf{Q}_{p}}(
1+I( R_{L}[G] ) ) =\mathcal{N}_{L/\Q_{p}}( 1+I( R[G] ) )$; hence we have $\mathcal{N}_{L/\Q_{p}}( M_{h}( R_{L}[G] ) )
=M_{h}( R[G] )$, and we have therefore shown that for a $p$-group $G$ we have $M_{h}( R[G] ) =\{1\}$. \endproof
\medskip

 \noindent\textbf{Step 4.}

\begin{proposition}
For each $\Q_{p}$-$p$-elementary group $G$, and for all integers $h $,
\begin{equation*}
M_{h}( R[G] ) =\{1\}.
\end{equation*}
\end{proposition}

\begin{proof} This is essentially the same as the proof of the corresponding statement  given in the
first part of Sect. 3 in Ch. 9 of [T].\end{proof}

\medskip

 \noindent\textbf{Step 5.}

 \begin{proposition} \label{pro78}
(a)\ Given an integer $h$ and a prime number $l\neq p$, then, given a finite
group $G,$ we can find an integer $m_{l},$ which is not divisible by $l,$ so
that $m_{l}M_{h}( R[ G] ) =\{1\}$ for all $h, $ and
so $M_{h}( R[ G] )$ is killed by a power of $p.$

(b)\ Given an integer $h,$ then $M_{h}( R[ G] ) $ is
killed by an integer which is coprime to $p.$

(c)\ $M_{h}( R[ G] ) =\{1\}$ for an arbitrary finite
group $G$.
\end{proposition}

\begin{proof} The arguments are similar to the proof of the corresponding statement
in [T]. For example for (a) we can start the argument by observing that by
Theorem 28 in [S] we can write%
\begin{equation*}
m_{l}=\sum\nolimits_{H\in \mathcal{E}_{l}( \mathbb{Q}_{p}) }n_{H}%
\mathrm{Ind}_{H}^{G}\theta _{H}
\end{equation*}%
and then use Step 2.

Part (b) follows similarly using Step 4, and then (c) follows immediately from
(a) and (b).
\end{proof}

As above, we now see that this implies  Theorem \ref{thm62}.

% ----------------------------------------------------------------------
% ----------------------------------------------------------------------
\section{Appendix B} \label{s9}

% ----------------------------------------------------------------------
\subsection{Proof of Corollary \ref{cor8}}\label{sss6d9}

The notation here is as in the Introduction.
For the convenience of the reader we recall some of the set-up.
We let $\{C_{i}\}_{i\in
I}$ denote the set of $G$-conjugacy classes in $G_{r}$, let $g_{i}$ be a chosen
group element whose conjugacy class lies in $C_{i}$ and let $G_{i}$ denote
the centralizer of $g_{i}$ in $G$; then we have a disjoint union
decomposition $G_{r}=\sqcup _{i}g_{i}^{G/G_{i}}$. Next consider the action
of $\Psi $ on the $\{C_{i}\}_{i\in I}$. We may view this action as an action
on $I$ and we let $J$ denote  the set of orbits of the action of $\Psi
$ on $I$;  for $j\in J$ we let $n_{j}$ denote the cardinality of $j$. We
then obtain a further disjoint union decomposition
\begin{equation*}
G_{r}=\sqcup_{j}\sqcup _{m=1}^{n_{j}}( g_{i_{j}}^{G/G_{i_{j}}})
^{\Psi ^{m}}
\end{equation*}%
where $i_{j}\,$ denotes a chosen element of the orbit $j$. This induces the
decomposition of homology groups%
\begin{eqnarray*}
H_{2}( G,R[G_{r}] ) &=&H_{2}( G,\Z [G_{r}] ) \otimes R \\
&=&\oplus _{i}H_{2}( G_{i},\Z g_{i}) \otimes  R
\end{eqnarray*}%
and hence a decomposition
\begin{eqnarray*}
 H_{2}( G,R[G_{r}] )   _{\Psi } &=&(
\oplus _{j}\oplus _{m=1}^{n_{j}}H_{2}( G_{i_{j}},\Z
g_{i_{j}}^{\Psi ^{m}}) \otimes R) _{\Psi } \\
&=& ( \oplus _{j}( H_{2}( G_{i_{j}},\Z) \otimes
\frac{\Z[\Psi ] }{( \Psi ^{n_{j}}-1) })
\otimes R ) _{\Psi } \\
&=&\oplus _{j}[H_{2}( G_{i_{j}},\Z) \otimes R_{\Psi
^{n_{j}}}] _{\Psi } \\
&=&\oplus _{j}\left[H_{2}( G_{i_{j}},\Z) \otimes \frac{R
}{( F-1) R}\right] .
\end{eqnarray*}
We now show that we have a similar decomposition
\begin{equation}\label{eq6.72}
( H_{2}^{\rm ab}( G,R[G_{r}] ) ) _{\Psi
}=\bigoplus _{j}\left[H_{2}^{\rm ab}( G_{i_{j}},\Z) \otimes
\frac{R}{( F-1) R}\right]
\end{equation}
and hence the decomposition
\begin{equation*}
( \overline{H}_{2}( G,R[G_{r}] ) ) _{\Psi
}=\bigoplus _{j}\left[\overline{H}_{2}( G_{i_{j}},\Z)
\otimes \frac{R}{( F-1) R}\right].
\end{equation*}%
To this end we observe that for an abelian subgroup $A$ of $G,$ if as
previously, we write $\{A_{r}\}=\{a_{1,...,}a_{k}\},$ then, because$\ A$
acts trivially on $A_{r},$ we can write%
\begin{equation*}
H_{2}( A,R[A_{r}] ) =\oplus _{l=1}^{k}H_{2}(
A,Ra_{l}) .
\end{equation*}%

If $a_{l}$ is a conjugate of $ g_{i_{j ( l )} }^{p^{m}}$, with say $%
a_{l}^{h}=g_{i_{j ( l )} }^{p^{m}}$ for $h\in G$; then, since $%
A^{h} $ centralizes $g_{i_{j ( l )} }^{p^{m}},$ we know that $A^{h}$
centralizes $g_{i_{j ( l ) }}$, and so we have  $A^{h}\subset
G_{i_{j ( l ) }} $ and hence
\begin{eqnarray*}
\mathrm{Cor}_{A}^{G}( H_{2}( A,Ra_{l}) )  \ = \ \mathrm{Cor}%
_{A^{h}}^{G}( H_{2}( A^{h},Ra_{l}^{h}) ) &\subset & \ \ \
\ \ \ \ \ \ \ \  \ \\
\subset
H_{2}( G,\sum\nolimits_{\gamma \in G_{i( j) }\backslash
G}\sum\nolimits_{m=1}^{n_{j}}Ra_{l}^{h\gamma \Psi ^{m}}) &=&H_{2}(
G_{i_{j}}, \sum\nolimits_{m=1}^{n_{j}}Rg_{i_{j( l) }}^{\Psi
^{m}}) .
\end{eqnarray*}
Conversely we know that $H_{2}^{\mathrm{ab}} ( G_{i_{j}}, R ) $ is
generated by the images under corestriction of   the $H_{2} ( B,R )
$ $\ $for maximal abelian subgroups $B $ of $G_{i_{j}}$. Since $g_{i_{j}}$
is centralized by such a maximal abelian subgroup of $G_{i_{j}}$, we see
that $B$ must contain $g_{i_{j}}$; therefore we have the inclusion
\[
\mathrm{Cor}_{B}^{G_{i_{j}}}(H_{2} ( B,R )) =\mathrm{Cor}%
_{B}^{G_{i_{j}}}(H_{2} ( B,Rg_{i_{j}} )) \subset H_{2}^{\mathrm{ab}%
} ( G_{i_{j} ( l ) },\ \sum\nolimits_{m=1}^{n_{j}}Rg_{i (
j ) }^{\Psi ^{m}} )
\]
which establishes (\ref{eq6.72}), as required.\endproof

% ----------------------------------------------------------------------
\subsection{Proof of Corollaries \ref{le11}, \ref{le12}}\label{sss6d9}

Recall that $F( t) =t^{p}$.  We start with Cor. \ref{le11}: Note that  $W\{\{t\}\}=W\langle\langle
t^{-1} \rangle  \rangle +$ $tW[[t] ] $.
Since $( 1-F) tW[[t] ] =tW[[t%
] ] $, we see that
\begin{equation*}
\frac{W\langle \langle t^{-1}\rangle \rangle }{(
1-F) W\left\langle \left\langle t^{-1}\right\rangle \right\rangle }%
\cong \frac{W\{\{t\}\}}{( 1-F) W\{\{t\}\}}
\end{equation*}%
and the result follows from Corollary \ref{cor9}.
\endproof
\smallskip

Now we consider the proof of Cor. \ref{le12}:

We have
\begin{equation*}
W\{\{t\}\}=t^{-1}W\left\langle \left\langle t^{-1}\right\rangle
\right\rangle \oplus W[[ t] ]
\end{equation*}
and we get
\begin{equation*}
\frac{W\{\{t\}\}}{( 1-F) W\{t\}\}}=\frac{t^{-1}W\left\langle
\left\langle t^{-1}\right\rangle \right\rangle }{( 1-F)
t^{-1}W\left\langle \left\langle t^{-1}\right\rangle \right\rangle }\oplus
\frac{W}{( 1-F) W}.
\end{equation*}%
This certainly shows that the map
\begin{equation*}
\frac{W[ [ t]] }{( 1-F) W[[ t] ] }\rightarrow \frac{W\{\{t\}\}}{( 1-F) W\{\{t\}\}}
\end{equation*}
is injective.

We write $W_{m}$ for $W/p^{m}W$. We have
\begin{equation*}
W_{m}( ( t) ) =t^{-1}W_{m}[ t^{-1}] \oplus
W_{m}[[ t] ] ,
\end{equation*}%
and we note that
\begin{equation*}
t^{-1}W_{m}[ t^{-1}] =\oplus _{p\nmid k>0}W_{m}t^{-k}\oplus
_{k>0}(1-F)W_{m} t^{-k} =\oplus _{p\nmid
k>0}W_{m}t^{-k}\oplus ( 1-F) t^{-1}W_{m}[ t^{-1}] ;
\end{equation*}
therefore, for each $m>0,$ we have shown
\begin{equation*}
\frac{W_{m}( ( t) ) /( 1-F) W_{m}(( t) ) }{W_{m}[[ t]] /(
1-F) W_{m}[[ t]] }=\oplus _{p\nmid
k>0}W_{m}t^{-k}.%\oplus \Z_{p}/p^{m}\Z_p.
\end{equation*}

Using the Mittag-Leffler condition twice we get
\begin{eqnarray*}
\frac{W\{\{t\}\}/( 1-F) W\{\{t\}\}}{W[ [ t] ] /( 1-F^{n}) W[ [ t] ] } &=&\frac{%
\varprojlim_{m }W_{m}( ( t) ) /( 1-F)
W_{m}( ( t) ) }{\varprojlim_{m}W_{m}[ [ t]] /( 1-F) W_{m}[[ t]] } \\
&=&\varprojlim_{m}\frac{W_{m}(( t) ) /(
1-F) W_{m}( ( t)) }{W_{m}[ [ t] ] /( 1-F) W_{m}[ [ t] ] } \\
&=&\varprojlim_{m}\oplus _{p\nmid k>0}W_{m}t^{-k}%\oplus \Z_{p}/p^m\Z_p
%\\
%&=&\oplus _{p\nmid k>0}\varprojlim_{m}W_{m}t^{-k}
\end{eqnarray*}%
which is torsion free.
\endproof

\end{document}